\documentclass[oneside, a4paper]{amsart}
\usepackage[a4paper, width=15cm, height=24cm]{geometry}
\usepackage[onehalfspacing]{setspace}
\usepackage{subfiles}
\usepackage[]{currvita}
\usepackage[latin1]{inputenc}
\usepackage[utf8]{inputenx}
\usepackage{amsfonts}
\usepackage{amsmath}
\usepackage{amssymb}
\usepackage{amsthm}
\usepackage{hyperref}
\usepackage[T1]{fontenc}
\usepackage{mathtools}
\usepackage{pstricks-add}
\usepackage{lmodern,dsfont}
\usepackage[english]{babel}
\usepackage{overpic}
\usepackage{fixltx2e,mparhack,url,varioref,}
\usepackage{esvect}
\usepackage{enumitem} 
\usepackage{mdframed} 
\newmdenv[
  topline=false,
  bottomline=false,
  skipabove=\topsep,
  skipbelow=\topsep
]{siderules}
\mdfsetup{linewidth=0.6pt}
\newcommand{\inner}[1]{\langle #1 \rangle}
\newcommand{\spann}[1]{\text{span} \{ #1 \}}
\newcommand{\f}{\mathfrak{f}}

\newcommand{\ttt}{\mathfrak{t}}
\newcommand{\g}{\mathfrak{g}}
\newcommand{\s}{\mathfrak{s}}
\newcommand{\bb}{\mathfrak{b}}
\newcommand{\rr}{\mathfrak{r}}
\newcommand{\cc}{\mathfrak{c}}
\newcommand{\aaa}{\mathfrak{a}}
\newcommand{\ddd}{\mathfrak{d}}
\newcommand{\q}{\mathfrak{q}}
\newcommand{\p}{\mathfrak{p}}
\newcommand{\m}{\mathfrak{m}}
\newcommand{\Light}{\mathbb{P}(\mathcal{L})}
\newcommand{\Q}{\mathcal{Q}}

\newcommand{\D}{\mathcal{D}}

\newcommand{\V}{\mathcal{V}}
\newcommand{\E}{\mathcal{E}}
\newcommand{\PP}{\mathcal{P}}
\newcommand{\bigskipp}{\bigskip \noindent}
\newtheoremstyle{dotless}{6pt}{18pt}{}{}{\bfseries}{.}{\newline}{}
\theoremstyle{dotless}
\newtheorem{thm}{Theorem}
\newtheorem{defi}[thm]{Definition}
\newtheorem{lem}[thm]{Lemma}
\newtheorem{lemdef}[thm]{Lemma and Definition}

\newtheorem{rem}[thm]{Remark}
\newtheorem{prop}[thm]{Proposition}
\newtheorem{cor}[thm]{Corollary}
\newtheorem{fact}[thm]{Fact}
\newcommand{\changefont}[3]{\fontfamily{#1} \fontseries{#2} \fontshape{#3} \selectfont}
\changefont{ptm}{m}{n}
\title{Isothermic nets with spherical parameter lines from discrete holomorphic maps}
\author{Tim Hoffmann and Gudrun Szewieczek}
\begin{document}
\bibliographystyle{plain}
\maketitle
\color{black}
\begin{center}
\begin{minipage}{13cm}\small
\textbf{Abstract.} We prove that all discrete isothermic nets with a family of planar or spherical lines of curvature can be obtained from special discrete holomorphic maps via lifted-folding. This novel approach is a generalization and discretization of a classical method to create planar curvature lines on smooth surfaces. In particular, this technique provides an efficient way to construct discrete isothermic topological tori composed of fundamental pieces from discrete periodic holomorphic maps. 
\end{minipage}
\end{center}
%
%
\section{Introduction}
\noindent The main goal of this work is to determine all  quadrilateral surfaces $f: \mathbb{Z}^2 \supset I \times J \to \mathbb{R}^3$ with circular faces satisfying the following two properties:
\begin{itemize}
\item[(1)] (at least) one family of parameter lines lie on planes or spheres and
\item[(2)] the cross-ratios of all faces are constant.
\end{itemize}
We tackle this problem from the viewpoint of integrable discrete differential geometry, a recent field of research that aims to develop structure-preserving and self-contained integrable discretizations of smooth differential geometric concepts. 

In this realm quadrilateral surfaces with circular faces can be understood as equivalents of smooth surfaces parametrized by lines of curvature \cite{book_discrete}. Circular nets fulfilling property (1) have been recently studied in \cite{bobenko2023circular} (see also \cite{multi_nets, cone_nets} for partial classifications). The architectural design of nets with planar parameter lines is discussed in~\cite{planar_pottmann, TELLIER2019102880}.

\bigskipp The additional requirement~(2) adds an integrable structure to those nets, mimicking a smooth conformal metric \cite{BobenkoPinkalliso}. Therefore, quadrilateral surfaces with the two required properties provide discrete counterparts to smooth isothermic surfaces with a family of planar or spherical lines of curvature (see for example \cite{bernstein,bobenko2023isothermic,chopembszew,Darboux_planar, Darboux_book_iso,musso_darboux_dupin} for the study of the smooth case). 

Notably, this rich class of smooth surfaces contains several pioneering examples that provided solutions to long-standing open problems. Foremost among these is Wente's torus \cite{wente}, a compact immersed surface of constant mean curvature. Recently, specific isothermic tori with one family of planar curvature lines played a crucial role in the discovery of compact Bonnet pairs \cite{bobenko2023compact}. Furthermore, great interest in this surface class also arises from free boundary and capillary problems: the authors in \cite{minimal_free_boundary} detected compact free boundary minimal annuli with a family of spherical curvature lines that are immersed in the unit ball $\mathbb{B}^3$ in $\mathbb{R}^3$; see~\cite{cerezo2022annular, cerezo2024free} for similar problems regarding CMC-surfaces. 

\bigskipp In the spirit of discrete differential geometry, in the paper we also consider discrete nets fulfilling a slightly generalized condition~(2) and require that
\begin{itemize}
\item[(2a)] the cross-ratios along each stripe bounded by two planar or spherical parameter lines are constant.
\end{itemize}
This amounts to a curvature lines preserving  reparametrization along the planar or spherical curvature lines of a smooth isothermic surface \cite{book_discrete}.
\\\\To the best of our knowledge a full description of discrete isothermic nets with a family of spherical curvature lines is an open problem. Simple examples are provided by discrete isothermic channel surfaces, where one family of parameter lines lie on circles (see~\cite{discrete_channel}). A numerical compact example on a (5x7)-lattice that satisfies both properties~(1) and (2) has been given in \cite{bobenko2023compact}. 
First explicit examples of compact discrete isothermic tori with cross-ratios factorizing function that are Darboux transformations of discrete homogeneous tori in 3-space have been recently computed in~\cite{iso_tori_discrete}.
\\\\As one main result in Theorem~\ref{thm_mflexible} and Corollary~\ref{cor_unfolding} we will prove that any discrete isothermic net fulfilling properties (1) and (2a) may be associated to a discrete holomorphic map (a planar circular net with constant cross-ratio along each stripe in one coordinate direction). Each planar or spherical parameter line of the isothermic net is then related by a specific M\"obius transformation to the corresponding parameter line in the discrete holomorphic map. Conversely, we will also characterize all discrete holomorphic maps that give rise to the sought-after isothermic nets in 3-space.

For discrete planar parameter lines our approach is a discrete version of a classical construction originally found by Ribaucour and Darboux. The case of spherical parameter lines is a generalization of this concept and a novel approach to the study of spherical curvature lines. Below, we will briefly describe the constructions that are known for creating smooth surfaces with planar curvature lines.
%
%
\\\\\textbf{Deformations of smooth surfaces with a family of planar curvature lines \cite{rib,  roquet}.} Given a smooth surface $f: \mathbb{R}^2 \supset U \times I \to \mathbb{R}^3$ with a family of planar curvature lines: $f_i(u):=f\vert_{i = const} \in \Pi_i$, where $(\Pi_i)_{i \in I}$ denotes the family of planes containing the planar $u$-lines of curvature. Then the planes $(\Pi_i)_i$ viewed as tangent planes, generically envelop a developable surface~$\Delta$. As a ruled surface with vanishing Gaussian curvature, $\Delta$ can be isometrically deformed such that the generators are preserved. Any such deformation induces a map $\Phi: I \to \{ \text{isometries in } \mathbb{R}^3 \}$ that maps in particular each tangent plane $\Pi_i$ to the corresponding tangent plane~$\Phi_i(\Pi_i)=:\hat{\Pi}_i$ of the deformed developable surface. 

If these isometries are applied to the planar curvature lines of~$f$, a new surface $\hat{f}:=(\Phi_i(f_i))_{i \in I}$ with planar parameter lines in $\hat{\Pi}_i$ is obtained. As Ribaucour has proven, those planar curves are again curvature lines on $\hat{f}$. 

Thus, generator-preserving isometric deformations of developable surfaces induce deformations within the class of surfaces with a family of planar lines of curvature. We call such deformations \emph{lifted-foldings}. As a special case, if $\Delta$ is unfolded to a plane, the new surface~$\hat{f}$ yields an  orthogonally parametrized planar region.
\\\\\textbf{Reduction to smooth isothermic surfaces with planar parameter lines.} Darboux~\cite{Darboux_planar, Darboux_book_iso} found that, astonishingly, this deformation for surfaces with planar lines of curvature interplays well with conformal metrics: an isothermic surface with a family of planar curvature lines stays isothermic under lifted-foldings. 

Particularly, by unfolding the associated developable surface~$\Delta$ to a plane, any isothermic surface with planar curvature lines is naturally related to a holomorphic map. Explicit parametrizations of these special holomorphic maps can be specified using elliptic functions \cite{adam, bobenko2023isothermic, Darboux_book_iso} and result in parametrizations for isothermic surfaces with a family of planar curvature lines.

Recently \cite{bobenko2023isothermic}, isothermic tori from fundamental pieces with one family of planar curvature lines were classified. The authors use this method to close the isothermic surfaces in a two-step procedure: once the potential planar parameter lines in the holomorphic map are closed, admissible lifted-foldings then lead to isothermic tori in $\mathbb{R}^3$.

Independently of this approach, in \cite{chopembszew} the geometry of planar and spherical curvature lines of Lie applicable surfaces (which also include isothermic surfaces) were studied. It turned out that such lines of curvature, when suitably projected, are Lie transforms of constrained elastic curves in various space forms as described in~\cite{pinkall_willmore, Heller_elastic}.  

\begin{figure}
\begin{minipage}{4cm}
\hspace*{-3cm}\includegraphics[scale=0.15]{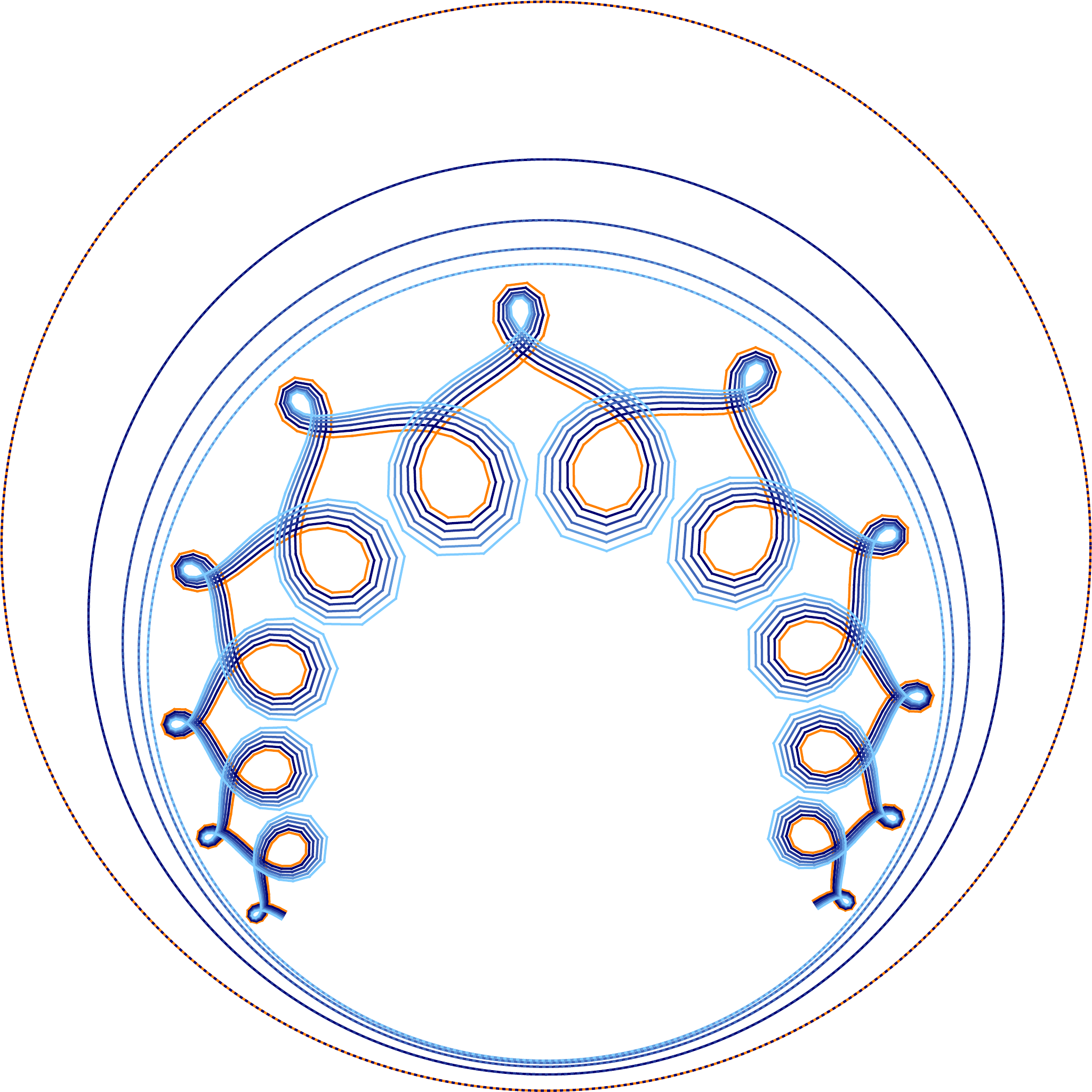}
\end{minipage}
\begin{minipage}{4cm}
\hspace*{-1.4cm}\includegraphics[scale=0.24]{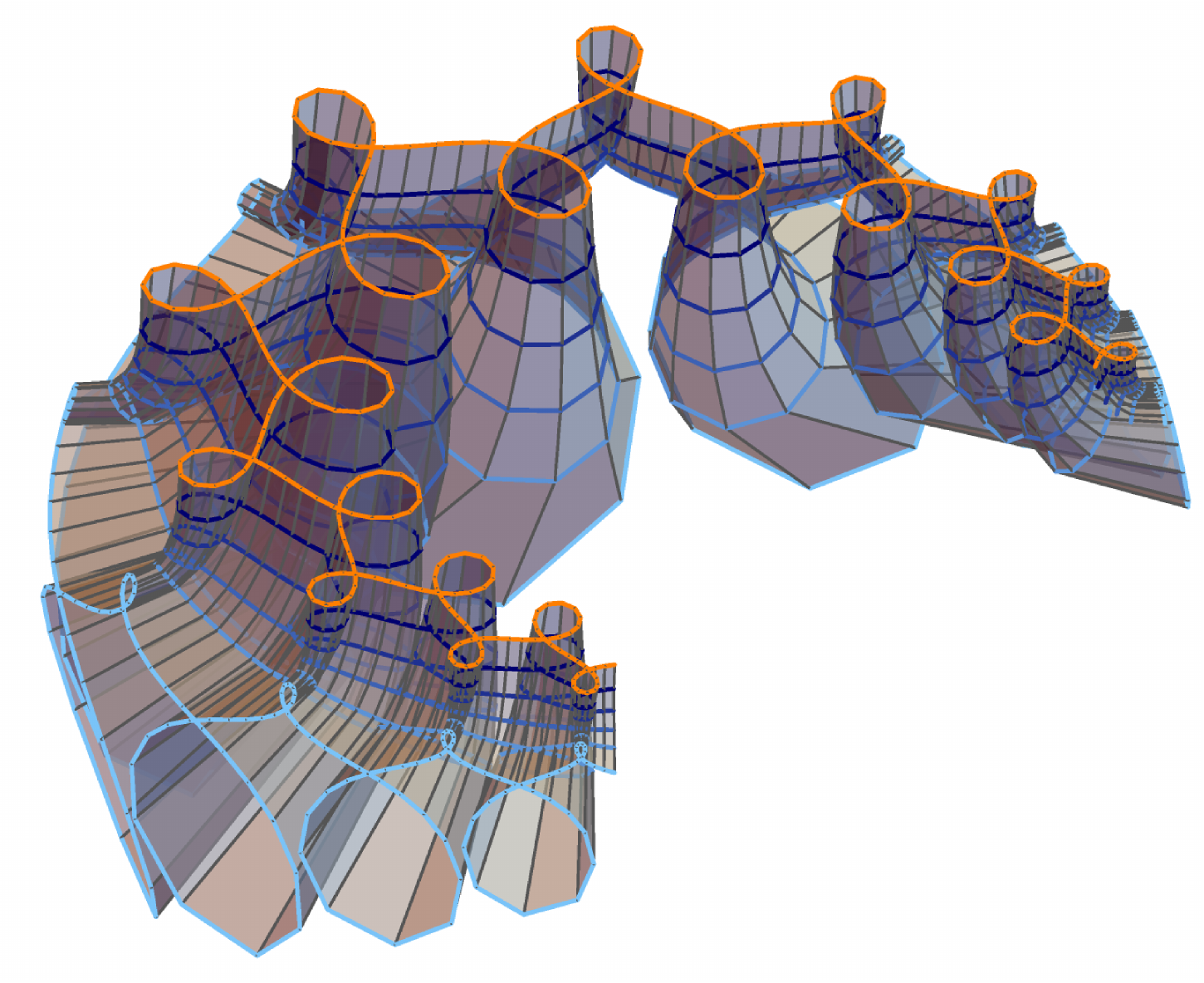}
\end{minipage}
\begin{minipage}{4cm}
\hspace*{0.8cm}\includegraphics[scale=0.25]{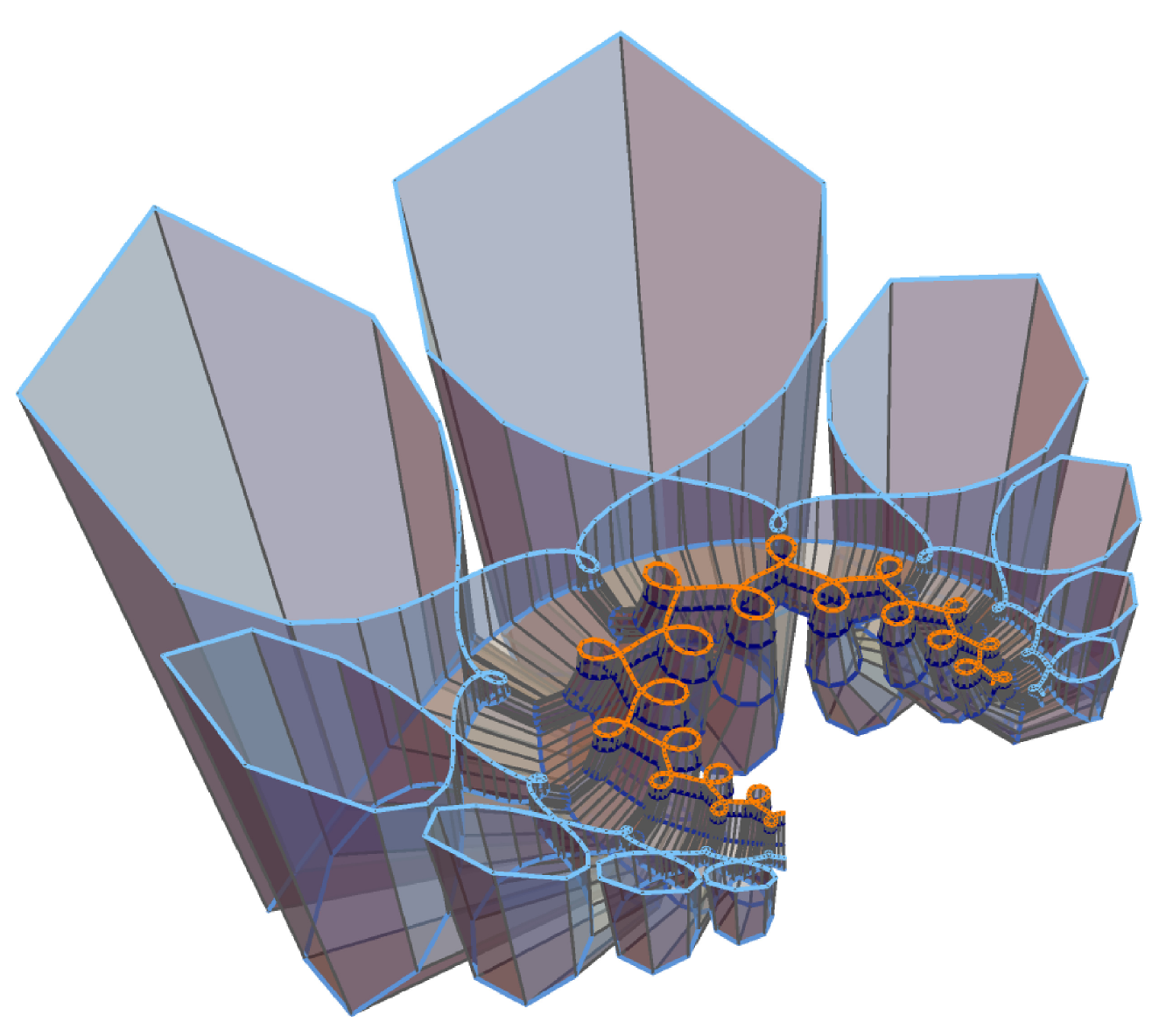}
\end{minipage}
\caption{A discrete holomorphic map with circular folding axes and two discrete isothermic nets with a family of spherical curvature lines obtained via lifted-folding.}\label{fig_girlande}
\end{figure}

\bigskipp The contents and main results of the paper are as follows.
%
%

\textbf{Section~\ref{sect_general_spherical}} is devoted to the construction of quadrilateral surfaces satisfying condition~(1), that is, circular nets with a family of planar or spherical lines of curvature. By using methodologies developed in~\cite{r_congr}, we show that the evolution map of any circular stripe bounded by two spherical parameter lines has a 3-dimensional subspace of fixed points in~$\mathbb{R}^{4,2}$ (see Proposition~\ref{prop_spherical_evolution}). Based on this observation, we develop the concept of lifted-folding to generate discrete circular nets with a family of planar or spherical curvature lines. As shown in~Theorem~\ref{thm_mflexible}, this approach provides deformations in the class of these circular nets. In particular, any such net can be “flattened” to a circular net in a plane or a 2-sphere.
%

To conclude this section we generate two special subclasses: circular nets with a family of parameter lines in Euclidean planes and discrete Joachimsthal surfaces. The constructions reveal that also in the discrete case lifted-foldings provide deformations within these two  subclasses.
%
%
\\\\Key for the study of nets satisfying properties (1) and (2a) in \textbf{Section~\ref{sect_iso}} is Proposition~\ref{prop_constant_cr}:  lifted-folding preserves constancy of cross-ratios along a circular stripe bounded by two spherical parameter lines. Thus, via lifted-folding, any discrete isothermic net with a family of spherical curvature lines can be associated to a discrete holomorphic map. 

As a consequence, generating all isothermic nets with a family of spherical parameter lines is equivalent to finding all discrete holomorphic maps that allow for lifted-foldings. We classify all these holomorphic maps and give a description in terms of appropriate initial data (see Construction~2).

As expected from their smooth counterparts, special solutions are provided by M\"obius transforms of discrete constrained elastic curves in space forms in the sense of~\cite{discrete_elastic}. In more detail, we prove that any discrete constrained elastic curve in a space form can be extended to a discrete holomorphic map that allows for lifted-foldings. The other coordinate lines are then also discrete constrained elastic in appropriate space forms. 

%
%

\bigskipp In \textbf{Section~\ref{sect_special}} we discuss some monodromy questions related to the posed problem. Closedness or quasi-periodicity of one initial curve is automatically propagated if the curve is extended to a holomorphic map that admits lifted-foldings. Therefore, our presented method turns out to be highly efficient to create isothermic nets with certain monodromies.  

We remark that the generated discrete quasi-periodic holomorphic maps are solutions of the cross-ratio dynamics studied in \cite{affolter2023integrable, AFIT_dynamics} (see also \cite{periodic_conformal} for the periodic case). Explicit examples obtained from quasi-periodic discrete elastic curves in Euclidean space \cite{fairing_elastica} as initial curves, as well as their natural generalizations to space forms, are investigated.

Furthermore, since lifted-folding preserves closedness of the spherical parameter lines, we immediately obtain discrete isothermic cylinders when starting with a periodic discrete holomorphic map. Thus, to construct topological tori, only the non-spherical coordinate direction has to be controlled.   In Subsection~\ref{subsect_tori}, by applying some reflection principle, we derive discrete isothermic tori composed of fundamental pieces.
%
%

\bigskipp\textbf{Technical framework.} We consider quadrilateral surfaces on a graph $\mathcal{G}^2=(\V^2, \E^2)$ that represents a simply connected subset of the lattice $\mathbb{Z}^2$, where $\V^2$ and $\E^2$ denote the vertices and edges, respectively. With respect to the underlying structure of $\mathbb{Z}^2$, any such quadrilateral surface decomposes into two sequences of discrete coordinate curves. Consecutive indices $\{i,j,k, \cdots \}$ denote neighbouring vertices in such a discrete coordinate curve. The edge that connects the vertex $i$ and $j$ is denoted by $(ij)$.

A discrete curve will live on a 1-dimensional connected graph~$\mathcal{G}=(\V, \E)$.

\bigskipp We will describe our theory in the light cone model, because this provides a simple formulation for the concept of lifted-folding. Moreover, it will also enable us in Subsection~\ref{subsect_holo_elastic} to directly interpret the geometry of the appearing discrete curves in space forms. 

For more details about the light cone model the reader is referred to the Appendix.

\bigskipp \textbf{Acknowledgements.} We thank Andrew Sageman-Furnas and Jannik Steinmeier for pleasant discussions around this topic. The second author is further very grateful for helpful discussions about smooth surfaces with spherical curvature lines with Fran Burstall, Joseph Cho and Mason Pember during a “Research in Pairs” stay in Oberwolfach. 

This research was supported by the DFG Collaborative Research Center TRR 109 “Discretization in Geometry and Dynamics”. 
\section{Discrete planar and spherical curvature lines}\label{sect_general_spherical}
\noindent In this section we demonstrate that circular nets with a family of discrete planar or spherical lines of curvature allow for canonical deformations that preserve circularity of the net as well as planarity or sphericity of those parameter lines. The discrete theory presented here imitates and generalizes a classical approach originally found by Ribaucour \cite{rib, roquet}.

In particular, it will be proven that these nets can be “flattened” to a circular net in a plane or a 2-sphere. Thus, this concept establishes a connection between circular nets in 3-space with parameter lines on spheres and specific circular nets in 2-space.

Since our approach is invariant under M\"obius transformations, we do not distinguish between planes and spheres unless explicitly stated. Thus, the set of spheres includes Euclidean planes.
%
%
\subsection{Ribaucour transformations of discrete curves}
It is well-established that curvature line parametrizations of smooth surfaces can be modelled by quadrilateral surfaces with circular faces~\cite{book_discrete}. That is, corresponding point pairs of two neighbouring discrete curvature lines are concircular. Thus, two adjacent discrete curvature lines are related by a so-called \emph{discrete Ribaucour transformation}. 

As pointed out in \cite{r_congr}, any Ribaucour pair of two discrete curves $(f, g)$ admits a description in terms of two initial points $(f_0, g_0)$ and a 1-parameter family of M\"obius transformations that simultaneously evolve this initial point pair. To be more precise, in the light cone model, those M\"obius transformations are given by simple \emph{M-inversions}, i.\,e.\,reflections that preserve the point sphere complex and therefore map points to points. Geometrically, those are either usual M\"obius reflections in spheres or orientation-reversing antipodal reflections (c.\,f.\,Appendix).

After introducing some basic notions, we show that Ribaucour pairs of spherical curves are characterized by specific evolutions.

\bigskip

\begin{defi}
Let $f: \mathcal{V} \rightarrow \Light$ be a discrete curve, then a map $\sigma: \mathcal{E} \rightarrow \{ \text{M-inversions} \}$ is said to be an \emph{evolution map of $f$} if $\sigma_{ij}(\mathfrak{f}_i)=\mathfrak{f}_j$.

In the light cone model such M-inversions are given by reflections 
\begin{equation*}
\sigma_{ij}(x)=x-\frac{2\inner{x,\mathfrak{a}_{ij}}}{\inner{\mathfrak{a}_{ij},\mathfrak{a}_{ij}}}\mathfrak{a}_{ij}, \ \text{where} \ \mathfrak{a}_{ij}:=\mathfrak{f}_i -\mathfrak{f}_j \perp \p
\end{equation*}
for a fixed choice of homogeneous coordinates $\mathfrak{f}_i$ and $\mathfrak{f}_j$. We call the family of vectors $\{a_{ij}\}_\mathcal{E}$ the \emph{evolution complexes}.
\end{defi}
\noindent Any evolution map of $f$ gives rise to discrete Ribaucour transforms: let $g_0 \in \mathbb{P}(\mathcal{L})$ be an arbitrary initial point, then the iteratively defined curve 
\begin{equation*}
g: \V \to \Light, \ \g_j:=\sigma_{ij}(\g_i)
\end{equation*}
yields a Ribaucour transform of $f$. 

This relies on the fact that any two neighbouring point pairs $f_i, f_j, g_i, g_j$ lie in the 3-dimensional subspace $\spann{\f_i, \g_i, \mathfrak{a}_{ij}} \subset \mathbb{R}^{4,2}$ and are therefore circular.

\bigskipp Conversely, any discrete Ribaucour pair of curves gives rise to a canonical evolution map:
\begin{lemdef}
For a discrete Ribaucour pair of curves there exists a unique evolution map $\sigma$ that simultaneously maps adjacent point pairs onto each other: $\sigma_{ij}(\f_i)=\f_j$ and $\sigma_{ij}(\g_i)=\g_j$ for suitable homogeneous coordinates.

We call $\sigma$ the \emph{R-evolution map} of the Ribaucour pair $(f,g)$ with corresponding \emph{R-evolution complexes}.
\end{lemdef}
\begin{proof}
Since the curve points $f_i, f_j, g_j$ and $g_i$ of any quadrilateral of a discrete Ribaucour pair are concircular, those four points lie in a 3-dimensional subspace of $\mathbb{R}^{4,2}$ and are therefore linearly dependent. In particular, there exist homogeneous coordinates such that
\begin{equation*}
\f_i - \f_j + \g_j - \g_i=0.
\end{equation*} 
As a consequence, the unique R-evolution complexes are then given by $\rr_{ij}:=\f_i - \f_j = \g_i - \g_j$.
\end{proof}
%
%

\bigskip 

%
\noindent As we will see in the following lemma, circular nets with a family of spherical curvature lines admit special types of evolution maps:
%
\begin{prop}\label{prop_spherical_evolution}
Let $f$ and $g$ be two discrete spherical curves with curve points lying on two distinct spheres $s^1, s^2 \in \Light$, that is,
\begin{equation*}
\inner{\f_i, \s^1}=\inner{\g_i, \s^2}=0 \ \ \text{for all } i \in \mathcal{V}.
\end{equation*}
If $(f, g)$ is a discrete Ribaucour pair, then all elements in the 3-dimensional subspace 
\begin{equation*}
\mathcal{M}:= \spann{\mathfrak{s}^1, \mathfrak{s}^2, \mathfrak{p}} \subset \mathbb{R}^{4,2}
\end{equation*}
provide fixed points for all M-inversions in the R-evolution map. 
\end{prop}
\begin{proof}
Due to the circularity of adjacent point pairs, for any quadrilateral we can choose homogeneous coordinates such that $\mathfrak{f}_i-\mathfrak{f}_j+\mathfrak{g}_j-\mathfrak{g}_i=0$. By taking inner products with the spheres $s_1$ and $s_2$, we obtain
\begin{equation*}
0=\inner{\f_i - \f_j,\s^2}=\inner{\g_i - \g_j,\s^1} 
\end{equation*}
and immediately conclude that $\rr_{ij}:=\f_i-\f_j=\g_i-\g_j \perp \s^1, \s^2, \p$ and the claim follows (see Fact~\ref{fact_fixedpoint}).
\end{proof}
%
%
\noindent This crucial property of R-evolution maps for spherical curves motivates the following definition:
\begin{defi}
A Ribaucour pair of two discrete curves $(f, g)$ is said to be of \emph{($\m^1,\m^2$)-type} if 
\begin{equation*}
\mathcal{M}:=\spann{\m_1, \m_2, \p} \subset \mathbb{R}^{4,2}
\end{equation*}
is a 3-dimensional subspace of fixed points for all inversions in the corresponding R-evolution map.
\end{defi} 
%
\begin{cor}\label{cor_mtype_spherical}
A discrete Ribaucour pair of $(\m^1, \m^2)$-type consists of two spherical curves.
\end{cor}
\begin{proof}
Let $f$ denote one of the discrete curves of an $(\m^1, \m^2)$-type Ribaucour pair. Excluding the trivial case where all curve points lie on a circle, we can assume without loss of generality that the four consecutive curve points $f_0, f_1, f_2$ and $f_3$ uniquely determine a sphere $s \in \Light$. 

Moreover, observe that all R-evolution complexes~$r$ along the discrete curve lie in a 3-dimensional subspace $\mathcal{M}^c:=\spann{\rr_{01}, \rr_{12}, \rr_{23}} \perp \m^1, \m^2, \p$. Thus, since all curve points of $f$ are given as linear combinations of an initial point~$f_0$ and the R-evolution complexes, we conclude that all curve points lie in the 4-dimensional subspace
\begin{equation*}
\spann{\f_0, \f_1,\f_2,\f_3} = \spann{\f_0, \rr_{01}, \rr_{12}, \rr_{23}} = \spann{f_0 \cup \mathcal{M}^c} \perp s,
\end{equation*} 
which shows that $f$ is indeed a spherical curve lying on $s$.
\end{proof}
%
%
\subsection{Lifted-folding} M\"obius transformations preserve circularity between points, therefore it is clear that discrete Ribaucour pairs are invariant under those transformations. 

Additionally, Ribaucour pairs of $(\m^1,\m^2)$-type have the remarkable property that circularity is also preserved while only \emph{one} of the curves is transformed by a suitable M\"obius inversion. Generically, there exists a 1-family of such M\"obius inversions, which we describe in the next proposition:
\begin{prop}\label{prop_flexible_pair}
Let $(f,g)$ be a Ribaucour pair of $(\m^1,\m^2)$-type and suppose that $\sigma_n$ is an inversion with respect to 
\begin{equation*}
\mathfrak{n}:= \mathfrak{m}^1 + \lambda \mathfrak{m}^2 + \inner{ \mathfrak{m}^1 + \lambda \mathfrak{m}^2, \mathfrak{p}}\mathfrak{p}
\end{equation*}
for some $\lambda \in \mathbb{R}$. Then $(f, \sigma_n(g))$ is also a Ribaucour pair with the same evolution map as $(f,g)$.
\end{prop}
\begin{proof}
The proof relies on the fact that $\sigma_n$ commutes with all inversions $\sigma_r$ in the evolution map of the original Ribaucour pair $( f,g )$: $\sigma_{r_{ij}} \circ \sigma_n = \sigma_n \circ \sigma_{r_{ij}}$ for all $(ij) \in \E$. Since two inversions commute if and only if the corresponding complexes are perpendicular, this follows directly from $\mathfrak{r}_{ij} \perp \mathfrak{m}^1, \mathfrak{m}^2, \mathfrak{p}$.

Thus, we obtain that the R-evolution map $\sigma_{r}$ of $( f,g )$ also evolves the vertices of the curve~$\sigma_n(g)$: 
\begin{equation*}
\sigma_{r_{ij}} (\sigma_n (\g_i))=\sigma_n (\sigma_{r_{ij}} (\g_i)) = \sigma_n (\g_j)
\end{equation*}
for suitable homogeneous coordinates.
\end{proof}
%
%
\begin{figure}
\hspace*{0.4cm}\begin{minipage}{2cm}
\vspace*{0.5cm}\hspace*{-2cm}\\\hspace*{-1.9cm}\includegraphics[scale=0.19]{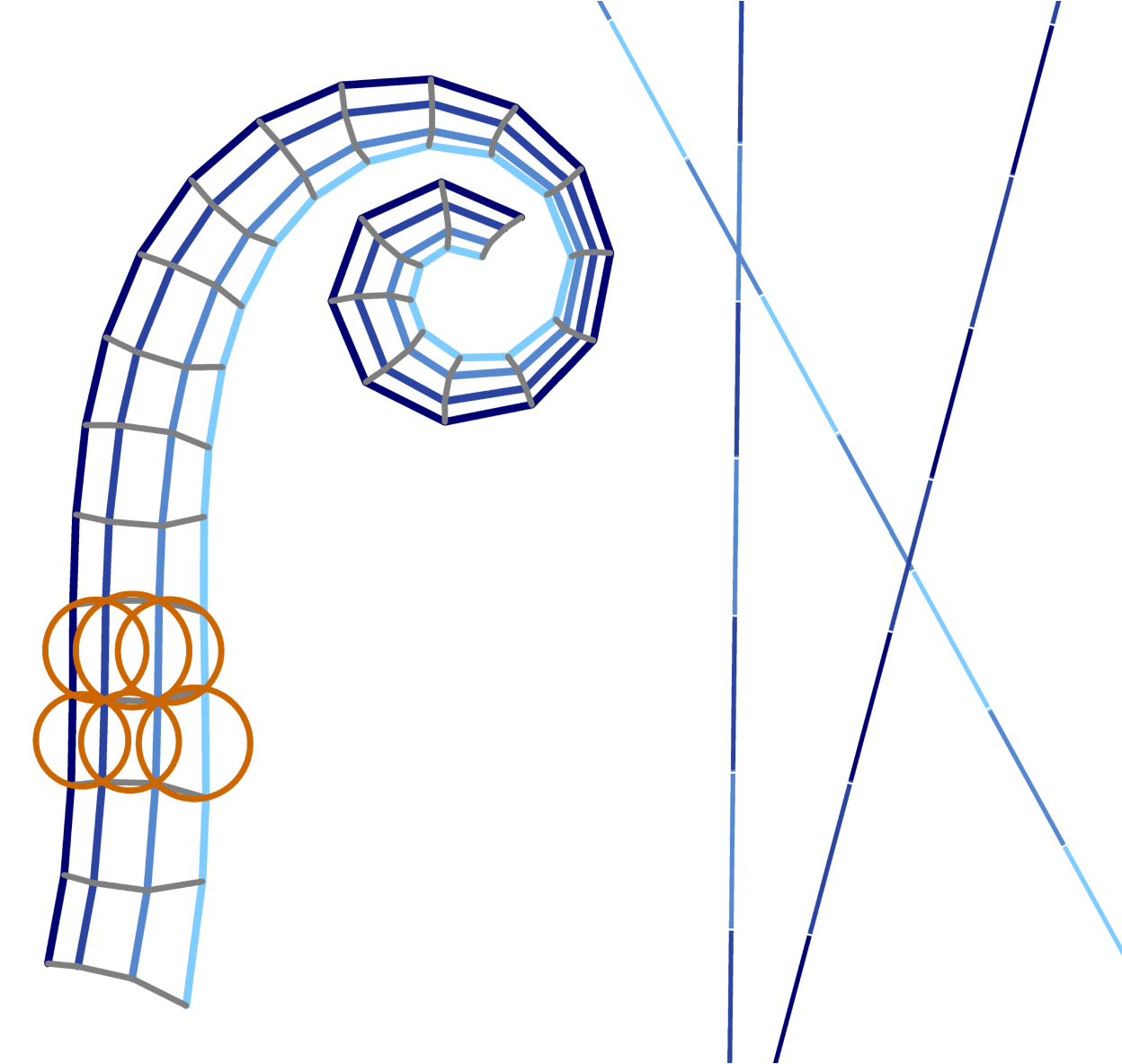}
\end{minipage}
\begin{minipage}{5cm}
\hspace*{0.1cm}\includegraphics[scale=0.6]{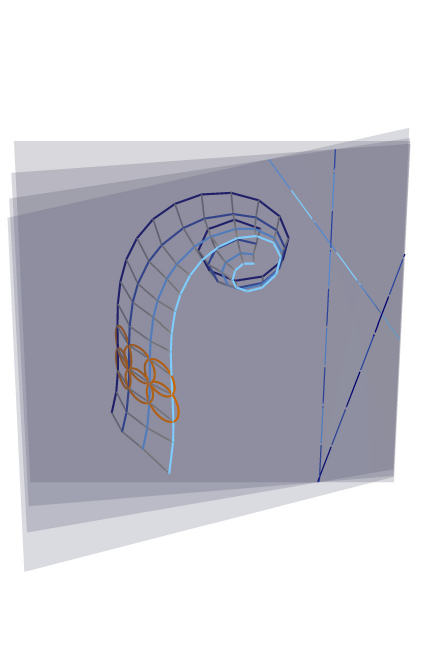}
\end{minipage}
\hspace*{-0.5cm}\begin{minipage}{4cm}
\includegraphics[scale=0.6]{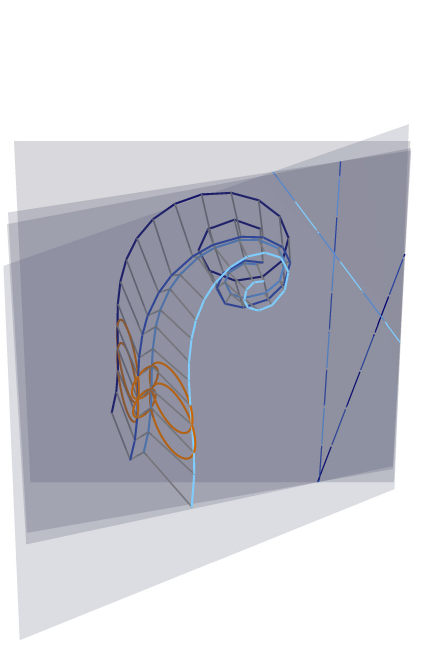}
\end{minipage}
\hspace*{0.4cm}\begin{minipage}{3.1cm}
\includegraphics[scale=0.6]{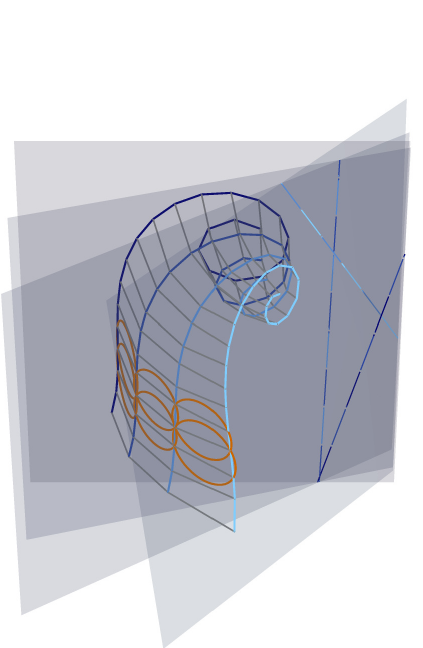}
\end{minipage}
\caption{\textit{Left}. A planar sequence of $(\m^{ij})$-type Ribaucour transformed discrete curves together with their folding axes (dashed lines). \textit{Middle and Right.} The nets in 3-space with planar curvature lines are generated from the planar net via lifted-foldings. The planes are rotated around their corresponding folding axes.}\label{fig_planar}
\end{figure}
%
%
\noindent Suppose that a discrete Ribaucour pair $(f, g)$ consists of two curves lying in distinct spheres $s^1$ and $s^2$. If the discrete curve~$g$ is transformed as specified in Proposition~\ref{prop_flexible_pair} to the curve $\sigma_n(\g)$, the obtained curve lies on the sphere $\bar{\s}:=\sigma_n(\s^2)$. By construction, this new sphere~$\bar{\s}$ lies in the subspace~$\spann{\s^1, \s^2, \p} \subset \mathbb{R}^{4,2}$ and therefore in the sphere pencil determined by the original spheres $s^1$ and $s^2$.

Note that, by varying the parameter~$\lambda \in \mathbb{R}$, the curve~$\g$ can be positioned on any oriented sphere in this pencil.
 
Thus, the following three geometric configurations can occur:
\begin{itemize}
\item \emph{elliptic pencil}: $s^1$ and $s^2$ intersect and all spheres in the pencil contain this intersection circle. 
\item \emph{parabolic pencil}: if $s^1$ and $s^2$ have exactly one point $v$ in common, then all spheres lie (up to orientation) in the contact element~$\spann{\s^1, \mathfrak{v}}$;
\item \emph{hyperbolic pencil}: if $s^1$ and $s^2$ have no points in common, then none of the spheres in the pencil have a point in common.
\end{itemize}
%
\bigskipp We now turn to the main objects of interest, namely to circular nets with a family of spherical curvature lines. By combining the above facts, we can now conclude that those nets allow for special deformations. In the following Theorem, we describe this flexibility in detail (for examples see Figure~\ref{fig_planar} and \ref{fig_spherical}). 

\bigskip
 
\noindent\textit{Regularity assumption.} In what follows, unless otherwise stated, we only consider \emph{non-degenerate} discrete circular nets with a family of spherical curvature lines. That is, we assume neighbouring curvature lines to lie on distinct spheres. Thus, we exclude totally spherical or planar patches, where the following results could obviously fail. 
\begin{thm}\label{thm_mflexible}
Suppose that $F:\V^2  \rightarrow \Light$ is a non-degenerate circular net with a family of spherical curvature lines, that is, $F=\{f^{(1)}, f^{(2)}, \cdots, f^{(j)}, \cdots \}$ is given by a sequence of consecutive Ribaucour transformed curves $f^{(j)}$ lying on pairwise distinct spheres $\{ \mathfrak{s}_1, \mathfrak{s}_2, \cdots, \mathfrak{s}_j, \cdots \}$, where we assume the normalization $\inner{\mathfrak{s}_j, \mathfrak{p}}=-1$ for all spheres. 

Then for any choice of \emph{folding parameters} $\{ \lambda_{12}, \lambda_{23}, \cdots, \lambda_{ij}, \cdots \}$, $\lambda_{ij} \in \mathbb{R}$, the net 
\begin{equation*}
\tilde{F}:=\{ f^{(1)}, \sigma_{12}(f^{(2)}), \cdots, \sigma_{jk} \circ \cdots \circ \sigma_{12} (f^{(k)}), \cdots \}  
\end{equation*}
is again a circular net with spherical curvature lines lying on the spheres $\{ \mathfrak{s}_1, \mathfrak{\bar{s}}_2, \cdots, \mathfrak{\bar{s}}_j, \cdots \}$, where the inversions $\sigma_{jk}$ are given with respect to the linear complexes
\begin{equation*}
\begin{aligned}
\mathfrak{n}_{12}:= \mathfrak{s}_1 + \lambda_{12}\mathfrak{s}_2 - & (1+\lambda_{12})\mathfrak{p},
\\ \mathfrak{n}_{jk}:=  \mathfrak{\bar{s}}_j + \lambda_{jk}\mathfrak{\hat{s}}_k - & (1+\lambda_{jk})\mathfrak{p}, 
\\\text{with } \mathfrak{\bar{s}}_j:=\sigma_{ij} \circ \cdots \circ \sigma_{12} (\mathfrak{s}_{j})  \text{ and } & \mathfrak{\hat{s}}_k:=\sigma_{ij} \circ \cdots \circ \sigma_{12} (\mathfrak{s}_{k}).
\end{aligned}
\end{equation*}
A deformation of this kind is said to be a \emph{lifted-folding}; the M-inversions~$\sigma$ are called \emph{folding inversions}.
\end{thm}
%
%
\begin{proof}
The claim follows from Proposition~\ref{prop_spherical_evolution} and \ref{prop_flexible_pair}.
\end{proof}
%
%
\begin{rem}\label{rem_alt_folding}
By Proposition~\ref{prop_spherical_evolution}, a circular net with a family of spherical lines of curvature is given by a sequence of $(\m^1, \m^2)$-type Ribaucour transforms. This fact gives an alternative description of the folding inversions described in Theorem~\ref{thm_mflexible}. Those are inversions in the linear complexes
\begin{equation*}
\mathfrak{n}_{ij}:= \mathfrak{m}_{ij}^1 + \tilde{\lambda}_{ij} \mathfrak{m}_{ij}^2 + \inner{ \mathfrak{m}_{ij}^1 + \tilde{\lambda}_{ij} \mathfrak{m}_{ij}^2, \mathfrak{p}}\mathfrak{p},
\end{equation*}
where $\tilde{\lambda}_{ij} \in \mathbb{R}$.

This reformulation extends the concept of lifted-folding also to sequences of $(\m_{ij}^{(1)},\m_{ij}^{(2)})$-type Ribaucour transforms on a 2-sphere.
\end{rem}
%
%
\noindent For circular nets with lines of curvature in Euclidean planes, the concept of lifted-foldings mimics the smooth construction as described in the introduction. For non-planar curvature lines it can be understood as a novel generalization of this approach to space form geometry, which we will explain below. This view is also the reason for the term “lifted-folding”.

\bigskipp As stated in the introduction, any smooth surface with a family of planar lines of curvature comes with a canonical developable surface enveloped by the planes containing the curvature lines. Lifted-foldings are then induced by generator-preserving isometric deformations of this developable surface. In the discrete setup, the geometric idea is analogous: the discrete family of planes determines a semi-discrete developable surface that has as smooth generators the intersection line of two adjacent planes. A lifted-folding then deforms the semi-discrete developable surface and the circular net with planar curvature lines simultaneously. Note that the folding inversions are rotations around the intersection lines and are therefore indeed Euclidean motions.

If two adjacent planes are parallel, this geometric interpretation fails as in the smooth case.

\bigskipp For circular nets with spherical curvature lines we obtain locally a similar interpretation in space forms. Recall from the Appendix that a space form~$\Q$ is determined by a space form vector $\q \in \mathbb{R}^{4,2}$ that satisfies $\q \perp \p$. Moreover, any sphere $s \in \Light$ with $\inner{\s, \q}=0$ has vanishing sectional curvature in this space form $\Q$. Thus, for any four consecutive spheres $s_1, s_2, s_3, s_4 \in \Light$ there exists a space form vector $\q \perp \p, \s_1, \s_2, \s_3, \s_4$ such that these four spheres have vanishing sectional curvature in the corresponding space form~$\Q$. If the consecutive spheres pairwise intersect, their intersection circles give rise to a semi-discrete surface which envelopes spheres with vanishing sectional curvature \cite{discrete_channel} and this channel surface can therefore be viewed as developable in $\Q$. The folding inversions are then isometries in $\Q$, since $\mathfrak{n}_{ij}\in \spann{\s_i,\s_j, \p} \perp \q$. 

In case the spheres do not intersect, the enveloped semi-discrete channel surface becomes imaginary.
%
%
\begin{figure}
\hspace*{1cm}\begin{minipage}{2cm}
\hspace*{-2cm}\hspace*{-3.6cm}\vspace*{0.8cm}\includegraphics[scale=0.22]{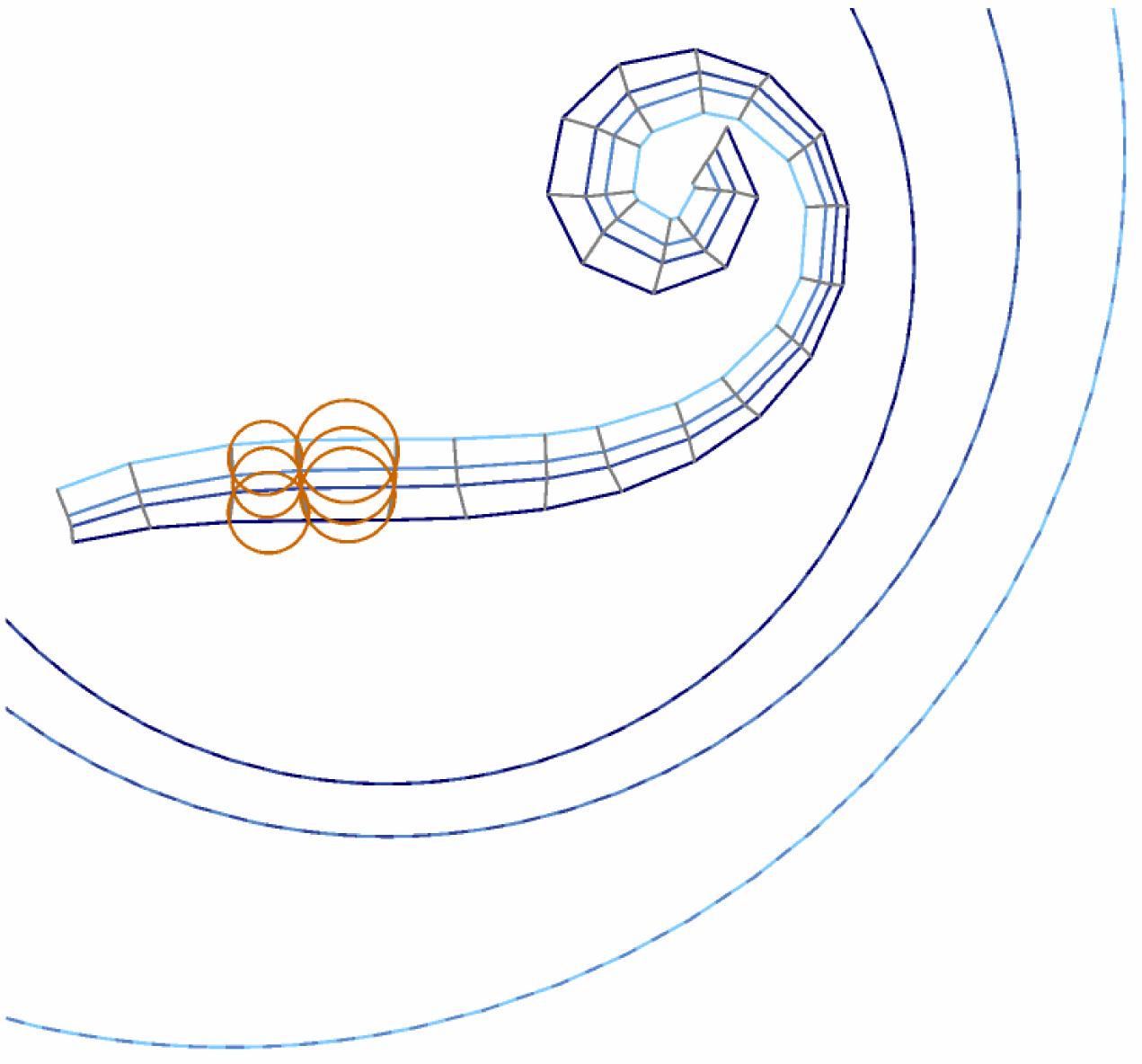}
\end{minipage}
\begin{minipage}{5cm}
\vspace*{-1.8cm}\hspace*{-1.7cm}\includegraphics[scale=0.3]{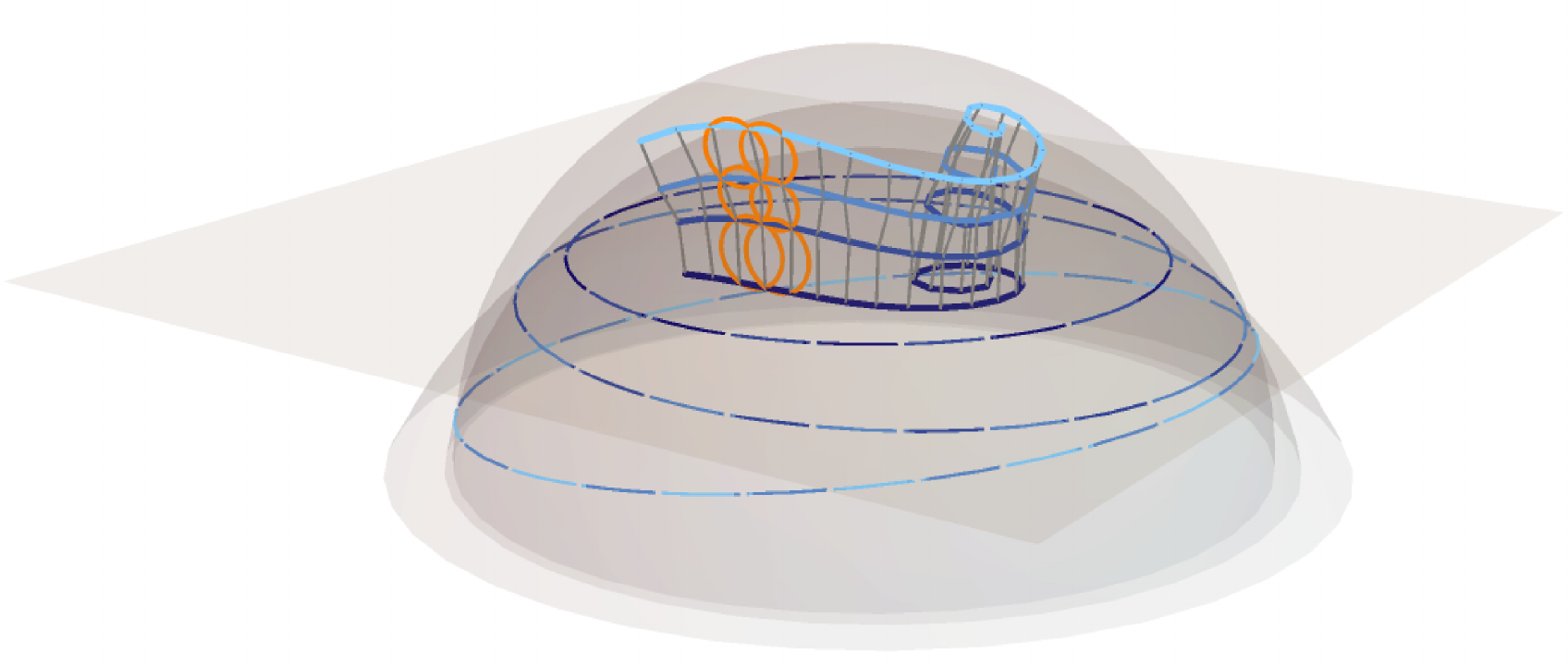}
\end{minipage}
\\\vspace*{-1.9cm}\begin{minipage}{5cm}
\hspace*{0.6cm}\vspace*{2cm}\hspace*{0.9cm}\includegraphics[scale=0.22]{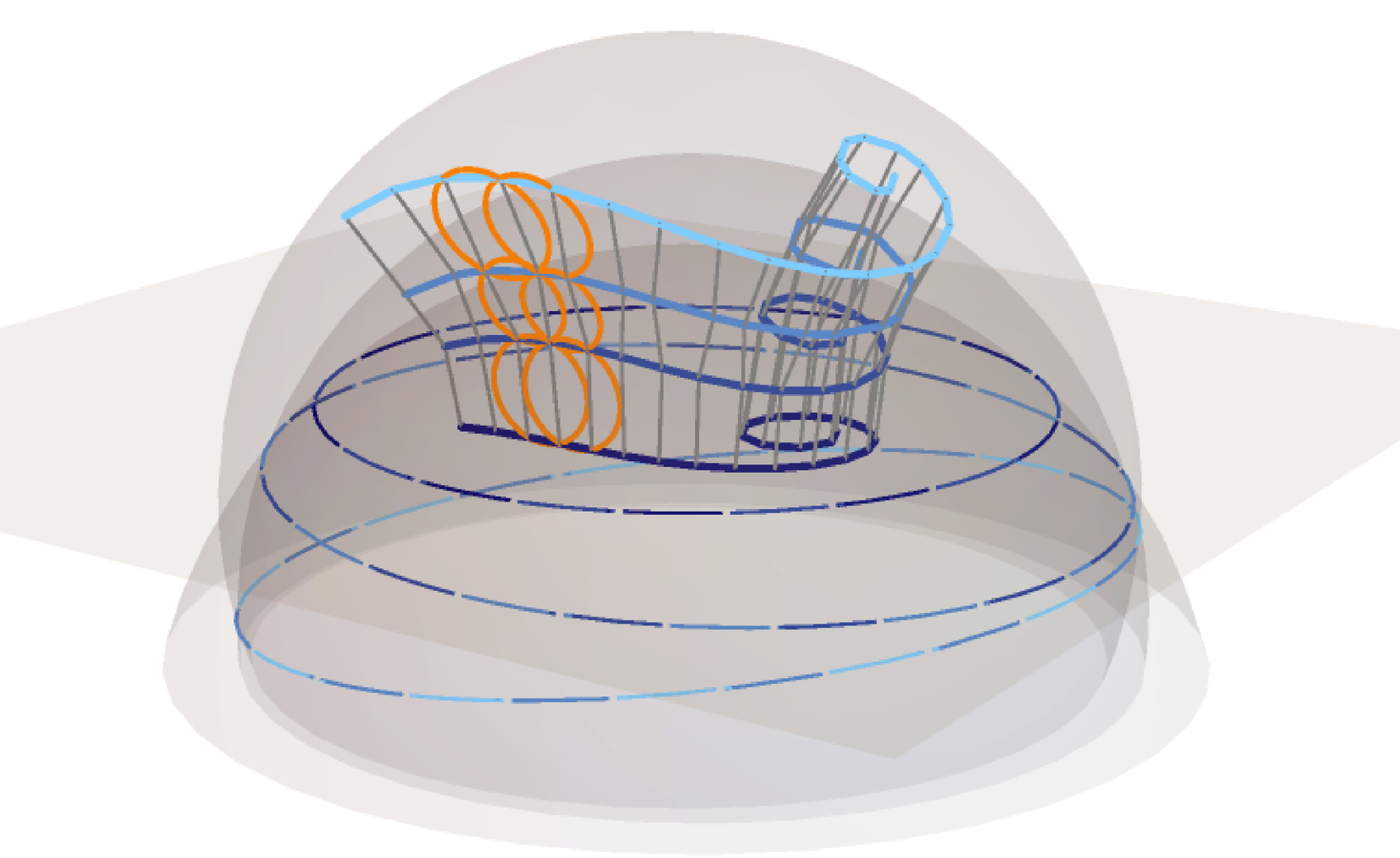}
\end{minipage}
\hspace*{3.1cm}
\begin{minipage}{5cm}
\vspace*{-1.7cm}\includegraphics[scale=0.24]{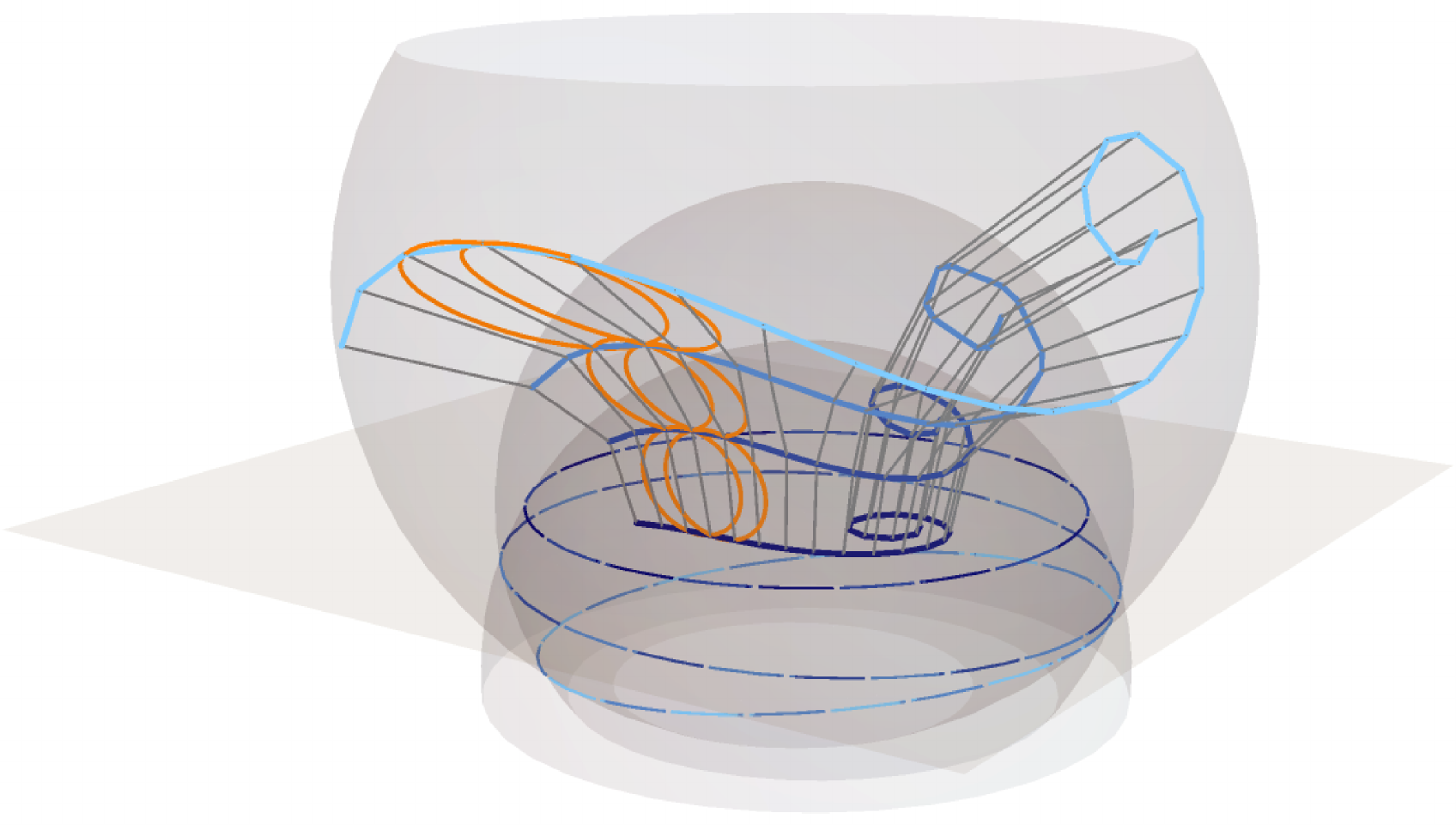}
\end{minipage}
\vspace*{-1.7cm}\caption{Circular nets with a family of spherical lines of curvature obtained from a planar circular net of $(M)$-type via lifted-foldings.}\label{fig_spherical}
\end{figure}
%
%
\subsection{Spherical curvature lines and 2-dimensional circular nets}\label{subsect_planar}
Lifted-foldings, as described in Theorem~\ref{thm_mflexible}, provide a way to trace back the geometric information about spherical parameter lines of a net in 3-space to a unique sphere: if the folding parameters are chosen as $\lambda_{ij}:=-1$, all spherical curvature lines are mapped to a single sphere. 

Conversely, Proposition~\ref{prop_flexible_pair} and Remark~\ref{rem_alt_folding} allow us to reconstruct circular nets with a family of spherical curvature lines in 3-space from circular nets on a 2-sphere $s$ that are composed of $(\s, \m^{(ij)})$-type Ribaucour transforms. 

\bigskip

In summary:   
%
\begin{cor}\label{cor_unfolding}
A non-degenerate circular net with a family of spherical curvature lines admits a lifted-folding to a circular net on a fixed sphere~$s$. The resulting curves on~$s$ are then M\"obius transformations of the original spherical parameter lines.

Conversely, any sequence of $(\s, \m^{(ij)})$-type Ribaucour transformed curves on a fixed 2-sphere~$s \in \Light \subset \mathbb{P}(\mathbb{R}^{4,2})$ gives rise to nets in 3-space with a family of spherical lines of curvature via lifted-folding. 
\end{cor}

In particular, since our constructions are all invariant under a global M\"obius transformation, for our further analysis it is sufficient to consider planar circular nets composed of $(\s, \m^{(ij)})$-type Ribaucour transforms.   

Thus, to describe them we consider now discrete curves in a plane which are modelled in $\mathbb{R}^{3,2}$. The notion of $(\m^1,\m^2)$-type Ribaucour transformations simplifies accordingly:
\begin{defi}
A discrete Ribaucour pair in $\mathbb{R}^{3,2}$ is of \emph{$(\m)$-type} if all elements in the 2-dimensional subspace $\mathcal{M}:=\spann{\m, \p} \subset \mathbb{R}^{3,2}$ are fixed points of all inversions in the corresponding R-evolution map.

We say that a planar circular net $G$ is of \emph{$(M)$-type} if it is given by a sequence $G=\{ g^{(1)}, g^{(2)}, \cdots \}$ of $(\m^{(ij)})$-type Ribaucour transforms.
\end{defi} 
\noindent Note that, by using a canonical embedding as discussed in the Appendix, a circular net of $(M)$-type in $\mathbb{R}^{3,2}$ can be embedded in $\mathbb{R}^{4,2}$. It then provides a planar circular net that admits lifted-foldings in 3-space.


%
%
\bigskipp If the curve points of an $(\m)$-type  Ribaucour pair~$(f,g)$ in $\mathbb{R}^{3,2}$ satisfy $\f_i, \g_i \notin \spann{\m}^\perp$, then in this case the R-evolution complexes simplify to  
\begin{equation}\label{equ_evolution_m}
r_{ij}:=\spann{\f_i \inner{\f_j, \m} - \f_j \inner{\f_i, \m}}=\spann{\g_i \inner{\g_j, \m} - \g_j \inner{\g_i, \m}}.
\end{equation}
This relies on the fact that the M-inversions in the evolution map interchange adjacent curve points and fix the vector $\m \in \mathbb{R}^{3,2}$ (see Fact~\ref{fact_inversion_complex}).
%
%

\bigskipp Planar circular nets of $(M)$-type in $\mathbb{R}^{3,2}$, given by a sequence $\{g^{(i)}\}_{i \in \hat{\mathcal{V}}}$ of discrete curves $\V \ni t \mapsto g_t$, are determined by the following initial data:
\begin{itemize}
\item a planar discrete curve $\V \ni t \mapsto g_t$, which then becomes a spherical curvature line under lifted-folding;
\\[-4pt]\item a planar discrete curve $\hat{\V} \ni i \mapsto h^{(i)}$ that coincides with $g$ at their intersection point: $g_0=h^{(0)}$; this curve provides initial points for the Ribaucour transforms;
\\[-4pt]\item a family $\hat{\E} \mapsto \m^{(ij)}$ of vectors in $\mathbb{R}^{3,2} \setminus \spann{\mathfrak{p}}$
\end{itemize}
The circular net $G=\{ g^{(i)} \}_{i \in \hat{V}}$ of $(M)$-type is then obtained iteratively via
\begin{equation*}
\begin{aligned}
\mathfrak{g}_t^{(0)}&:=\mathfrak{g}_t,
\\[6pt]\mathfrak{g}_{0}^{(j)}&:=\mathfrak{h}^{(j)} \ \text{and } \mathfrak{g}_{t+1}^{(j)}:=\sigma_a(\mathfrak{g}_{t}^{(j)}), \ \text{where } \mathfrak{a}:=\mathfrak{g}_t^{(i)}\inner{\mathfrak{g}_{t+1}^{(i)},\mathfrak{m}^{(ij)}} - \mathfrak{g}_{t+1}^{(i)}\inner{\mathfrak{g}_t^{(i)},\mathfrak{m}^{(ij)}}.
\end{aligned}
\end{equation*}

\bigskipp To conclude this section we show how the two subclasses of discrete surfaces with a family of planar curvature lines and discrete Joachimsthal surfaces are obtained via lifted-folding:
%
%
\\\\$\bullet$ \textbf{Nets with a family of planar parameter lines.} Smooth surfaces with one or two families of planar lines of curvature is a classical topic in differential geometry that has been studied from various perspectives (see for example~\cite{blaschke, bobenko2023isothermic, Darboux_planar, musso_nic}). 

Recent interest in discrete counterparts comes originally from architectural applications:  the authors in \cite{TELLIER2019102880} generate circular nets with planar parameter lines via Combescure transformations of suitable Gauss maps, while in \cite{planar_pottmann} a Laguerre geometric approach has been used to design those nets. From a purely geometric point of view, in \cite{bobenko2023circular} circular nets with planar parameter lines are studied using projective geometry. This work provides in particular insights into possible plane configurations for the special case that both families are planar.

\bigskipp The concept of lifted-folding becomes especially simple in the case of planar curvature lines. An example illustrating the geometric situation can be found in Figure~\ref{fig_planar}. First assume that two consecutive planes containing the parameter lines intersect in a line. Then a lifted-folding amounts to a rotation of one of the curves around this intersect line.  If the two planes are parallel, lifted-folding changes the distance between those parallel planes. 

From these observations it is clear that lifted-foldings deform nets with planar curvature lines within this class. Therefore, there are distinguished planar circular nets of $(M)$-type that give rise to all such nets in 3-space (see Figure~\ref{fig_planar}, \emph{left}):
\begin{cor}
Let $\q:=\q_0 \perp \p$ be a space form vector that determines a 2-dimensional Euclidean space form. Further, suppose that $G$ is a planar circular net of $(M)$-type given by a sequence of $(\m^{(ij)})$-type Ribaucour transforms. Lifted-foldings of the planar circular net $G$ create circular nets with a family of planar parameter lines if and only if all vectors $\m^{(ij)}$ satisfy $\inner{\m^{(ij)}, \q}=0$.
\end{cor} 
\begin{proof}
Note that during lifted-foldings the M-inversions applied to the discrete curves of the planar net~$G$ are Euclidean motions in 3-space if and only if $\m^{(ij)} \perp \q$. Together with Theorem~\ref{thm_mflexible} this leads to the assertion.
\end{proof}
%
$\bullet$ \textbf{Discrete Joachimsthal surfaces.} A smooth surface with a family of curvature lines lying in planes that intersect in a common line are classically called Joachimsthal surfaces. As a consequence, the other family of curvature lines lie on spheres that have their centers on this line. Moreover, the planar curvature lines are then orthogonal trajectories to those spheres (see for example~\cite{eisenhart}). 
Suppose that all planar curvature lines are rotated around this line to lie in one plane (which corresponds to a degenerate lifted-folding). Then the curves are orthogonal trajectories to a 1-parameter family of circles with centers on this rotational axis, that is, the 2-dimensional orthogonal coordinate system given by the planar rotated curvature lines together with the family of circles provides a 2-dimensional cyclic system \cite{salkowski}.   

Discrete analogs of the latter have been introduced in~\cite{discrete_cyclic}: given a family of circles $(c_i)_i \in \Light \subset \mathbb{R}^{3,2}$ with centers on a line $l \in \Light$, $\mathfrak{l} \perp \q_0$. Any discrete sampling on an initial circle propagates via the M-inversions with respect to the linear complexes
\begin{equation*}
\mathfrak{a}_{ij}:=\mathfrak{c}_i \inner{\mathfrak{c}_j, \p} - \mathfrak{c}_j\inner{\mathfrak{c}_i, \p} 
\end{equation*} 
to a planar circular net of $(\mathfrak{l})$-type. Lifted-folding then provides discrete Joachimsthal surfaces in the sense that the family of planar curvature lines is composed of discrete orthogonal trajectories to a family of spheres with center on the line $l$ (see~Figure~\ref{fig_joachimsthal}). 
\begin{figure}
\includegraphics[scale=0.25,angle=90,origin=c]{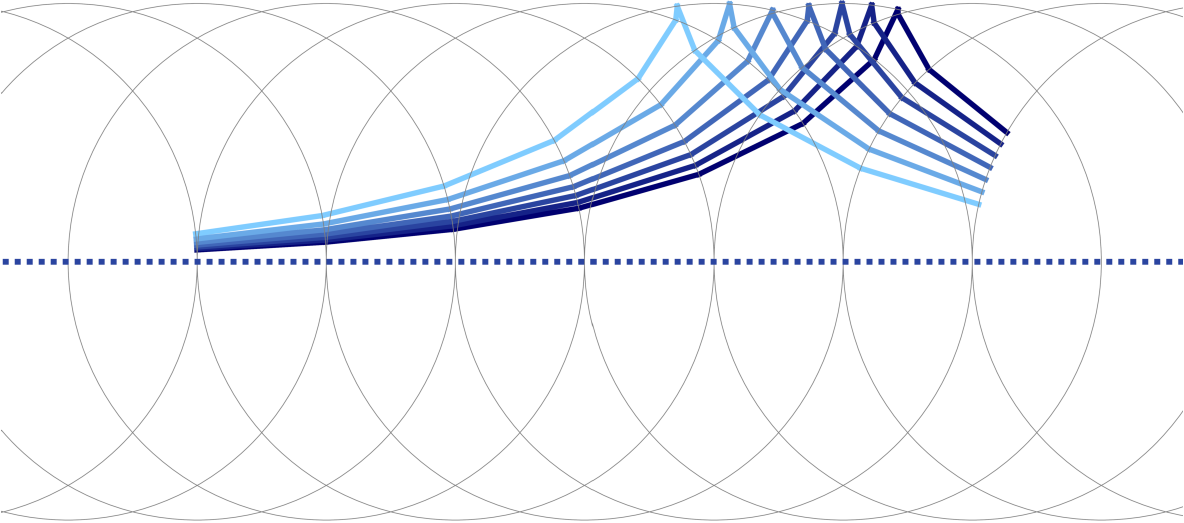}
\hspace*{0.8cm}\includegraphics[scale=0.3]{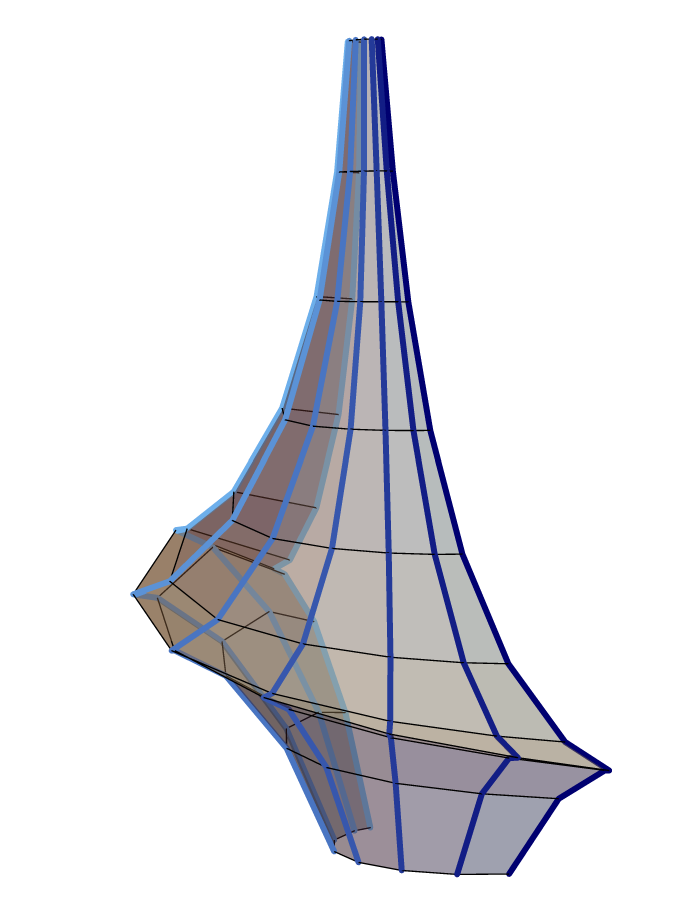}
\hspace*{0.8cm}\includegraphics[scale=0.4]{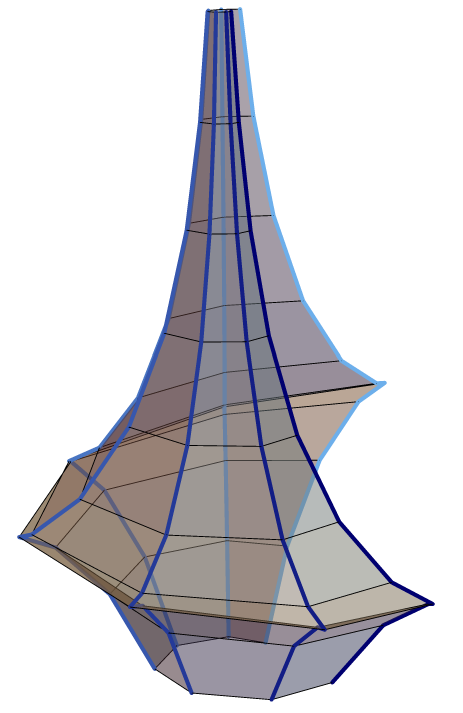}
\caption{Discrete Joachimsthal surfaces generated via lifted-folding from discrete orthogonal trajectories of a 2-dimensional discrete cyclic system~\cite{discrete_cyclic}.}\label{fig_joachimsthal}
\end{figure}
%
%
\section{Discrete isothermic surfaces with a family of spherical curvature lines}\label{sect_iso}
Smooth isothermic surfaces are characterized by the existence of conformal curvature line coordinates.
The following integrable discrete analogs to isothermic surfaces are well-studied: while discrete curvature line coordinates are modelled via circular nets, discrete conformality is mimicked by a constant cross-ratio for all quadrilaterals \cite{BobenkoPinkalliso}. A slight generalization, including admissible curvature line preserving reparametrizations, requires the existence of a consistent edge-labelling $\varepsilon:  \mathcal{E}^2 \to \mathbb{R}$ such that the cross-ratio of any quadrilateral factorizes: 
\begin{equation*}
\text{cr}(f_i, f_j, f_k, f_l)=\frac{\varepsilon_{ij}}{\varepsilon_{kl}}.
\end{equation*}

In this section we characterize discrete isothermic surfaces with a family of planar or spherical curvature lines. We restrict our study to  edge-labellings that are constant in the direction of the spherical coordinate lines. Geometrically, this means that the cross-ratios of the faces bounded by two adjacent spherical curvature lines are constant. 

As we will see below, this additional assumption on the edge-labels interplays well with the concept of lifted-folding. It will simplify our arguments and will place special emphasis on the geometric ideas. 

Thus, we use the following adapted definitions: 

\begin{defi}
Two discrete curves related by a Ribaucour transformation are called a \emph{discrete Darboux pair} if the cross-ratios of all circular faces are constant. A sequence of Darboux transformed curves in a fixed plane is said to be a \emph{discrete holomorphic map}.

A circular net with a family of spherical curvature lines is \emph{isothermic} if the spherical parameter lines are given by a sequence of discrete Darboux transforms. 
\end{defi}
%
%
\subsection{Lifted-foldings of isothermic nets}\label{subsect_iso} In this subsection we investigate the behaviour of discrete isothermic surfaces with a family of spherical curvature lines under the deformations induced by lifted-foldings as introduced in Theorem~\ref{thm_mflexible}.

\bigskipp The cross-ratio of a circular quadrilateral with vertices $f_1,f_2,f_3, f_4$ can be explicitly expressed in terms of two surface points and the corresponding evolution inversion $\sigma_r$ \cite{blaschke, r_congr}:   
\begin{equation}\label{form_cross}
\text{cr}(f_1, f_2,f_3, f_4 )=\text{cr}(f_1, f_2,\sigma_r(f_2), \sigma_r(f_1) )= \frac{\inner{\mathfrak{f}_1,\mathfrak{f}_2}\inner{\mathfrak{r},\mathfrak{r}}}{2 \inner{\mathfrak{f}_1,\mathfrak{r}}\inner{\mathfrak{f}_2,\mathfrak{r}}}.
\end{equation}
%
%

\bigskipp As we have seen in the previous section, lifted-folding of a circular net with spherical curvature lines yields new circular nets that have again curvature lines on spheres. Contrary to circularity, the cross-ratios of the circular quadrilaterals are changed during lifted-foldings. However, rather surprisingly, Darboux pairs of spherical curvature lines are mapped to new Darboux pairs under lifted-folding:  
\begin{prop}
If the cross-ratios of the quadrilaterals along an $(\m_1,\m_2)$-type Ribaucour pair are constant, then the cross-ratios after a lifted-folding are again constant. 
\end{prop}
\begin{proof}
As we have seen in the proof of Proposition~\ref{prop_flexible_pair}, the R-evolution map~$\sigma_r$ is preserved under lifted-folding and admissible folding inversions $\sigma_n$ commute with it: $\mathfrak{n} \perp \mathfrak{r}_{ij}$ for any $(ij) \in \mathcal{E}^2$. Therefore, the cross-ratios of an $(\m_1,\m_2)$-type Ribaucour pair $(f,g)$ change under an admissible lifted-folding via $\sigma_n$ in the following way:
\begin{equation*}
\begin{aligned}
\text{cr}(\mathfrak{f}_j, \mathfrak{g}_j,\sigma_{r_{ij}}(\mathfrak{g}_j), \sigma_{r_{ij}}(\mathfrak{f}_j)) &= \frac{\inner{\mathfrak{f}_j, \mathfrak{g}_j}\inner{\mathfrak{r}_{ij},\mathfrak{r}_{ij}}}{2 \inner{\mathfrak{f}_j,\mathfrak{r}_{ij} }\inner{\mathfrak{g}_j,\mathfrak{r}_{ij}}} \\&\downarrow \
%
\\\text{cr}(\mathfrak{f}_j, \sigma_n(\mathfrak{g}_j),\sigma_{r_{ij}}(\sigma_n(\mathfrak{g}_j)), \sigma_{r_{ij}}(\mathfrak{f}_j)) &= \frac{\inner{\mathfrak{f}_j,\sigma_n(\mathfrak{g}_j)}\inner{\mathfrak{r}_{ij},\mathfrak{r}_{ij}}}{2 \inner{\mathfrak{f}_j,\mathfrak{r}_{ij} }\inner{\mathfrak{g}_j,\mathfrak{r}_{ij}}}.
\end{aligned}
\end{equation*}

Moreover, from formula (\ref{form_cross}), we learn that two adjacent quadrilaterals of the Ribaucour pair~$(f,g)$ have the same cross-ratio,
\begin{equation*}
\text{cr}(\mathfrak{f}_j, \mathfrak{g}_j,\sigma_{r_{ij}}(\mathfrak{g}_j), \sigma_{r_{ij}}(\mathfrak{g}_i))= \text{cr}(\mathfrak{f}_j, \mathfrak{g}_j,\sigma_{r_{jk}}(\mathfrak{g}_j), \sigma_{r_{jk}}(\mathfrak{g}_i)),
\end{equation*}
if and only if
\begin{equation*}
\frac{\inner{\mathfrak{r}_{ij},\mathfrak{r}_{ij}}}{\inner{\mathfrak{f}_j,\mathfrak{r}_{ij}}\inner{\mathfrak{g}_j,\mathfrak{r}_{ij}}} = \frac{\inner{\mathfrak{r}_{jk},\mathfrak{r}_{jk}}}{\inner{\mathfrak{f}_j,\mathfrak{r}_{jk}}\inner{\mathfrak{g}_j,\mathfrak{r}_{jk}}}.
\end{equation*}
Since the last equation is invariant under lifted-foldings, the claim follows.
\end{proof}
\noindent As a consequence, together with Corollary~\ref{cor_unfolding}, the concept of lifted-folding reveals that describing isothermic nets with a family of spherical parameter lines is equivalent to finding all planar isothermic nets of $(M)$-type, i.\,e.\,discrete holomorphic nets of $(M)$-type.

\bigskipp In summary:
\begin{thm}\label{thm_iso_folding}
Lifted-foldings preserve the class of isothermic circular nets with a family of spherical curvature lines. In particular, all such nets can be obtained from holomorphic maps of $(M)$-type via lifted-foldings.
\end{thm}
%
%
\subsection{Darboux transforms of $(\m)$-type}\label{subsect_mDarboux} Since isothermic nets with a family of spherical lines of curvature can be generated from a discrete holomorphic map of $(M)$-type, in this subsection we aim to determine all those special holomorphic nets. 

Therefore, we again restrict to planar discrete curves, where $(\m)$-type Darboux pairs admit the following simple characterization:
\begin{prop}\label{prop_constant_cr}
Let $( f,g)$ be a planar discrete Ribaucour pair of $(\m)$-type. The cross-ratios of all circular faces are constant if and only if along the curve we have
\begin{equation}\label{equ_iso_cond}
\mathfrak{g}_j \perp \f_i \inner{\f_k, \m} - \f_k \inner{f_i, \m}   \ \ \ \Big( \Leftrightarrow \ \mathfrak{f}_j \perp \g_i \inner{\g_k, \m} - \g_k \inner{\g_i, \m} \Big). 
\end{equation}
\end{prop}
\begin{proof}
Without loss of generality, we choose homogeneous coordinates for $f$ and $g$ such that 
\begin{equation*}
\inner{\mathfrak{f}_i, \mathfrak{m}}=\inner{\mathfrak{g}_i, \mathfrak{m}}=1.
\end{equation*}
Then, by formula (\ref{form_cross}), the cross-ratios of two adjacent quadrilaterals are equal if and only if 
\begin{equation*}
\frac{\inner{\mathfrak{r}_{ij},\mathfrak{r}_{ij}}}{\inner{\mathfrak{f}_j,\mathfrak{r}_{ij}}\inner{\mathfrak{g}_j,\mathfrak{r}_{ij}}} = \frac{\inner{\mathfrak{r}_{jk},\mathfrak{r}_{jk}}}{\inner{\mathfrak{f}_j,\mathfrak{r}_{jk}}\inner{\mathfrak{g}_j,\mathfrak{r}_{jk}}}.
\end{equation*}
Due to (\ref{equ_evolution_m}), this equation amounts to $\inner{\mathfrak{g}_j, \mathfrak{f}_j - \mathfrak{f}_i}=\inner{\mathfrak{g}_j, \mathfrak{f}_j-\mathfrak{f}_k}$ and the claim follows.
\end{proof}
As we now show, this simple condition (\ref{equ_iso_cond}) that characterizes planar Darboux pairs of $(\m)$-type gives rise to special circle congruences associated to the discrete curves. This observation will later give insights into the geometry of these curves and will provide  the crucial link to the notion of discrete constrained elastic curves in space forms. 
\\\\For any discrete curve $f: \V \to \Light$, we consider the map
\begin{equation*}
\PP^\pm: \V \to \{ \text{circle pencils} \}, \ \ i \mapsto \spann{\f_{i-1}, \f_{i+1}}^\perp \cap \mathcal{L},
\end{equation*}
which describes associated circle pencils along the curve~$f$. Geometrically, the circle pencil~$\PP_i^\pm$ consists of all (oriented) circles that pass through the two curve points $f_{i-1}$ and $f_{i+1}$ (see Figure~\ref{fig_pencils_congr}, left).
%
\begin{prop}\label{prop_circle_congr_iso}
Let $f$ be a discrete curve in a planar Darboux pair of $(\m)$-type. Then there exists a circle congruence~$c \in \mathcal{P}^\pm$ that is evolved by the R-evolution map~$\sigma_r$, that is, for fixed homogeneous coordinates $\mathfrak{c} \in c$ with $\inner{\mathfrak{c},\mathfrak{p}}=-1$, we have that 
\begin{equation*}
\sigma_{r_{ij}}(\mathfrak{c_{i}})=\mathfrak{c}_j.
\end{equation*}
This circle congruence~$c \in \mathcal{P}^\pm$ then  lies in a fixed linear complex determined by~$\overline{\m} \in \spann{\m, \p}$.
\end{prop}
\begin{proof}
Suppose that $f$ and $g$ are related by a Darboux transformation of $(\m)$-type. By Proposition~\ref{prop_constant_cr}, we then know that
\begin{equation*}
\mathfrak{a}_{ik}:=\f_i \inner{\f_k, \m} - \f_k \inner{f_i, \m} \perp \spann{\mathfrak{g}_j, \mathfrak{m}, \mathfrak{p}}.
\end{equation*}
Hence, the inversion with respect to $a_{ik}$ fixes any circle in $\spann{\mathfrak{g}_j, \mathfrak{m}, \mathfrak{p}}$ and interchanges the curve point~$f_i$ with~$f_k$. Thus, a circle $c_j \in \spann{\mathfrak{g}_j, \mathfrak{m}, \mathfrak{p}}$ that passes through $f_i$, also has to contain the point $f_k$ and therefore lies in~$\mathcal{P}^\pm$. Note that this circle $c_j$ exists and is unique up to orientation. 

Since the evolution map $\sigma_r$ evolves the subspaces $\spann{\mathfrak{g}_j, \mathfrak{m}, \mathfrak{p}}$ and the curve points of $f$ simultaneously, the circles $c$ with appropriately chosen orientation provide the sought-after circle congruence.

Furthermore, since $\m$ and $\p$ are fixed points for all inversions in the evolution map, the quantity $\frac{\inner{\cc_i, \m}}{\inner{\cc_i, \p}}\equiv\xi \in \mathbb{R}$ is constant for all $i \in \V$. Therefore, we conclude that the circle congruence $c$ satisfies
\begin{equation*}
\inner{\cc_i, \m- \xi\p}=0 \ \text{for all } i \in \V
\end{equation*}
and hence lies in the linear complex determined by $(\m - \xi\p) \in \spann{\m, \p}$.
\end{proof} 
%
\noindent As we will see below, $(\m)$-type evolution maps that evolve a circle congruence which lies in a hyperbolic complex give rise to Darboux transformations. This motivates the following definition:
\begin{defi}
A map $\sigma: \E \to \{ \text{M-inversions} \}$ is called an \emph{$(\m)$-type Darboux evolution map} for the discrete curve~$f$ if it is an $(\m)$-type evolution map for $f$ which evolves a circle congruence $c \in \mathcal{P}^\pm$, that is, $\sigma_{ij}(\cc_i)=\cc_j$ for homogeneous coordinates satisfying $\inner{\cc_i, \p}=\inner{\cc_j, \p}=-1$. 
\end{defi}
%
%
\noindent The following lemma will give a hint on how to construct discrete curves that admit $(\m)$-type Darboux evolution maps.
%
%
\begin{lemdef}\label{lem_evo_tangential}
Suppose that $f$ is a planar discrete curve and $c \in \mathcal{P}^\pm$ a circle congruence in a linear complex determined by $\overline{\m}$. The circle congruence~$c$ induces an $(\overline{\m})$-type Darboux evolution map $\sigma_r$ via
\begin{equation*}
\rr_{ij}:=\cc_{i} \inner{\cc_j, \p} - \cc_j \inner{\cc_i, \p}
\end{equation*}
if and only if there exists a \emph{tangential circle congruence~$\ttt: \mathcal{E} \to \Light$} such that $\ttt_{ij} \perp \f_i, \f_j, \cc_i, \cc_j$.
\end{lemdef}
\begin{proof}
The inversions~$\sigma_r$ provide an evolution map for~$f$ if and only if, for any edge $(ij)$, the M-inversion~$\sigma_{r_{ij}}$ interchanges the curve points $f_i$ and $f_j$. This is the case if and only if for each edge~$(ij)$ the contact elements $\spann{\f_i, \cc_j}$ and $\spann{\f_j, \cc_i}$ lie in a 3-dimensional subspace. That is, there exists a common circle in both contact elements (which is obviously equivalent to the existence of a tangential circle). 
\end{proof}
We remark that for any discrete curve there exists a 1-parameter family of circle congruences that admit a tangential circle congruence. This is because the choice of one circle $c_0 \in \mathcal{P}_0^\pm$ at an initial vertex already uniquely determines both, $c \in \mathcal{P}^\pm$ and the tangential circle  congruence. However, as we will see below, the  additional requirement that those circle congruences $c \in \mathcal{P}^\pm$ lie in fixed linear complexes restricts the geometry of the curve.

Moreover, those observations also tell us that a discrete curve can admit at most a 1-parameter family of $(\m^\lambda)$-type Darboux evolution maps.
%
%
\\\\\textbf{Construction 1.} Planar discrete curves $f$ that admit at least \emph{one} $(\m)$-type Darboux evolution map can be iteratively constructed from three initial curve points $(f_0, f_1, f_2) \in \Light$ and one initial circle $\cc_1 \perp \m$ that passes through $f_0$ and $f_2$: 
\begin{itemize}
\item let $\ttt_{01}:=\f_0 \inner{\cc_1, \f_1} - \cc_1\inner{\f_0, \f_1}$ and $\ttt_{12}:=\f_2 \inner{\cc_1, \f_1} - \cc_1\inner{\f_2, \f_1}$ be the two tangential circles to $c_1$ that go through the curve point $\f_1$;
\item define the next circle $\cc_2:= \f_1 \inner{\ttt_{12}, \m} - \ttt_{12} \inner{\f_1, \m}$;
\item choose an arbitrary point $\f_4 \perp \cc_2$ that lies on the circle $c_2$ and is distinct from $f_1$
\end{itemize} 

Iteration then generates a discrete curve with an $(\m)$-type Darboux evolution map (c.\,f.\,Lemma~\ref{lem_evo_tangential}). As this construction reveals, this condition is a rather weak property for discrete curves and there is a 1-parameter choice for each additional curve point.
%
\\\\\textbf{Construction 2.} However, by requiring the existence of \emph{two} Darboux evolution maps of $(\m^1)$- and $(\m^2)$-type with $\spann{\m^1}\neq \spann{\m^2}$, the construction becomes unique: assume we prescribe two linearly independent vectors $\m^1, \m^2 \in \mathbb{R}^{3,2}$, three initial curve points $(f_0, f_1, f_2) \in \Light$ and two circles $\cc_1^1 \perp \m^1$ and $\cc_1^2 \perp \m^2$ that both go through the curve points $f_0$ and $f_2$. Then the following procedure uniquely determines the sought-after discrete curve (see Figure~\ref{fig_circle_dyn}):
\begin{itemize}
\item let $\ttt_{01}^1:=\f_0 \inner{\cc_1^1, \f_1} - \cc_1^1\inner{\f_0, \f_1}$ and $\ttt_{12}^1:=\f_2 \inner{\cc_1^1, \f_1} - \cc_1^1\inner{\f_2, \f_1}$ be the two tangential circles to $c_1^1$ that go through the curve point $\f_1$;
\item let $\ttt_{01}^2:=\f_0 \inner{\cc_1^2, \f_1} - \cc_1^2\inner{\f_0, \f_1}$ and $\ttt_{12}^2:=\f_2 \inner{\cc_1^2, \f_1} - \cc_1^2\inner{\f_2, \f_1}$ be the two tangential circles to $c_1^2$ that go through the curve point $\f_1$;
\item define the next circles $\cc_2^1:= \f_1 \inner{\ttt_{12}^1, \m^1} - \ttt_{12}^1 \inner{\f_1, \m^1}$ and $\cc_2^2:= \f_1 \inner{\ttt_{12}^2, \m^2} - \ttt_{12}^2 \inner{\f_1, \m^2}$;
\item the circles $c_2^1$ and $c_2^2$ intersect in the curve point $f_1$ and, generically, in a second point $f_3$, which provides the next curve point; 
\end{itemize}
%
%
%
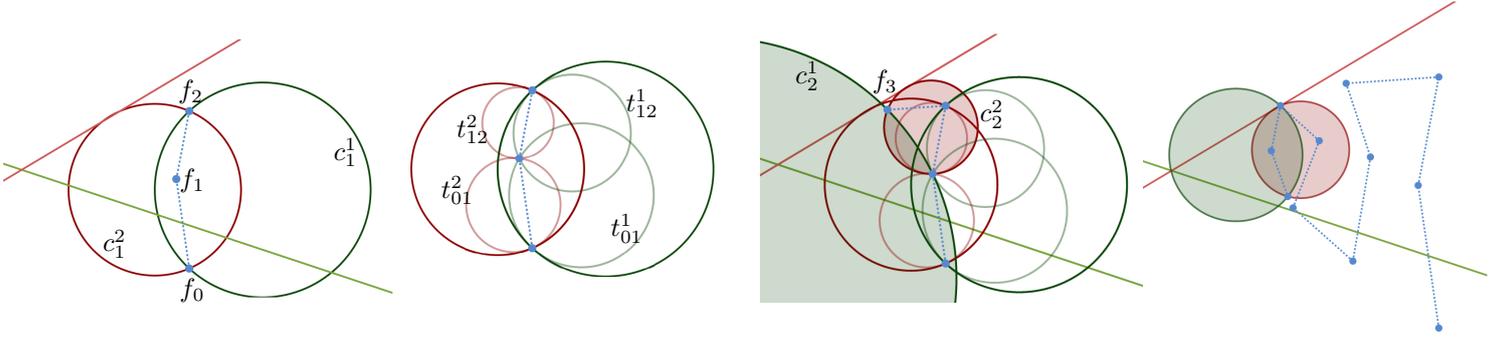
\begin{figure}
\hspace*{-2.7cm} \begin{overpic}[width=1.35\linewidth]{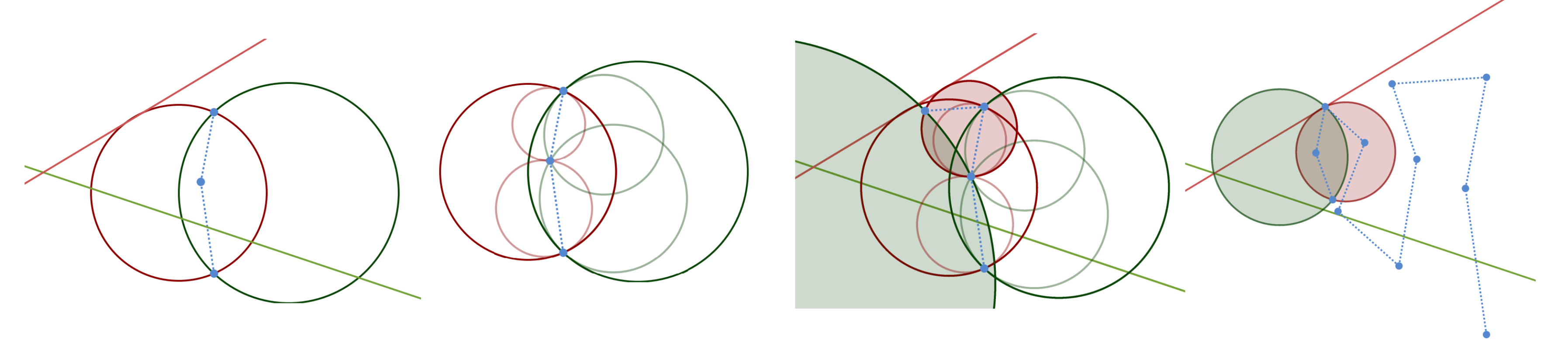}
    \put(12.9,16){$f_2$}
     \put(13,10.2){$f_1$}
     \put(23,12){$c^1_1$}
     \put(8,6){$c^2_1$}
      \put(13,3){$f_0$}
      \put(42,15.2){$t^1_{12}$}
     \put(41,7){$t^1_{01}$}
     \put(30,9.5){$t^2_{01}$}
     \put(31,13.5){$t^2_{12}$}
     \put(53,17){$c^1_2$}
     \put(65,14.5){$c^2_2$}
     \put(58,16.7){$f_3$}
     \end{overpic}
     \caption{Iterative construction of a discrete curve (blue) with two Darboux evolution maps of $(\m_1)$-type and $(\m_2)$-type (see Construction~2).}\label{fig_circle_dyn}
     \end{figure}
%
%
We remark that this construction may have singularities. This is the case when the constructed circles $\cc_2^1$ and $\cc_2^2$ are in oriented contact and do not determine the next curve point. Examples without singularities are provided by discrete constrained elastic curves in space forms as discussed in Subsection~\ref{subsect_holo_elastic}.
%
%
\\\\The latter construction already describes the entire class of discrete curves with a 1-parameter family of $(\m^\lambda)$-type Darboux evolution maps:
%
%
\begin{prop}\label{prop_combination_s}
Let $f$ be a discrete planar curve admitting two circle congruences $c^1 \in \mathcal{P}^\pm$ and $c^2 \in \mathcal{P}^\pm$ that induce Darboux evolution maps of $(\m^1)$-type and $(\m^2)$-type such that $\spann{\m^1} \neq \spann{\m^2}$. Then $c^1$ and $c^2$ intersect at a constant angle along~$f$, i.\,e.\,
\begin{equation*}
\frac{\inner{\cc_i, \cc_j}}{\inner{\cc_i, \p}\inner{\cc_j,\p}} \equiv \xi \in \mathbb{R} \ \text{for all } i \in \V,
\end{equation*}
and each circle congruence $c^\lambda \in \mathcal{P}^\pm$,
\begin{equation}\label{equ_further_circle_congr}
\mathfrak{c}^\lambda:=\lambda \mathfrak{c}^{1} + \tilde{\lambda}\mathfrak{c}^{2} + \mathfrak{p}, \ \text{where } \tilde{\lambda}:=\frac{1+2\lambda}{2 \xi \lambda -2} \ \text{and } \lambda \in \mathbb{R}\setminus \{\frac{1}{\xi}\},
\end{equation}
also induces an $(\m^\lambda)$-type Darboux evolution map with $\m^\lambda \in \spann{\m^1, \m^2}$.
\end{prop}
\begin{proof}
We fix homogeneous coordinates $\cc^1 \in c^1$ and $\cc^2 \in c^2$ such that $\inner{\cc^1, \p}=\inner{\cc^2, \p}=-1$. Since, by using the evolution maps, we obtain 
\begin{equation}\label{equ_inpoints}
\begin{aligned}
\cc_j^{1}= \cc_i^1 +\frac{\inner{\cc_i^1,\f_i}}{\inner{\f_i,\f_j}\inner{\f_i,\m^1}} \Big( \f_i\inner{\f_j, \m^1} - \f_j\inner{\f_i, \m^1} \Big),
\\[5pt]\cc_j^{2}= \cc_i^{2} +\frac{\inner{\cc_i^{2},\f_i}}{\inner{\f_i,\f_j}\inner{\f_i,\m^2}} \Big( \f_i\inner{\f_j, \m^2} - \f_j\inner{\f_i, \m^2} \Big),
\end{aligned}
\end{equation}
and a straightforward calculation shows that $\inner{\cc_i^{1}, \cc_i^{2}}\equiv \xi \in \mathbb{R}$ for all $i \in \V$. Therefore the circle congruences $ c^{1}$ and $c^{2}$ intersect at a constant angle along the discrete curve~$f$ (see (\ref{equ_intersection_angle})). 
 
\bigskipp Appropriate linear combinations defined by (\ref{equ_further_circle_congr}), then provide further circle congruences $c^\lambda \in \mathcal{P}^\pm$. Any circle congruence $c^\lambda$ then intersects $c^{1}$ and $c^{2}$ at a constant angle along the curve~$g$. 

Suppose that $t^1$ and $t^2$ denote the tangential circle congruences of $c^1$ and $c^2$, respectively. For homogeneous coordinates $\inner{\ttt_i^1,\p}=\inner{\ttt_i^2,\p}=-1$, the congruence
\begin{equation}\label{equ_tang}
\mathfrak{t}^\lambda:=\lambda \mathfrak{t}^{1} + \tilde{\lambda}\mathfrak{t}^{2} + \mathfrak{p}
\end{equation}
then yields the tangential circle congruence of $c^\lambda$. Therefore, the circle congruence~$c^\lambda$ indeed induces an evolution map for~$f$ via
\begin{equation*}
\rr_{ij}^\lambda:= \cc_i^\lambda - \cc_j^\lambda = \lambda (\cc_i^1 - \cc_j^1) + \tilde{\lambda} (\cc_i^2 - \cc_j^2) \in \spann{\f_i, \f_j}.
\end{equation*}

It remains to prove that for fixed $\lambda \in \mathbb{R}\setminus \{\frac{1}{\xi}\}$ this evolution map is of $(\m^\lambda)$-type. Thus, we define the edge-function $\eta:\mathcal{E}\to \mathbb{R}$ via
\begin{equation}\label{equ_comp_pencil}
0=\inner{\rr_{ij}, \m^1+\eta_{ij}\m^2}.
\end{equation}
By using equations (\ref{equ_inpoints}), we obtain
\begin{equation*}
\eta_{0+}=-\frac{\tilde{\lambda} \inner{\cc_0^2 - \cc_+^2, \m^1}}{\lambda \inner{\cc_0^1 - \cc_+^1, \m^2}} = \frac{\tilde{\lambda} \inner{\cc_0^2, \f_0} \inner{\f_0, \m^1}}{\lambda \inner{\f_0, \m^2} \inner{\cc_0^1, \f_0} }.
\end{equation*}
Since the quantities $\frac{\inner{\cc_0^1, \f_0}}{\inner{\cc_0^1, \p}\inner{\f_0, \m^1}}$ and $\frac{\inner{\cc_0^2, \f_0}}{\inner{\cc_0^2, \p}\inner{\f_0, \m^2}}$ are both constant along the curve~$g$, we conclude that $\eta$ is also constant along the curve. 

Hence, the circle congruence $c^\lambda$ induces an $(\m^1 + \eta \m^2)$-type Darboux evolution map.
\end{proof}
%
Finally, we prove that, for a suitable choice of initial points, $(\m)$-type Darboux evolution maps indeed yield Darboux transformed curves. The only condition is that the linear complex containing the evolved circle congruence is hyperbolic or parabolic. If this is the case, then the discrete Darboux transforms can be even given by explicit parametrizations in the light cone:
%
%
\begin{thm}\label{thm_darboux_trafos}
Let $f$ be a planar discrete curve that admits an $(\m)$-type Darboux evolution map $\sigma_r$, i.\,e.\,there exists an evolved  circle congruence $c \in \mathcal{P}^\pm$ that lies in the linear complex $\overline{\m}^\perp$. Further suppose that the homogeneous coordinates $\cc_i \in c_i$ satisfy $\inner{\cc_i,\p}=-1$. 

In the following three cases, there are initial points such that this evolution map gives rise to real Darboux transforms:
\begin{itemize}
\item[(i)] If $\inner{\overline{\m}, \overline{\m}}< 0$ and $\inner{\overline{\m}, \p}^2+\inner{\overline{\m}, \overline{\m}}\neq 0$, then there are two $(\m)$-type Darboux transforms~$g^\pm$ of~$f$ explicitly described by
\begin{equation}\label{equ_darboux_expl}
\g_i^\pm = \cc_i + \mu^\pm \overline{\m} + (\mu^\pm\inner{\overline{\m}, \p} -1)\p, \ \ \text{where } \  \mu^\pm:=\frac{\inner{\overline{\m}, \p} \pm \sqrt{-\inner{\overline{\m}, \overline{\m}}}}{\inner{\overline{\m}, \p}^2+\inner{\overline{\m}, \overline{\m}}}.
\end{equation}
\item[(ii)] If $\inner{\overline{\m}, \p}^2+\inner{\overline{\m}, \overline{\m}}=0$ and $\inner{\overline{\m}, \p} \neq 0$, then there exists one Darboux transform $g$ given by
\begin{equation}\label{equ_darboux_expl2}
\g_i = \cc_i + \frac{1}{2\inner{\overline{\m},\p}}\overline{\m}-\frac{1}{2}\p.
\end{equation}
\item[(iii)] If $\inner{\overline{\m}, \overline{\m}}=0$ and $\inner{\overline{\m}, \p}\neq0$, then the curve points of the induced Darboux transform $g$ lie on the circle represented by $\overline{m} \in \Light$ and are given by
\begin{equation}\label{equ_darboux_expl3}
\g_i = \cc_i + \frac{1}{\inner{\overline{\m},\p}}\overline{\m}.
\end{equation}
\end{itemize}
\end{thm}
%
%
\begin{proof}
By Proposition~\ref{prop_circle_congr_iso}, there exists a constant $\xi \in \mathbb{R}$ such that $\cc_i \perp \m + \xi \p=:\overline{\m}$ for all  $i \in \V$. Straightforward computations then show that $\g^\pm$ defined by (\ref{equ_darboux_expl}),  (\ref{equ_darboux_expl2}) or  (\ref{equ_darboux_expl3}) indeed provide vectors in the projective light cone that represent points. Therefore, in all cases the formulas describe well-defined discrete curves.

Moreover, since $\p$ and $\overline{\m}$ are fixed points of all inversions in the evolution map and the circle congruence $c$ is evolved by it, the discrete curves~$g^\pm$ are related to $f$ via the $(\m)$-type Darboux evolution map~$\sigma_r$.
Finally, by Proposition~\ref{prop_constant_cr}, we conclude that $g^\pm$ are Darboux transforms of~$f$: the cross-ratios are constant because of
\begin{equation*}
\f_i\inner{\f_k, \m } - \f_k\inner{\f_i, \m} \perp \g_j \in \spann{\cc_j, \m, \p}.
\end{equation*}
\end{proof}
%
\noindent The explicit formulas (\ref{equ_darboux_expl}) and (\ref{equ_darboux_expl2}) reveal that in these cases some data of the discrete Darboux pairs are related by simple Lie inversions: suppose we are in case (i) of Theorem~\ref{thm_darboux_trafos} and define the vectors 
\begin{equation*}
\bb^\pm:=\mu^\pm \overline{\m} + (\mu^\pm\inner{\overline{\m}, \p} -1)\p.
\end{equation*}
Then we obtain that
\begin{equation*}
\sigma_{b^\pm}(\cc_i)=\cc_i - \frac{2 \inner{\cc_i, \bb^\pm}}{\inner{\bb^\pm, \bb^\pm}} \bb^\pm = \cc_i + \bb^\pm = \g_i^\pm.
\end{equation*}
Hence, the Lie inversion $\sigma_{b^\pm}$ maps the circle congruence $c \in \mathcal{P}^\pm$, which induces the Darboux transformation, to the new curve points~$g^\pm$. 

Moreover, since Lie transformations preserve oriented contact between circles, we conclude that the following data of the discrete curve~$f$ and the Darboux transform~$g^\pm$ are related by the Lie inversion~$\sigma_{b^\pm}$:
\\\begin{equation*}
\begin{aligned}
\spann{\overline{\m}} \ \ &\longleftrightarrow \ \ \spann{\p}
\\[4pt]\text{circle congruence } \cc \in \mathcal{P}^\pm \ \text{evolved along } f \ \ &\longleftrightarrow \ \ \text{curve points of the Darboux transform } \g^\pm=\sigma_{b^\pm}(\cc)
\\[4pt] \text{curve points of } f \ \ &\longleftrightarrow  \ \ \text{circle congruence} \ \hat{\cc}^\pm:=\sigma_{b^\pm}(\f) \in \hat{\mathcal{P}}^\pm \text{ evolved along } g^\pm
\\[4pt] \text{tangential circle congruence } t \text{ of } c \in \mathcal{P}^\pm \ \ &\longleftrightarrow  \ \ \text{tangential circle congruence} \ \hat{\ttt}^\pm:=\sigma_{b^\pm}(\ttt) \text{ of } \hat{c} \in \hat{\mathcal{P}}^\pm,
\end{aligned}
\end{equation*}
\ \\where $\hat{\mathcal{P}}^\pm$ denotes the circle pencils associated to the Darboux transforms~$g^\pm$.
\ \\\\Equivalent results hold for case~(ii) in Theorem~\ref{thm_darboux_trafos}, where the Lie inversion is determined by the vector  
\begin{equation*}
\bb:= \frac{1}{2\inner{\overline{\m},\p}}\overline{\m}-\frac{1}{2}\p.
\end{equation*}
%
%
\begin{figure}
\includegraphics[scale=0.5]{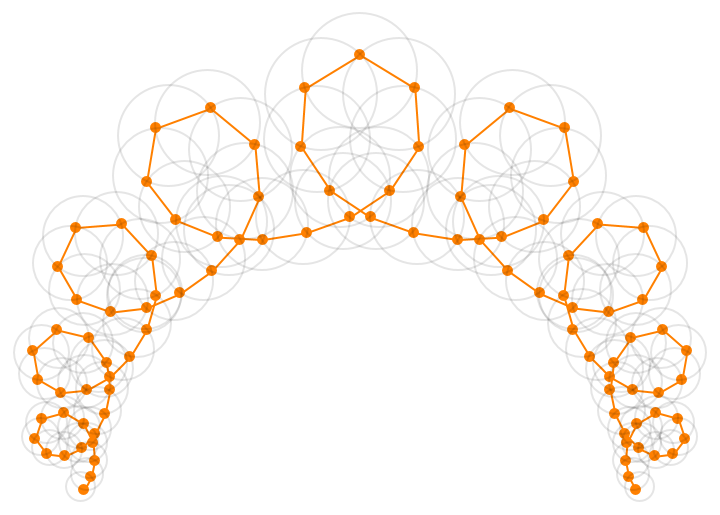}
\includegraphics[scale=0.5]{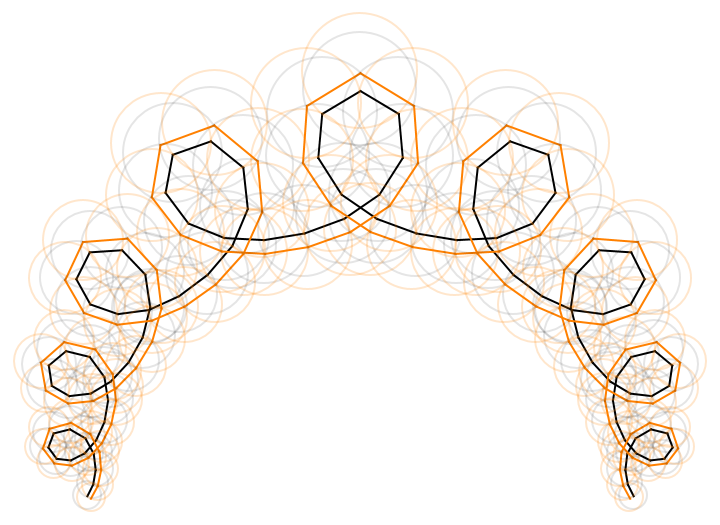}
\caption{A discrete constrained elastic curve in hyperbolic space with a circle congruence (gray) that induces an $(\m)$-type Darboux transform (gray). For the Darboux pair exists a Lie inversion that maps the orange curve points to the orange circle congruence, and the gray circles to the gray curve points.}
\end{figure}
%
%

\bigskipp Therefore, as a consequence of Proposition~\ref{prop_combination_s} and Theorem~\ref{thm_darboux_trafos}, the existence of two discrete Darboux transforms of $(\m^1)$-type and $(\m^2)$-type as obtained in Construction~2 entails a family of $(\m^\lambda)$-type Darboux transforms. 

\bigskipp Together with Proposition~\ref{prop_circle_congr_iso} and Lemma~\ref{lem_evo_tangential}, we conclude:
\begin{cor}\label{cor_constr2}
A spherical curvature line of a discrete isothermic net that has two spherical neighbouring parameter lines is, after a stereographic projection to the plane, a discrete curve obtained by Construction~2.
\end{cor}
\subsection{Discrete holomorphic maps of $(M)$-type and constrained elastic curves.}\label{subsect_holo_elastic} In \cite{chopembszew} it was proven that a family of spherical lines of curvature of a generic smooth Lie applicable surface are Lie transformed constrained elastic curves in space forms. In particular, this integrable surface class also includes isothermic surfaces.

In this subsection we examine whether similar results  also hold for discrete isothermic nets with a family of spherical parameter lines. Our investigations rely on a novel integrable discretization of constrained elastic curves in space forms, see \cite{discrete_elastic}. This class also includes planar discrete elastic curves in Euclidean space as discussed in \cite{BS_discreteelastic, fairing_elastica}.
%
\\\\We will again exploit the concept of lifted-folding and prove as main result that any discrete constrained elastic curve in a space form can be extended to an infinite sequence of discrete constrained elastic curves that gives a discrete holomorphic map of $(M)$-type. Thus, by applying suitable lifted-foldings, we may generate isothermic nets with a family of spherical curvature lines that are M\"obius transforms of discrete constrained elastica. 

Hence, as expected from the smooth theory, discrete constrained elastic curves provide special solutions to the circle dynamics described in the previous section (c.\,f.\,Construction~2).

\bigskipp We start by giving a brief introduction to  constrained elastic curves in space forms. The part on discrete curves is self-contained and focuses only on the facts necessary for our purposes. A detailed treatise on this discretization can be found in \cite{discrete_elastic}, where it is shown that those discrete curves satisfy a variety of properties known from their smooth counterparts: a discrete version of the smooth curvature equation and invariance under a (semi)-discrete mKdV-flow, to name just a few.

\bigskipp From now on we will consider curves in appropriate space form geometries, that is, the curves are embedded in a space of constant sectional curvature~$\kappa_\Q$. Depending on its sign, the space form is said to be \emph{Euclidean} $(\kappa_\Q=0)$, \emph{hyperbolic} $(\kappa_\Q<0)$ or \emph{spherical} $(\kappa_\Q>0)$. 

In the light cone model, those space forms are realised by the choice of a \emph{space form vector}. Thus, we still consider geometric configurations of points and circles in the plane as in the previous sections. The additional choice of a space form vector~$\q \in \mathbb{R}^{3,2}$ with $\q \perp \p$, then determines the space form into which we may project the configuration in order to obtain the embedded curves with their geometric interpretations in space forms (see the Appendix for some more details and references that cover this material in more depth).
 
\bigskipp A smooth arc-length parametrized curve in a 2-dimensional space form~$\mathcal{Q}$ with constant sectional curvature $\kappa_Q$ is constrained elastic \cite{pinkall_willmore, Heller_elastic, langer_singer} if its geodesic curvature $\kappa$ satisfies the stationary mKdV-equation 
\begin{equation*}
\kappa''+\frac{1}{2}\kappa^3+(\xi + \kappa_Q)\kappa + \chi =0
\end{equation*}
for some real parameters $\xi$ and $\chi$. Thus, up to reparametrization, those curves are invariant under the MkdV-flow. For closed curves this curvature equation tells us that the curve is a critical point of the energy functional $\int_\gamma \kappa^2 ds$ with fixed length and enclosed area.
%
%
\begin{figure}
\hspace*{-3cm}\begin{minipage}{4cm}
\includegraphics[scale=0.45]{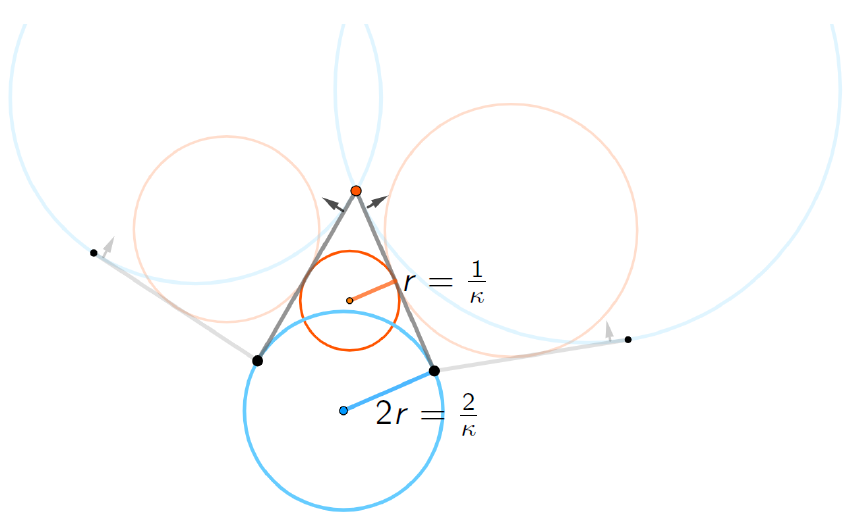}
\end{minipage}
\begin{minipage}{4cm}
\hspace*{1cm}\includegraphics[scale=0.3]{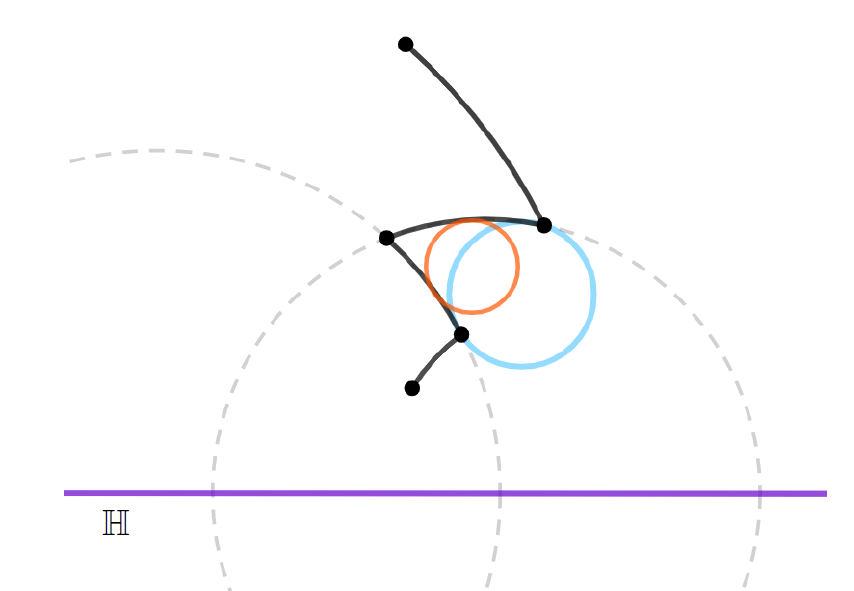}
\end{minipage}
\begin{minipage}{4cm}
\hspace*{1.8cm}\includegraphics[scale=0.4]{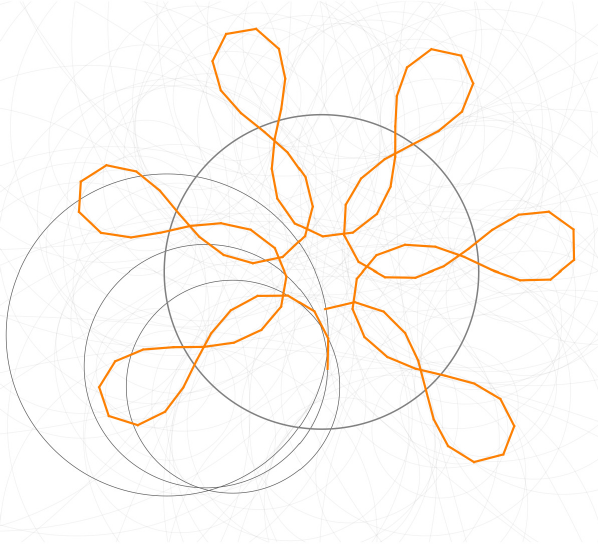}
\end{minipage}
\begin{minipage}{4cm}
\hspace*{2.1cm}\includegraphics[scale=0.3]{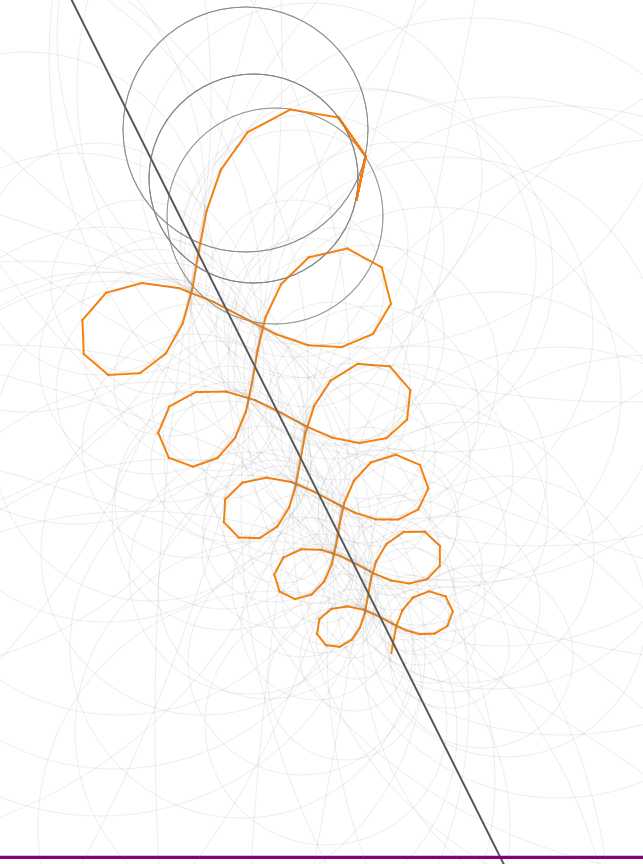}
\end{minipage}
\caption{\emph{Left.} Curvature circles (orange) in Euclidean and hyperbolic space. \emph{Right.} Discrete constrained elastic curves with their elastic circle congruence intersecting a directrix at a constant angle.} \label{fig_discrete_curves}
\end{figure}
%
%
\\\\In what follows, $\mathcal{Q}$ shall denote a 2-dimensional space form determined by the space form vector~$\mathfrak{q} \in \mathbb{R}^{3,2}$ and which then has constant sectional curvature $\kappa_\Q=-\inner{\q,\q}$. 

We say that a discrete curve~$f:\mathcal{V} \to \Light$ has \emph{constant arc-length in $\mathcal{Q}$} if the distance (measured in$~Q$) between any two consecutive points is constant. Thus, the distance-shortening geodesic segments connecting two adjacent curve points, called \emph{geodesic edges}, have the same length in~$\Q$ (see Figure~\ref{fig_discrete_curves}, left). 

We only consider \emph{regular} curves in $\Q$, that is, any three consecutive curve points are pairwise distinct and, additionally,
\begin{itemize}
\item if $\Q$ is a hyperbolic space, all curve points lie inside or outside the hyperbolic boundary;
\item if $\Q$ is a spherical space form, then two consecutive points are not antipodal.
\end{itemize} 
Note that those assumptions make the geodesic edge between two consecutive curve points unique. 

\bigskipp We equip the family of geodesic edges with orientations such that for any vertex $j \in \V$ there exists a circle $k_j \in \mathcal{P}_j^\pm$ which is in oriented contact with the two geodesic edges going through the curve point~$f_j$ (see blue circles in Figure~\ref{fig_discrete_curves}). Any discrete curve with constant arc-length admits two global choices for such an orientation.

We usually consider discrete curves in space forms together with a global orientation. The class of regular discrete oriented curves that have constant arc-length in $\mathcal{Q}$ is denoted by $\mathcal{F}^Q$.
\\\\Any oriented geodesic edge admits a unique extension to a complete geodesic in $Q$. In the light cone model, for the edge $(ij)$ this is described by a vector $t_{ij} \in \Light$ with $0=\inner{\mathfrak{t}_{ij},\q}
=\inner{\mathfrak{t}_{ij},\mathfrak{f}_{i}}=\inner{\mathfrak{t}_{ij},\mathfrak{f}_{j}}$. This circle congruence $t: \mathcal{E} \to \Light$ that extends the oriented geodesic edges is called the \emph{tangent line congruence in $\Q$}.

\bigskipp To introduce a discrete curvature notion for those curves in space forms, we generalize the Euclidean vertex-curvature circles as defined in \cite{hoffmann_lect}, which we recall here briefly (see also Figure~\ref{fig_discrete_curves}, left): suppose that $\Q$ is the usual Euclidean space form determined by $\q:=\q_0$ and consider $f \in \mathcal{F}_\Q$ with constant Euclidean arc-length. Then the Euclidean curvature at the curve point~$f_j$ is defined as $\kappa_j:= \frac{2}{r_j}$, where $r_j$ is the radius of the circle $k_j$ which is in oriented contact with the two oriented geodesic edges passing through the curve point~$f_j$. Thus, the curvature~$\kappa_j$ equals two times the geodesic curvature of the circle~$k_j$.

We naturally extend this curvature notion to space forms (see Figure~\ref{fig_discrete_curves}):
\begin{defi}
Suppose that $f \in \mathcal{F}_\Q$ is a discrete curve in the space form $\Q$. Its geodesic curvature~$\kappa_j$ at the curve point~$f_j$ is defined as two times the geodesic curvature of the circle~$k_j \in \mathcal{P}_j^\pm$ tangent to the two geodesic edges going through the curve point~$f_j$.
\end{defi}
%
%
\noindent A curve $f \in \mathcal{F}_\Q$ of constant arc-length may be characterized by the existence of a special circle congruence in $\mathcal{P}^\pm$:
\begin{lem}\label{lem_arclength}
Let $f$ be a planar discrete curve, then the following are equivalent:
\begin{itemize}
\item[(i)] $f \in \mathcal{F}_Q$, i.\,e.\,$f$ has constant arc-length in the space form $\Q$;
\item[(ii)] there exists $\q \perp \p \in \mathbb{R}^{3,2}$ such that $\f_i \inner{\f_k, \f_j} - \f_k \inner{\f_i, \f_j} \perp \q$; 
\item[(iii)] there exists $\q \perp \p \in \mathbb{R}^{3,2}$ such that the quantity
\begin{equation*}
\frac{\inner{\f_i, \f_j}}{\inner{\f_i, \q}\inner{\f_j, \q}} \equiv \chi \in \mathbb{R}
\end{equation*}
is constant along the curve;
\item[(iv)] there exists $\q \perp \p \in \mathbb{R}^{3,2}$ and a circle congruence $a \in \mathcal{P}^\pm$ such that $\aaa_j \in \spann{\f_j, \q, \p}$. 
\end{itemize}
\end{lem}
\begin{proof}
A discrete curve $f$ has constant arc-length in the space form $\Q$ if and only if, for any three consecutive vertices, there exists a space form motion that has $\f_j$ as fixed point and interchanges the curve points $\f_i$ and $\f_k$. Such an isometry is given by a reflection in the vector $\f_i \inner{\f_k, \f_j} - \f_k \inner{\f_i, \f_j}$. This shows that (i) is equivalent to (ii). Moreover, equivalence between (ii) and (iii) follows by straightforward computations. 

Thus, it remains to discuss property (iv). Suppose that $f \in \mathcal{F}_\mathcal{Q}$ and choose homogeneous coordinates such that $\inner{\f_i, \q}=-1$. Then 
\begin{equation*}
a:\E \to \Light, \ \aaa_i:= \frac{\alpha}{\chi} \f_i + \alpha \q + \p, \ \text{where } \ \alpha := \pm \sqrt{\frac{-\chi}{2-\inner{\q,\q}\chi}}
\end{equation*}
provides a well-defined circle congruence. Since $\inner{\aaa_j, \f_i}=\inner{\aaa_j, \f_k}=0$, we obtain that $a \in \mathcal{P}^\pm$.

Conversely, suppose that $\aaa_j \in \spann{\f_j, \q, \p}$ lies in $\mathcal{P}_j^\pm$. Then $\inner{\aaa_j, \frac{1}{\inner{\f_i, \q}}\f_i}=\inner{\aaa_j, \frac{1}{\inner{\f_k, \q}}\f_k}=0$ reveals property (iii).
\end{proof}

We remark that for a discrete curve with constant arc-length, $f \in \mathcal{F}_\mathcal{Q}$, there exists a canonical $(\q)$-type evolution map given by the linear complexes $\rr_{ij}:=\f_i\inner{\f_j, \q} - \f_j\inner{\f_i, \q}$. This evolution map~$\sigma_r$ evolves the circles $a \in \mathcal{P}^\pm$ described in Lemma~\ref{lem_arclength}(iv). Up to a global choice of orientation of the circles, this evolution provides a unique circle congruence; we call it the \emph{arc-length circle congruence}. 
%
%
\begin{figure}
\includegraphics[scale=0.3]{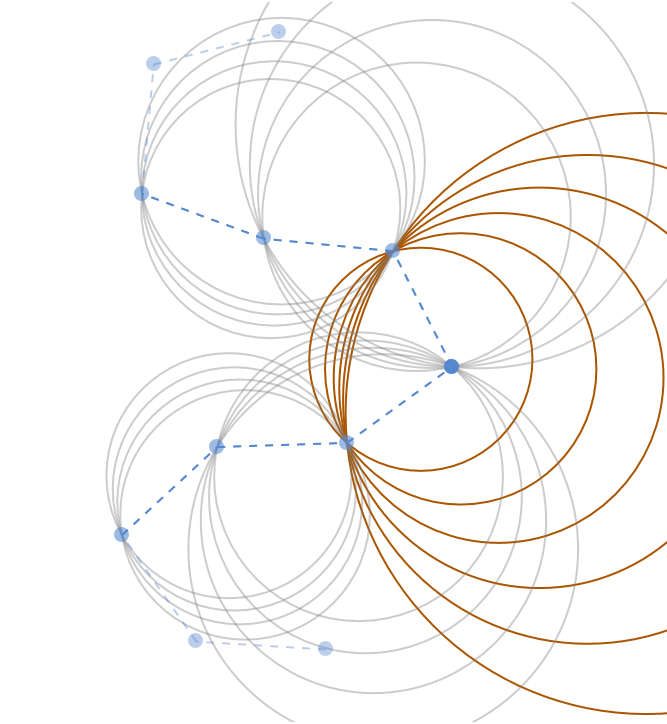}
\hspace*{0.8cm}\includegraphics[scale=0.3]{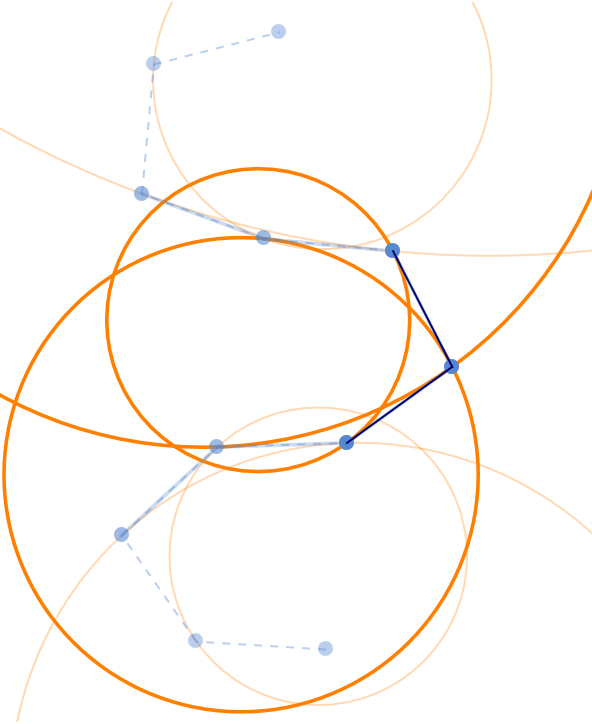}
\hspace*{-0.2cm}\includegraphics[scale=0.3]{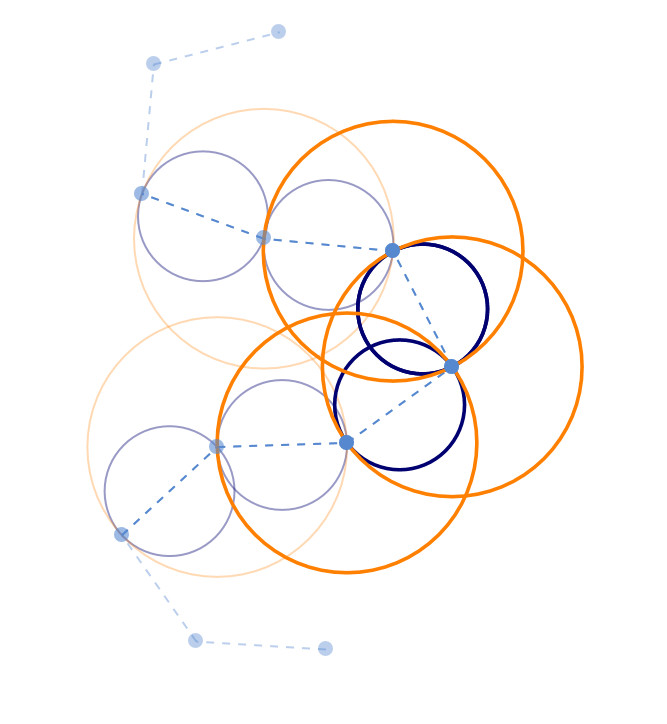}
\hspace*{-0.5cm}\includegraphics[scale=0.3]{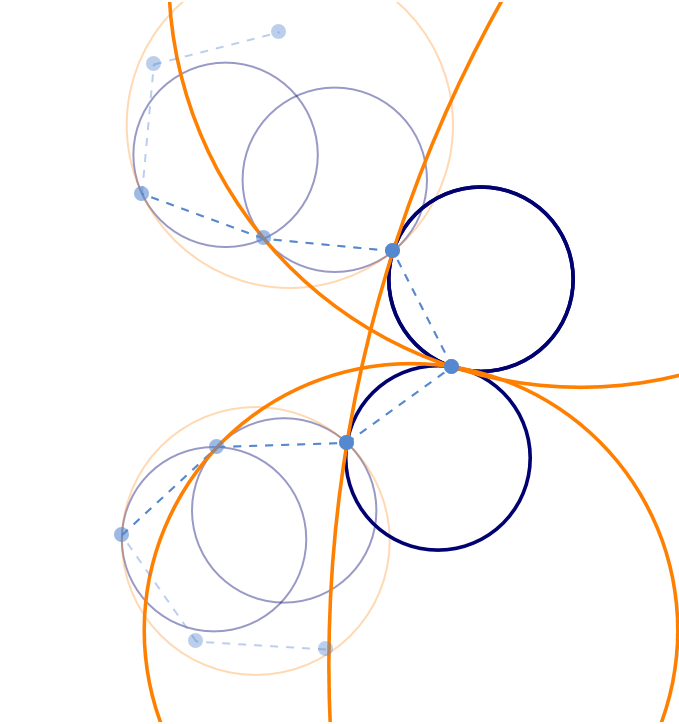}
\\[18pt]\centering\hspace*{-1cm}\includegraphics[scale=0.3]{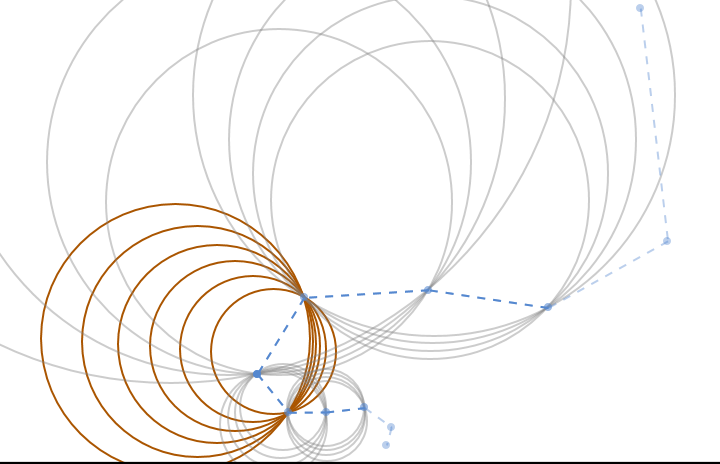}
\hspace*{0.2cm}\includegraphics[scale=0.3]{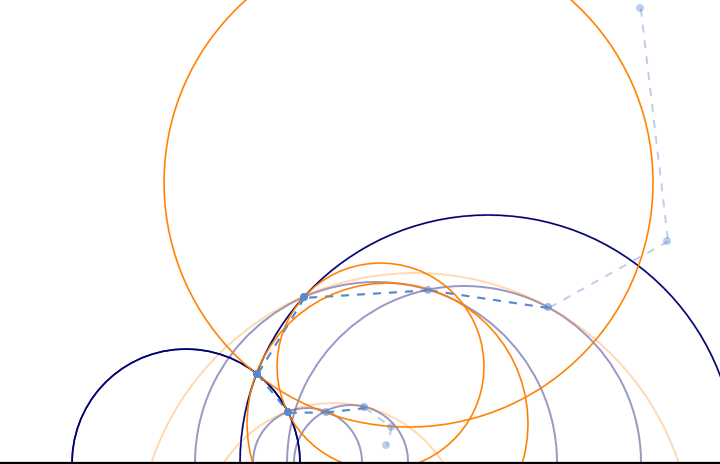}
\hspace*{0.1cm}\includegraphics[scale=0.3]{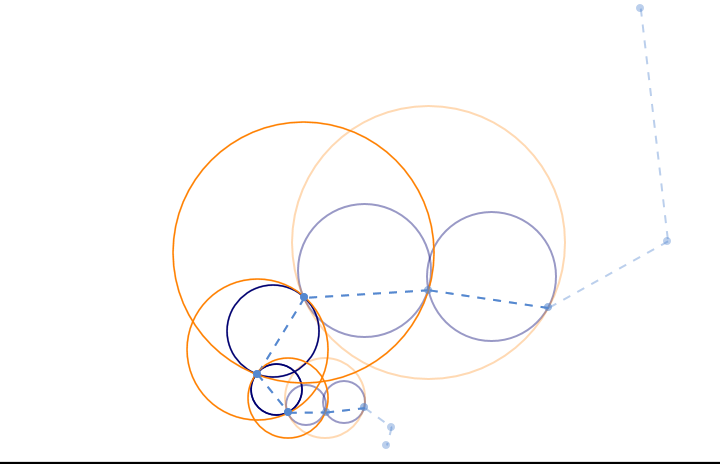}
\hspace*{0.1cm}\includegraphics[scale=0.3]{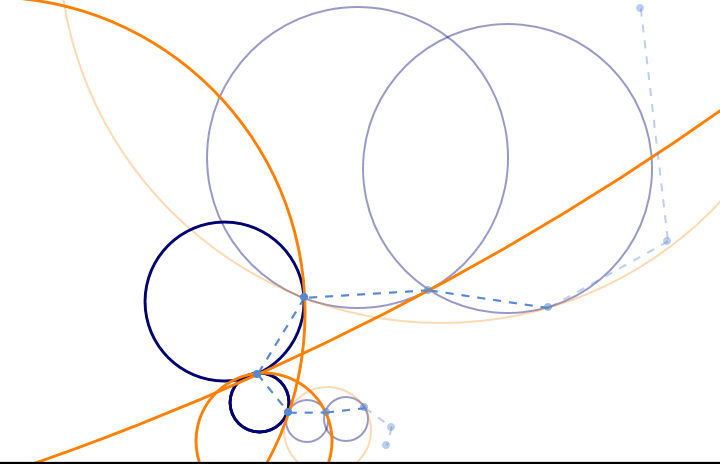}
\caption{Two discrete constrained elastic curves in Euclidean and hyperbolic space with the associated circle pencils $\mathcal{P}^\pm$ (\emph{left}). Several circle congruences (orange) together with their tangential congruences (blue) are illustrated: the elastic congruence with the tangents of the curve, the arc-length circle congruence and a congruence that intersects both at a constant angle and gives rise to an $(\m)$-type Darboux evolution map (\emph{from left to right}).}\label{fig_pencils_congr}
\end{figure}
%
%
\noindent The existence of a further special circle congruence in $\mathcal{P}^\pm$ provides one of the characterizing properties for discrete constrained elastic curves as discussed in \cite{discrete_elastic}. Here we use it as a definition:
%
%
\begin{defi}\label{def_constr_elastic}
A discrete curve $f \in \mathcal{F}_Q$ with constant arc-length in the space form $\Q$ is said to be  \emph{constrained elastic in $\mathcal{Q}$} if the circle congruence $d \in \mathcal{P}^\pm$ that is tangent to the corresponding geodesic edges~$t$ lies in a linear  complex, that is, there exists an \emph{elastic complex} $D \in \mathbb{R}^{3,2}$ such that any circle $d_j$ in the \emph{elastic circle congruence $d \in \mathcal{P}^\pm$} satisfies $\inner{\mathfrak{d}_i, \mathfrak{D}}=0$ for all $i \in \mathcal{V}$.
\end{defi}
%
%
\noindent Thus, any discrete constrained elastic curve in a space form naturally comes with two Darboux evolution maps: the constant arc-length of the discrete curve induces an $(\q)$-type evolution map, while the one induced by the elastic circle congruence is of $(\D)$-type.
\begin{lem}\label{lem_darboux_elastic}
Let $f$ be a discrete constrained elastic curve in a space form $\Q$ with arc-length circle congruence~$a$ and elastic circle congruence~$d$. Then each circle congruence
\begin{equation}\label{equ_s_lambda}
\s^\lambda=\lambda \aaa + \tilde{\lambda}\ddd + \mathfrak{p}, \ \text{where } \tilde{\lambda}:=-\frac{1+2\lambda}{2(1+\lambda)}, \ \mathbb{R} \ni \lambda \neq -1,
\end{equation}
induces an $(\m^\lambda)$-type Darboux evolution map. The corresponding tangential circle congruence $t^\lambda$ lies in a fixed linear complex~$\mathfrak{n} \in \spann{\q,\p}$. 
\end{lem}
\begin{proof}
The first claim is an immediate consequence of Proposition~\ref{prop_combination_s}. 

Moreover, by definition, the tangential circle congruence of the elastic circle congruence~$\mathfrak{d}$ is the tangent line congruence~$\ttt \perp \q$. The circle congruence tangential to the arc-length circle congruence~$a$ is described by 
\begin{equation*}
\tilde{\ttt}_j=\mathfrak{a}_j - \frac{\inner{\mathfrak{a}_j, \mathfrak{f}_j}}{\inner{\mathfrak{f}_i,\mathfrak{f}_j }}\mathfrak{f}_i
\end{equation*}
and therefore lies in the linear complex $(\q + \nu\p)^\perp$ for a suitable $\nu \in \mathbb{R}$ (c.\,f.\,Fact~\ref{fact_const_complex}).

Since the tangential circle congruence~$t^\lambda$ of $s^\lambda$  is given by (\ref{equ_tang}), the claim follows. 
\end{proof}
%
%
Hence, by Lemma~\ref{lem_darboux_elastic} and Theorem~\ref{thm_darboux_trafos}, any discrete constrained elastic curve in a space form admits a family of $(\m^\lambda)$-type Darboux transforms. We will prove that those transforms are again constrained elastic in an appropriate space form. 

To start with, we observe that for a discrete curve with constant arc-length in $\Q$, it follows from Lemma~\ref{lem_arclength} that the vectors
\begin{equation*}
\bb_j:=\aaa_i \inner{\aaa_k, \aaa_j} - \aaa_k \inner{\aaa_i, \aaa_j} \in \spann{ \f_i \inner{\f_k, \q} - \f_k \inner{\f_i, \q}}
\end{equation*}
satisfy $\bb_j \perp \p, \q$. Thus, along the discrete curve, all those vectors $b_j$ lie in a fixed 3-dimensional subspace of $\mathbb{R}^{3,2}$. 

For discrete constrained elastic curves in space forms this property is transferred to further specific circle congruences in $\mathcal{P}^\pm$:
%
%
\begin{lem}\label{lem_arc_length_congr_elastic}
Let $f$ be a constrained elastic curve in the space form~$\Q$ with arc-length circle congruence~$a$ and elastic circle congruence~$d$. Then, for fixed $\lambda$, all vectors
\begin{equation*}
\bb_{j}^\lambda:=\s_i^\lambda \inner{\s_k^\lambda,\s_j^\lambda} - \s_k^\lambda \inner{\s_i^\lambda,\s_j^\lambda}
\end{equation*}
along the curve~$f$ lie in a 3-dimensional subspace, where $s^\lambda$ is a circle congruence as defined in~(\ref{equ_s_lambda}).
\end{lem}
\begin{proof}
We fix homogeneous coordinates $\inner{\f_i, \q} = -1$ for the curve points~$f$. By Lemma~\ref{lem_darboux_elastic}, any circle congruence~$\s^\lambda \perp \overline{\m}^\lambda$ comes together with a tangential circle congruence~$\ttt^\lambda \perp \q + \nu \p$ for a suitable constant $\nu \in \mathbb{R}$. 

If we write the circles $s^\lambda$ as
\begin{equation*}
\s_j^\lambda=\f_i - \frac{\inner{\f_i,\overline{\m}^\lambda}}{\inner{\ttt_{ij}^\lambda, \overline{\m}^\lambda}} \ttt_{ij}^\lambda = \f_k - \frac{\inner{\f_k,\overline{\m}^\lambda}}{\inner{\ttt_{jk}^\lambda, \overline{\m}^\lambda}} \ttt_{jk}^\lambda,
\end{equation*}
for those homogeneous coordinates we have $\xi:=\inner{\s_i^\lambda, \s_j^\lambda}= \inner{\f_i, \f_j}\equiv const.$ due to Lemma~\ref{lem_arclength}(iii). Thus, it follows that
\begin{equation*}
\bb_{j}^\lambda:=\s_i^\lambda \inner{\s_k^\lambda,\s_j^\lambda} - \s_k^\lambda \inner{\s_i^\lambda,\s_j^\lambda} \perp \q + \nu \p, \overline{\m}^\lambda.
\end{equation*}
\end{proof}
%
\noindent Recall from Subsection~\ref{subsect_mDarboux} that various data of the $(\m)$-type Darboux pairs under consideration are related by a fixed Lie inversion. This is the underlying reason that $(\m^\lambda)$-type Darboux transforms of discrete constrained elastic curves provide are again constrained elastic. Note, however, that the space form is typically changing.
%
\begin{prop}\label{prop_again_constrained}
Let $f$ be a discrete constrained elastic curve in the space form~$\Q$. All real and non-circular $(\m^\lambda)$-type Darboux transforms induced by the circle congruences $s^\lambda$ defined in (\ref{equ_s_lambda}) are again constrained elastic in an appropriate space form. 
\end{prop}
\begin{proof}
Suppose that $(f, g)$ is an $(\m)$-type Darboux pair induced by the circle congruence $\s^\lambda \perp \overline{\m}$ (c.\,f.\,Theorem~\ref{thm_darboux_trafos}). Further, let $\sigma_b$ denote the Lie inversion that maps $s^\lambda$ to the curve points of the Darboux transform~$g$.  

Firstly, we prove that the Darboux transform has again constant arc-length in a space form~$\tilde{\Q}$. By above, we have
\begin{equation*}
\sigma_b(\spann{\overline{\m}^\lambda})=\spann{\p}. 
\end{equation*}
Further, using Lemma~\ref{lem_arc_length_congr_elastic}, we choose $\nu \in \mathbb{R}$ such that $\q + \nu \p \perp \s_i^\lambda \inner{\s_k^\lambda,\s_j^\lambda} - \s_k^\lambda \inner{\s_i^\lambda,\s_j^\lambda}$ for all vertices and define $\q^g:=\mathfrak{n}+ \inner{\mathfrak{n},\p} \p \perp \p$, where $\mathfrak{n}:=\sigma_b(\q + \nu \p)$. Since the Lie inversion~$\sigma_b$ preserves inner products, we conclude from Lemma~\ref{lem_arc_length_congr_elastic} that the Darboux transform~$g$ satisfies
\begin{equation*}
\g_i \inner{\g_k,\g_j} - \g_k \inner{\g_i,\g_j} \perp \q^g,
\end{equation*}
which together with Lemma~\ref{lem_arclength}(ii) shows that $g$ indeed has constant arc-length in the space form determined by the space form vector~$\q^g$.
\\\\Thus, the discrete Darboux transform~$g$ admits two $(\m^1)$-type and $(\m^2)$-type Darboux evolution maps induced by the arc-length circle congruence of $g$ and the circle congruence $\sigma_b(\f) \in \hat{\mathcal{P}}^\pm$. Therefore, by formula~(\ref{equ_further_circle_congr}), we obtain a family of circle congruences in $\hat{\mathcal{P}}^\pm$ that lie in fixed linear complexes. In particular, there exists a circle congruence that intersects the arc-length circles of $g$ orthogonally and is therefore in oriented contact with the tangent line congruence of $g$ in $\tilde{\Q}$. Hence, we have proven that $g$ is a discrete constrained elastic curve in the space form $\tilde{\Q}$.
\end{proof}
%
\noindent Applying Lemma~\ref{lem_darboux_elastic} and Proposition~\ref{prop_again_constrained} repeatedly, we conclude the following construction for holomorphic maps consisting of constrained elastic parameter lines. Examples of such discrete holomorphic maps can be found in the Figures~\ref{fig_girlande}, \ref{fig_conical} and \ref{fig_eight}.
\begin{thm}\label{thm_extension}
Any discrete constrained elastic curve in a space form~$\Q$ with elastic complex~$D$ can be extended to an infinite sequence of constrained elastic curves in various space forms that provides a discrete holomorphic map of $(M)$-type. The space form vectors and elastic complexes of all constrained elastic curves that can occur in the holomorphic map lie in $\spann{\q, \mathfrak{D}, \p} \subset \mathbb{R}^{3,2}$.
\end{thm}
\noindent We remark that this extension is not unique. For any further Darboux transform there is a 1-parameter choice that basically represents the stepsize between the two curves in the discrete holomorphic map.

Via lifted-foldings, we can therefore generate discrete isothermic nets with a family of planar or spherical curvature lines so that those parameter lines are M\"obius images of constrained elastic curves in various space forms.

Moreover, from Proposition~\ref{prop_flexible_pair} we learn that the configuration of the spheres containing the spherical parameter lines of isothermic nets generated from the holomorphic maps described in Theorem~\ref{thm_extension} is restricted: if the holomorphic map is canonically embedded in a plane~$e_0 \in \Light \subset \mathbb{R}^{4,2}$, then the spheres lie in the subspace $\spann{\mathfrak{e}_0, \hat{\q}, \hat{\mathfrak{D}}, \p}$, where $\hat{q}$ and $\hat{D}$ denote the embeddings of the space form vector and the elastic complex of the initial constrained elastic curve, respectively.

Thus, we have:

\begin{cor}
The spheres containing the curvature lines of an isothermic net obtained via lifted-folding from a discrete holomorphic map induced by a discrete constrained elastic curve lie in a 4-dimensional subspace of $\mathbb{R}^{4,2}$. 
\end{cor}

Theorem~\ref{thm_extension} raises the obvious question of whether spherical curvature lines of discrete isothermic nets can be always interpreted as discrete constrained elastic curves in space forms in the sense of Definition~\ref{def_constr_elastic}.  To answer this question we have to investigate whether the Construction~2 given in Subsection~\ref{subsect_mDarboux} always provides a constrained elastic curve in a space form (see Corollary~\ref{cor_constr2}).

\bigskipp Suppose we have given a discrete constrained elastic curve in a space form~$\Q$ with elastic complex~$D$. Recall that this curve then has Darboux evolution maps of $(D)$-type and $(\q)$-type (see Lemma~\ref{lem_darboux_elastic}). Let us consider now a part of this discrete curve consisting of five consecutive points and the five corresponding elastic circles. This data uniquely determines the elastic complex~$D$ and the space form in which those five points have constant arc-length (c.\,f.\,Lemma~\ref{lem_arclength}). However, those five points only provide three arc-length circles, which provides us with the 1-parameter freedom to extend this data to a discrete curve that has two Darboux evolution maps of $(\D)$- and $(\m_2)$-type. 

If $\spann{\m_2}=\spann{\q}$, then the extension gives the constrained elastic curve we started with. If not, then we generate a discrete curve that does not have constant arc-length in a space form. An example is illustrated in Figure~\ref{fig_nonarclength}.
%
%
\begin{figure}
\begin{minipage}{4cm}
\hspace*{-1.3cm}\includegraphics[scale=0.4]{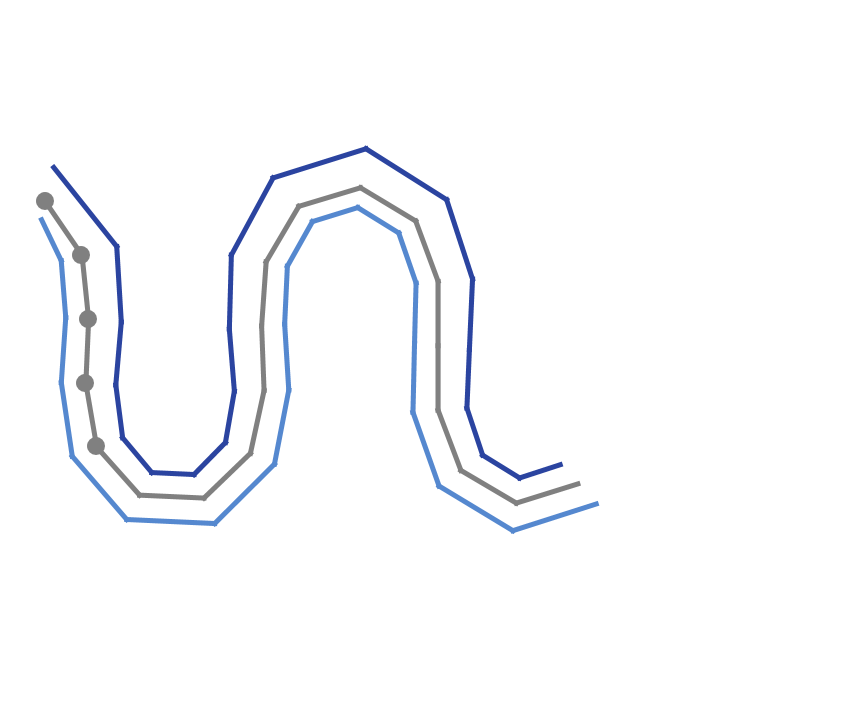}
\end{minipage}
\begin{minipage}{4cm}
\hspace*{-0.5cm}\includegraphics[scale=0.4]{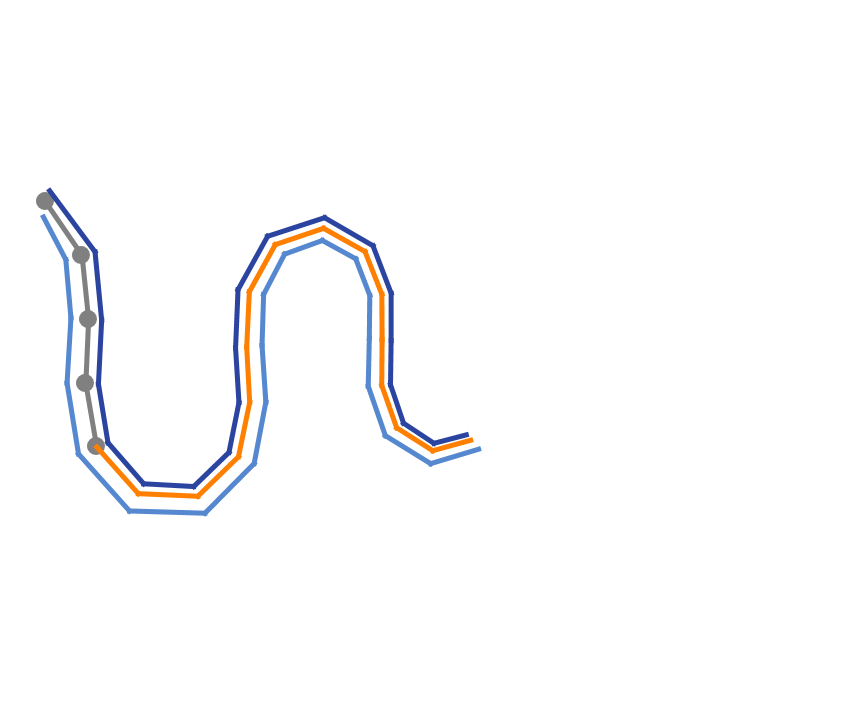}
\end{minipage}
\begin{minipage}{4cm}
\hspace*{-0.6cm}\includegraphics[scale=0.4]{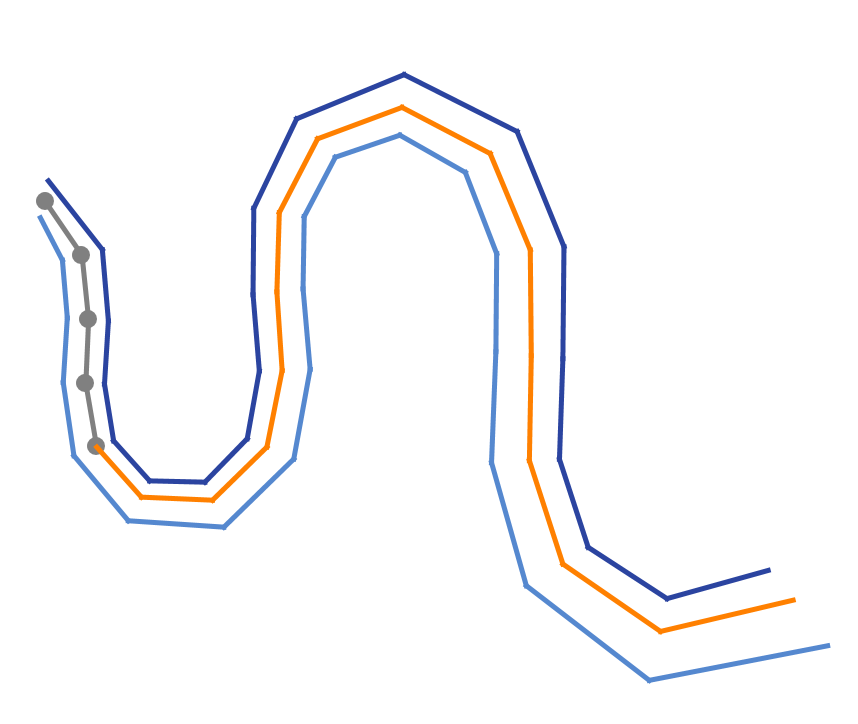}
\end{minipage}
\caption{Three examples of sequences of $(\m^{(ij)})$-type Darboux transforms that all contain the same five gray curve points. \emph{Left.} The middle gray curve has constant arc-length in Euclidean space and is a discrete elastic curve. \emph{Middle and Right.} Using Construction~2, the five gray points are extended differently. These curves do not have constant arc-length in a space form.}\label{fig_nonarclength}
\end{figure}
%
%
\\\\To sum up, we can say that a planar discrete curve admitting two Darboux transforms of $(\m_1)$- and $(\m_2)$-type is not necessarily parametrized by constant arc-length and therefore does not need to be discrete constrained elastic in the sense of Definition~\ref{def_constr_elastic}.

However, we conclude that the discrete curves obtained by Construction~2 are “piecewise” constrained elastic in an appropriate space form: any five consecutive points determine a space form, where this part of the curve has constant arc-length. For this space form we obtain distinct geodesic edges of constant length. By construction, the five circles tangent to those geodesic edges lie in a fixed linear complex and thus give rise to an elastic circle congruence. 

Note that being piecewise constrained elastic is indeed a condition for a discrete curve, since the five tangential circles do not have to lie in a 4-dimensional subspace of $\mathbb{R}^{3,2}$.
%
%
\subsection{Isothermic nets with a family of planar lines of curvature}\label{subsect_iso_planar} We have seen in Subsection~\ref{subsect_planar} that planar parameter lines for circular nets arise from planar circular nets of $(\m^{(ij)})$-type if and only if all linear complexes satisfy $\m^{(ij)} \perp \q_0$, where $\q_0$ is the space form vector that determines a Euclidean space form. Hence, by Theorem~\ref{thm_iso_folding}, discrete isothermic nets with a family of planar curvature lines arise from discrete holomorphic nets of this type.

In particular, if we use Theorem~\ref{thm_extension} and build the holomorphic map from an initial discrete constrained elastic curve, planarity of the parameter lines restricts the choice of the starting curve~$f$: the space form vector~$\q$, as well as the elastic complex~$D$ have to fulfill $\q, \D \perp \q_0$. This is because suitable circle congruences along~$f$ induce $(\m^\lambda)$-type Darboux transforms with $\m^\lambda \in \spann{\q, \D, \p}$ (see (\ref{equ_comp_pencil})). 
By Proposition~\ref{prop_again_constrained}, this property also holds for the other constrained elastic curves in the holomorphic map.

\bigskipp Therefore, we conclude that admissible discrete starting curves and their Darboux transforms are constrained elastic either in Euclidean space or in a Poincar\'e half-plane modelling a hyperbolic space form. Moreover, since all folding axes also lie in $\spann{\q, \D, \p}$, those generically lie in a circle pencil. Thus, they either intersect in one common point or are parallel. Similar to the smooth case \cite{adam,bobenko2023isothermic, Darboux_book_iso}, this outlined construction gives examples of conical and cylindrical type; see also Figure~\ref{fig_conical}.

\begin{figure}
\begin{minipage}{4cm}
\hspace*{-2cm}\includegraphics[scale=0.6]{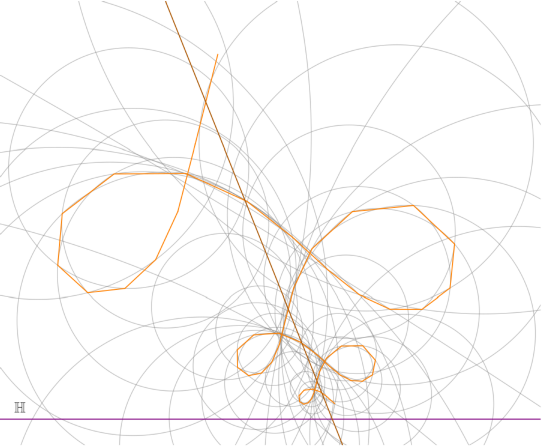}
\end{minipage}
\begin{minipage}{4cm}
\hspace*{0cm}\vspace*{0.8cm}\includegraphics[scale=0.6]{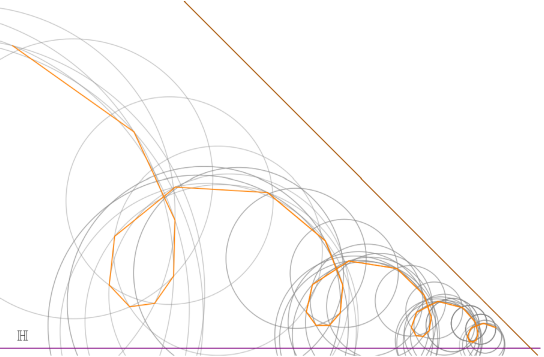}
\end{minipage}
\begin{minipage}{4cm}
\hspace*{2cm}\vspace*{0.25cm}\includegraphics[scale=0.5]{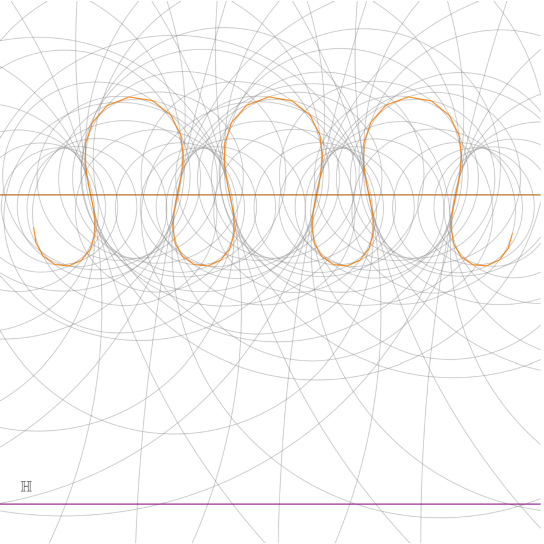}
\end{minipage}
%
\begin{minipage}{4cm}
\hspace*{-2cm}\includegraphics[scale=0.6]{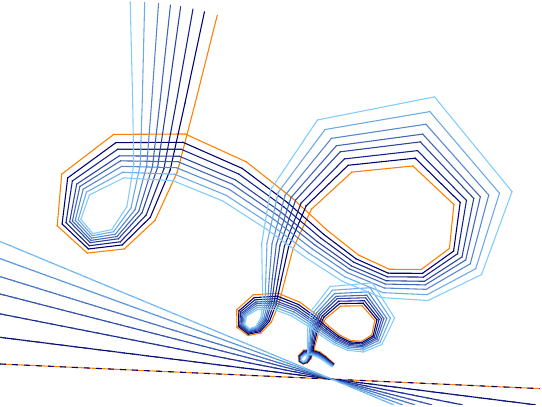}
\end{minipage}
\begin{minipage}{4cm}
\hspace*{0cm}\vspace*{1cm}\includegraphics[scale=0.6]{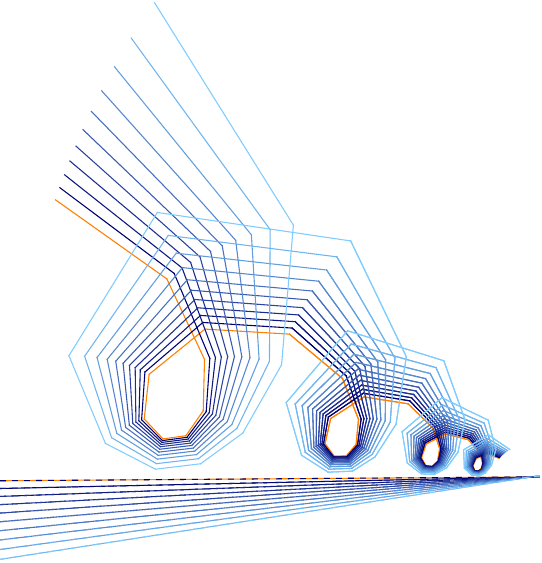}
\end{minipage}
\begin{minipage}{4cm}
\hspace*{2cm}\includegraphics[scale=0.5]{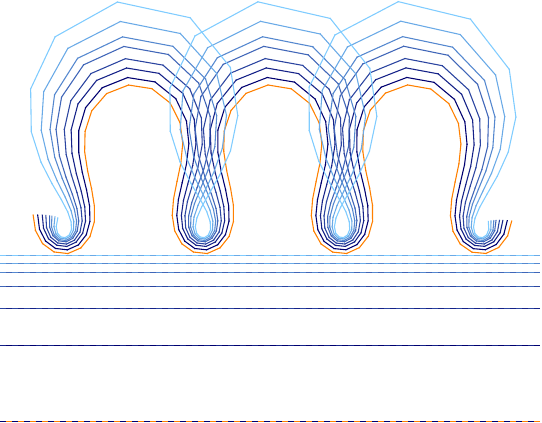}
\end{minipage}
\caption{\emph{First row.} Three discrete constrained elastic curves in the Poincar\'e half-plane model. The elastic circle congruence (gray) intersects a directrix at a (possibly imaginary) constant angle.
\emph{Second row.} Holomorphic maps of $(M)$-type that give rise to isothermic nets with a family of parameter lines in Euclidean planes via lifted-foldings. The folding axes (dashed) lie in the line pencil determined by the hyperbolic boundary and the corresponding directrix.}
\label{fig_conical}
\end{figure}
%
%
\section{Special solutions: quasi-periodicity, symmetries and closedness}\label{sect_special}
\noindent We have established a connection between specific discrete holomorphic maps and isothermic circular nets with a family of spherical parameter lines via lifted-folding. This section aims to show how this technique can be used to solve specific monodromy problems. 
%
%
\subsection{(Quasi-)periodic discrete holomorphic maps related to isothermic nets with spherical curvature lines}
We say that a map $w: \mathbb{Z} \to \Light \subset \mathbb{R}^{3,2}$ is \emph{T-quasi-periodic} if there exists a choice of homogeneous coordinates $\mathfrak{w} \in w$, a constant $T \in \mathbb{R}^\times$ and a Lie transformation $L \in O(3,2)$, its \emph{monodromy}, such that $\mathfrak{w}_{(i + T)} = L( \mathfrak{w}_i)$ for all $i \in \mathbb{Z}$. 
Furthermore, the map~$w$ is called \emph{T-periodic}, if $\mathfrak{w}_{(i + T)} = \mathfrak{w}_i$ for a suitable choice of homogeneous coordinates. We say that a map is \emph{quasi-periodic with respect to a space form~$\Q$}, if it is quasi-periodic and the monodromy $L$ is a space form motion in $\Q$.

In particular, a T-periodic discrete curve $f: \mathbb{Z} \to \Light$ is closed, while a quasi-periodic discrete curve is often also called twisted polygon \cite{affolter2023integrable, AFIT_dynamics}.
\\\\Discrete Darboux transformations that preserve (quasi)-periodicity (with the same monodromy) of discrete curves have recently attracted a lot of attention. In particular, their integrability has been studied from different perspectives (see \cite{affolter2023integrable, AFIT_dynamics, CHO2023102065, periodic_conformal}). 

We point out that the $(M)$-type Darboux transformations used to generate the discrete holomorphic maps of $(M)$-type may give explicit examples in this realm (c.\,f.\,Figure~\ref{fig_quasi_H}): quasi-periodicity of an initial discrete curve and its associated circle congruences in $\mathcal{P}^\pm$ is propagated to their $(\m)$-type Darboux transforms. This simply follows from the fact that the curve points of any such special Darboux transform are obtained by a Lie inversion from the quasi-periodic circle congruence in $\mathcal{P}^\pm$ of the initial curve that induces the transform (see Subsection~\ref{subsect_mDarboux}).
%

\bigskipp Explicit examples of quasi-periodic initial curves are provided by discrete elastic curves in Euclidean space as introduced in \cite{BS_discreteelastic}. Those are included in the class of constrained elastic curves and therefore give rise to discrete holomorphic maps of $(M)$-type by Theorem~\ref{thm_extension}. 

As found in \cite{fairing_elastica}, a discrete elastic curve $\gamma$ in Euclidean space with constant arc-length~$h>0$ can be explicitly described by Jacobi elliptic functions via 
\begin{equation}\label{equ_expl_elastic}
\begin{aligned}
\gamma_{n+1}-\gamma_n = h \begin{pmatrix}f(n)f(n+1) - g(n)g(n+1) 
\\f(n)g(n+1) + g(n)f(n+1) \end{pmatrix}&, \text{where } z \in \mathbb{R}, \ q \in [ 0, 2 \pi ] \text{ and }
\\[10pt] \text{for }  0< k < 1: \ f(n)&:=\text{cn}(\frac{zn+q}{k}; k),  
\\\phantom{ \text{for }  0< k < 0: } \ g(n)&:=\text{sn}(\frac{zn+q}{k}; k),
\\[10pt] \text{for }  1< k:  \ f(n)&:=\text{dn} (zn+q; k), 
\\\phantom{ \text{for }  1< k: xxx } \ g(n)&:=k \text{sn} (zn+q; k).
\end{aligned}
\end{equation}
Moreover, the curvature of those curves is given by 
\begin{equation*}
\begin{aligned}
\text{for }  1< k: \ \kappa(n)&= \frac{2}{h}\frac{\text{sn}(\frac{z}{k};k)}{\text{cn}(\frac{z}{k};k)}\text{dn}(\frac{zn+q}{k};k),
\\[10pt]\text{for }  0< k < 1: \ \kappa(n)&= \frac{2k}{h}\frac{\text{sn}(z;k)}{\text{dn}(z;k)}\text{cn}(zn+q;k).
\end{aligned}
\end{equation*}
Thus, periodic curvature functions are obtained, for any $r \in \mathbb{R}^\times$, by defining
\begin{equation*}
\begin{aligned}
0<k<1: \ z:= \frac{4k K(k)}{r},
\\[6pt] 1<k: \ z:= \frac{4 K(\frac{1}{k})}{r}.
\end{aligned}
\end{equation*}
By (\ref{equ_expl_elastic}), those periodic curvature functions induce (quasi)-periodic discrete curves. Moreover, via the arc-length and the elastic circle congruence, the quasi-periodicity is transferred to the circle congruences in $\mathcal{P}^\pm$ that induce the $(m)$-type Darboux transformations. Generically, in this case the monodromy is given by a translation along the directrix, which is a Euclidean line. 

As a special case, we also obtain periodic curves, namely discrete versions of the elastic eight in Euclidean space (see Figure~\ref{fig_eight}). 
\begin{figure}
\begin{minipage}{4cm}
\hspace*{-2.5cm}\includegraphics[scale=0.2]{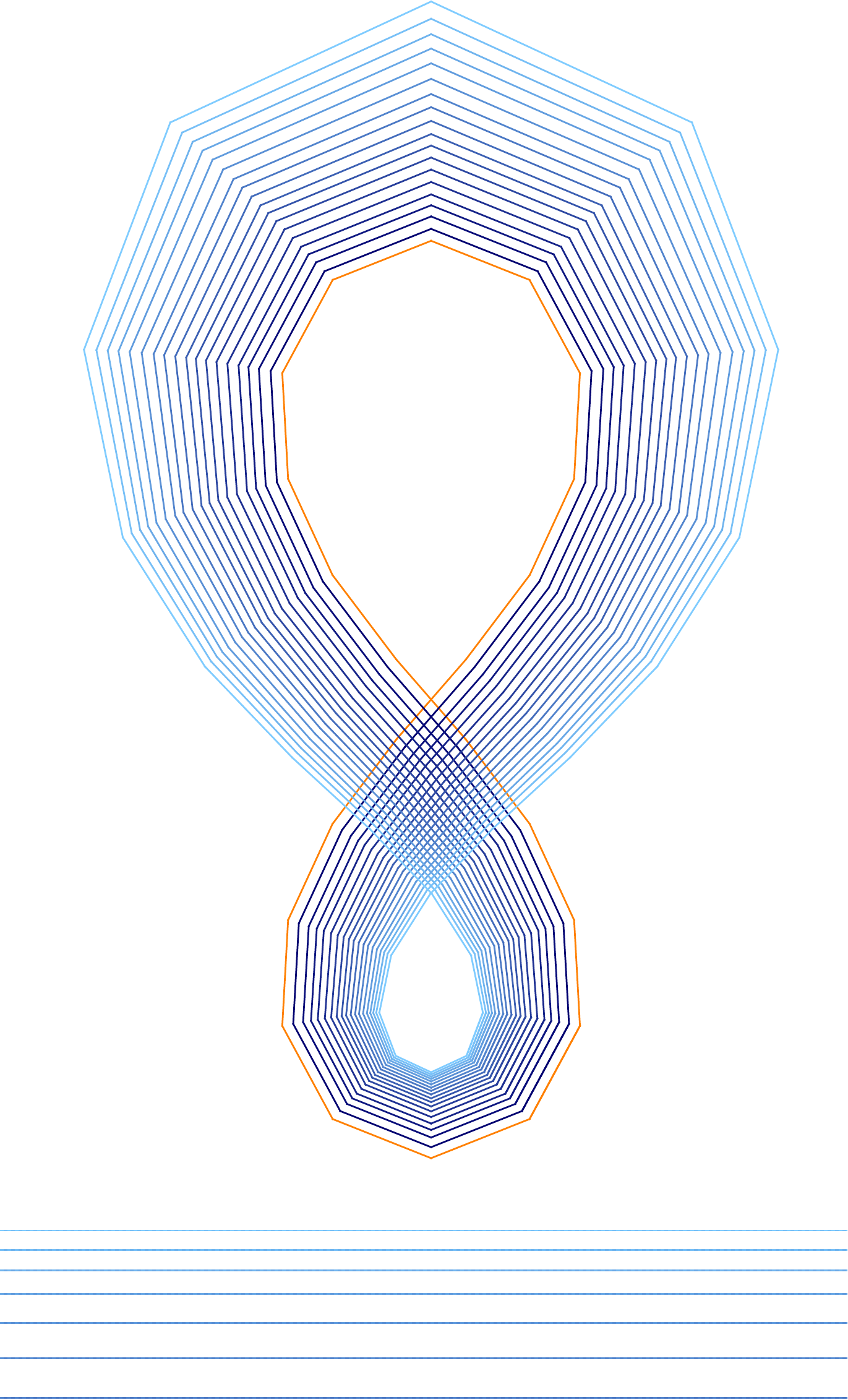}
\end{minipage}
\begin{minipage}{4cm}
\hspace*{-1.9cm}\includegraphics[scale=0.45]{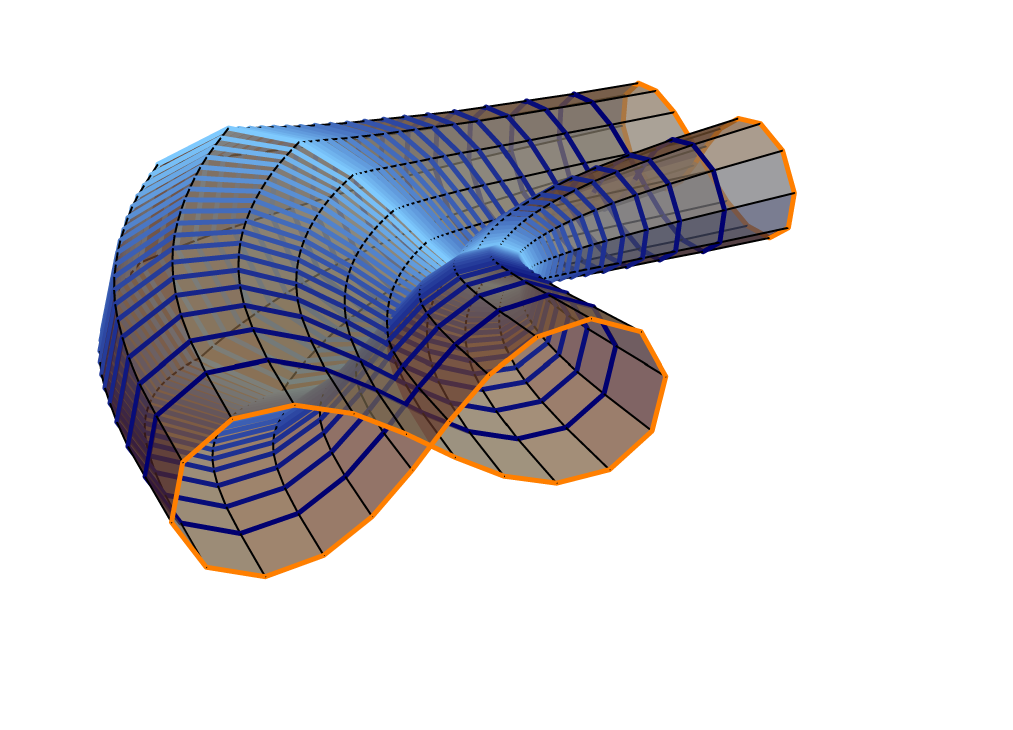}
\end{minipage}
\begin{minipage}{3cm}
\hspace*{0.65cm}\includegraphics[scale=0.3]{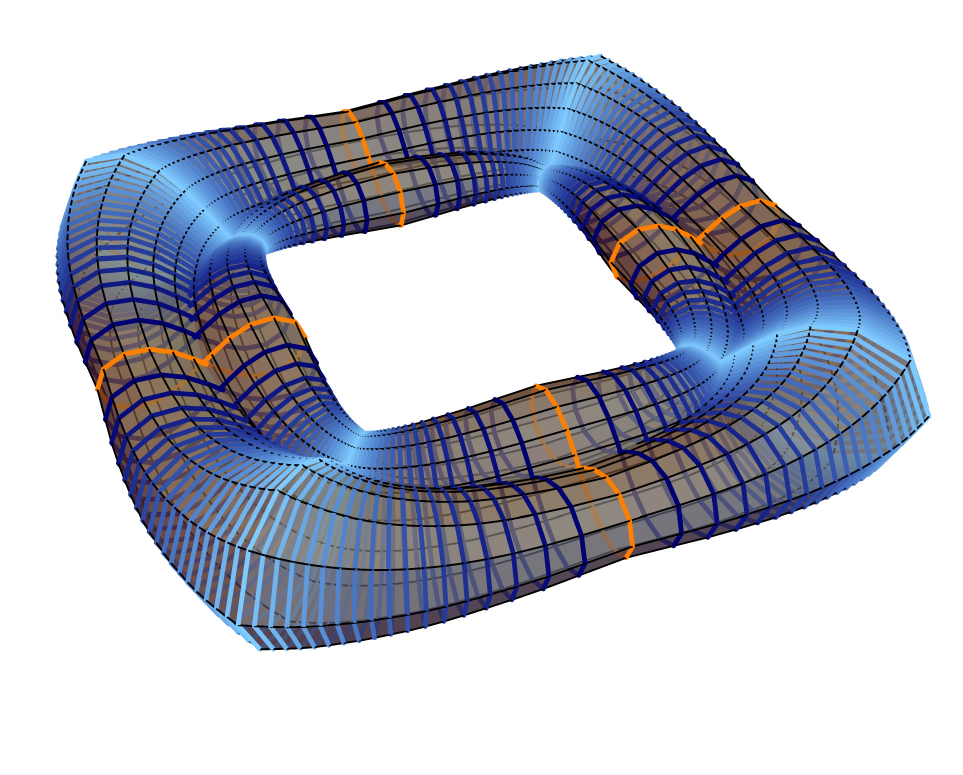}
\hspace*{1.25cm}\includegraphics[scale=0.24]{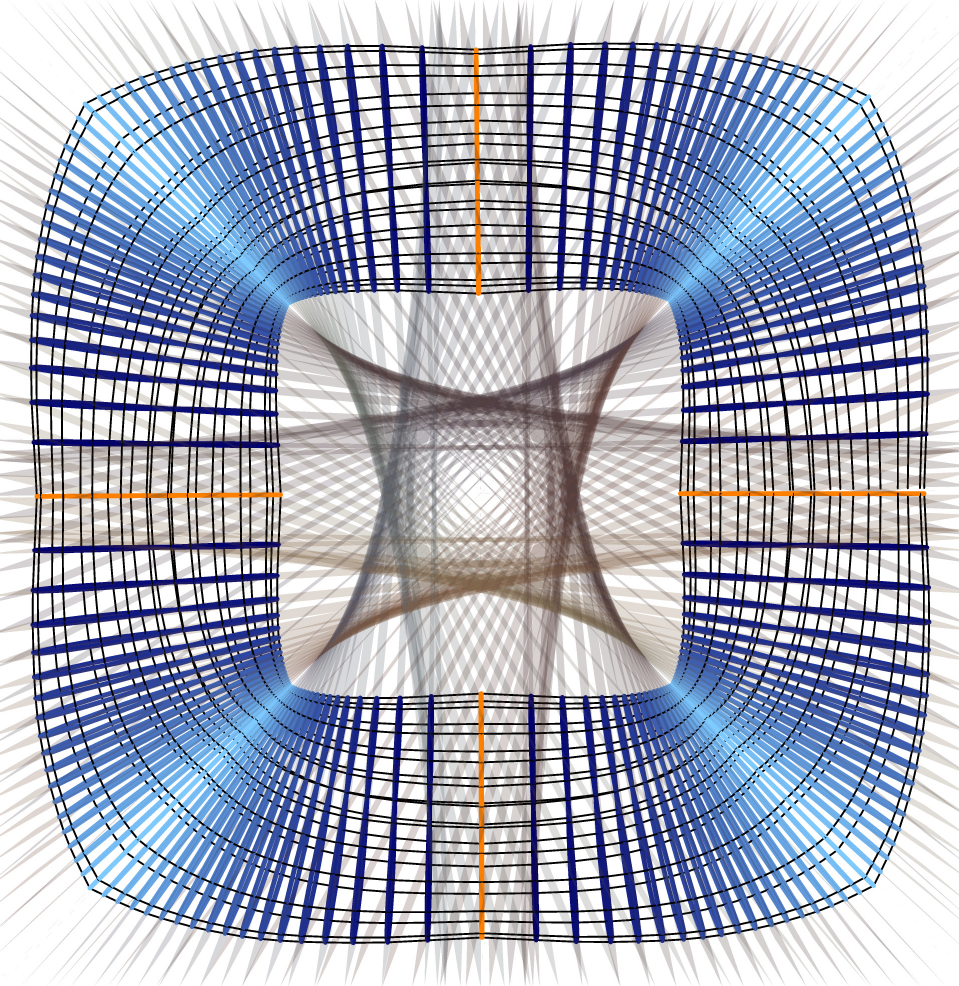}
\end{minipage}
\caption{\emph{Left.} A discrete version of the Euclidean elastic eight as obtained in \cite{fairing_elastica} (orange) and its extension to a discrete $(M)$-type holomorphic map. The corresponding folding axes are given by parallel lines. \emph{Middle.} Lifted-folding then yields an extended fundamental piece with a reflectional symmetry in a Euclidean plane. \emph{Right.} A discrete isothermic torus with a family of planar lines of curvature compound of four extended fundamental pieces.}\label{fig_eight}
\end{figure}

\bigskipp These periodic curvature functions also give rise to discrete elastic curves in other space forms. As in the Euclidean case, those curves are then quasi-periodic with respect to the considered space form (see Figure~\ref{fig_quasi_H} for an example in hyperbolic space). 
%
%
\begin{figure}
\begin{minipage}{5cm}
\hspace*{-3.5cm}\includegraphics[scale=0.48]{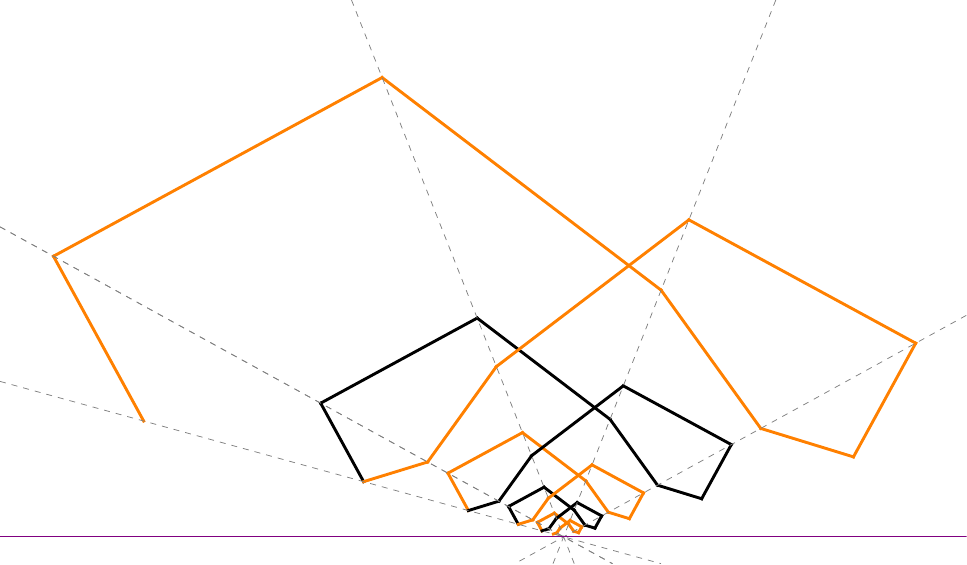}
\end{minipage}
\begin{minipage}{5cm}
\hspace*{-1cm}\includegraphics[scale=1]{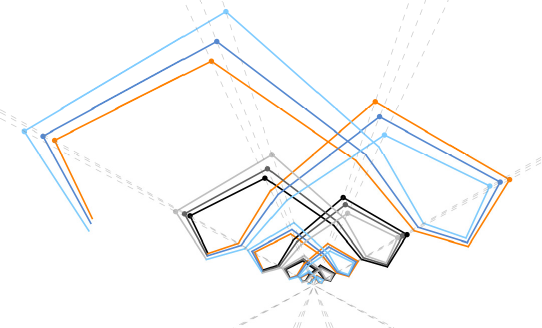}
\end{minipage}
\caption{A quasi-periodic discrete elastic curve (orange) in the Poincar\'e half-plane model with its directrix and its extension to a discrete holomorphic map of $(M)$-type. The monodromy is given by a homothety whose center is the intersection point of the hyperbolic boundary with the directrix.}\label{fig_quasi_H}
\end{figure}
%
%
\begin{figure}
\hspace*{-1cm}\begin{minipage}{4cm}
\hspace*{-2.6cm}\includegraphics[scale=0.17]{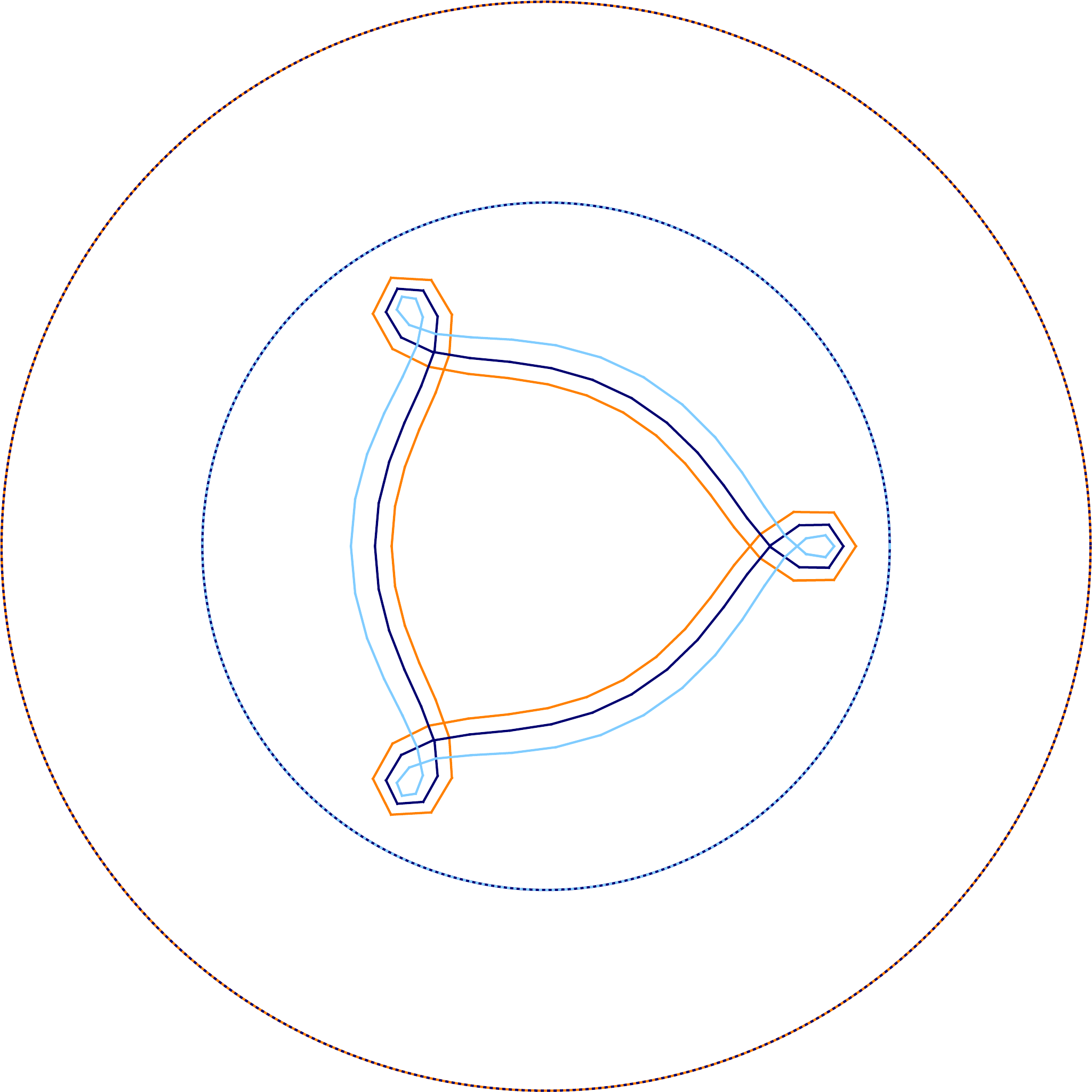}
\\\hspace*{-3.9cm}\includegraphics[scale=0.36]{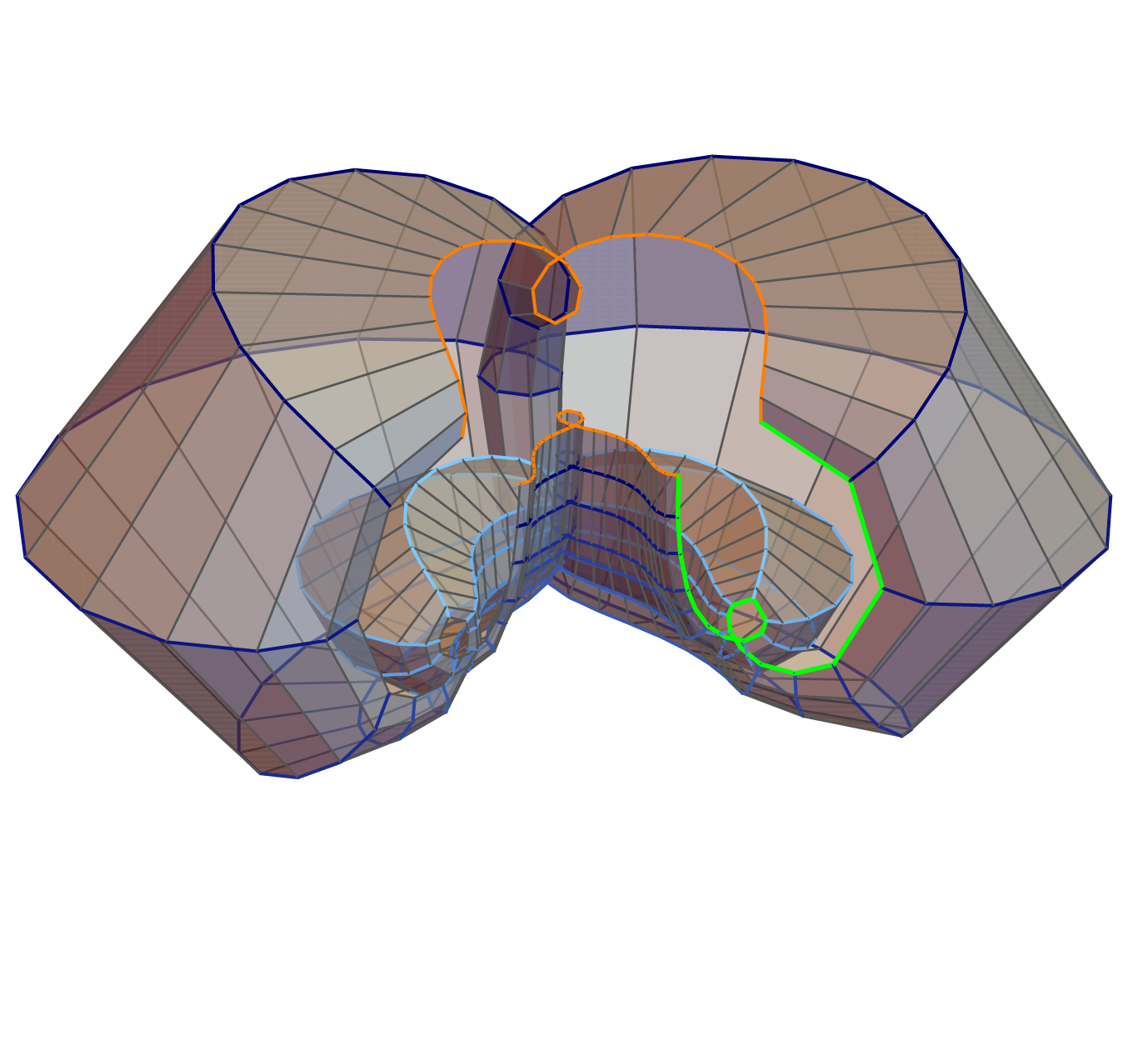}
\end{minipage}
\begin{minipage}{4cm}
\hspace*{1cm}\vspace*{-0.5cm}\includegraphics[scale=0.36]{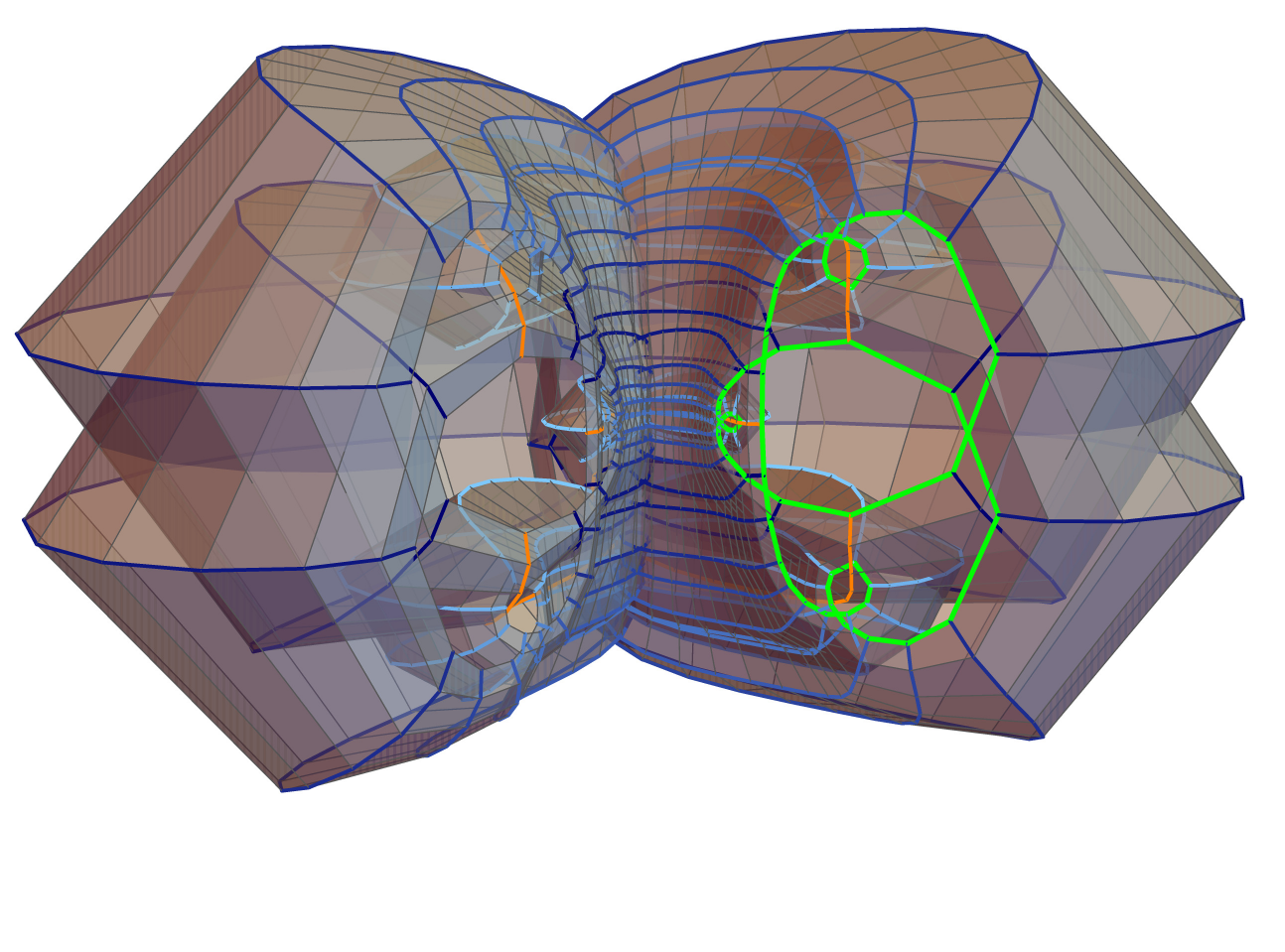}
\\\hspace*{1cm}\vspace*{1.7cm}\includegraphics[scale=0.3]{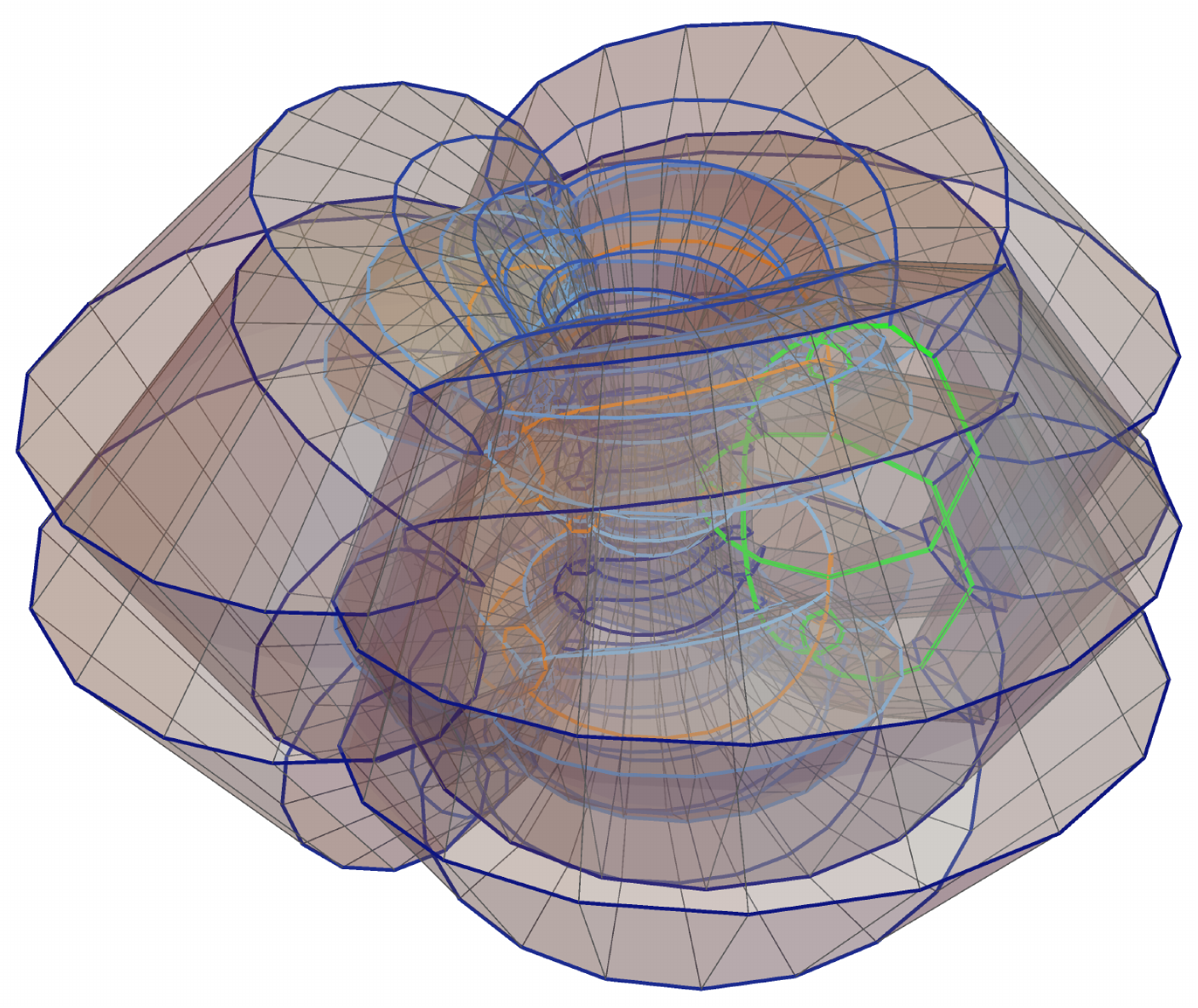}
\end{minipage}
\vspace*{-1.8cm}\caption{A closed discrete arc-length preserving Darboux transform of a discrete circle in Euclidean space (orange, see~\cite{CHO2023102065}) and its extension to a discrete $(M)$-type holomorphic map. The corresponding folding axes are concentric circles. The discrete isothermic net in 3-space is obtained by an lifted-folding and provides a topological torus that consists of three fundamental pieces and their reflectional extensions.}\label{fig_dcircle}
\end{figure}

%
%
\begin{figure}
\begin{minipage}{7cm}
\includegraphics[scale=0.29]{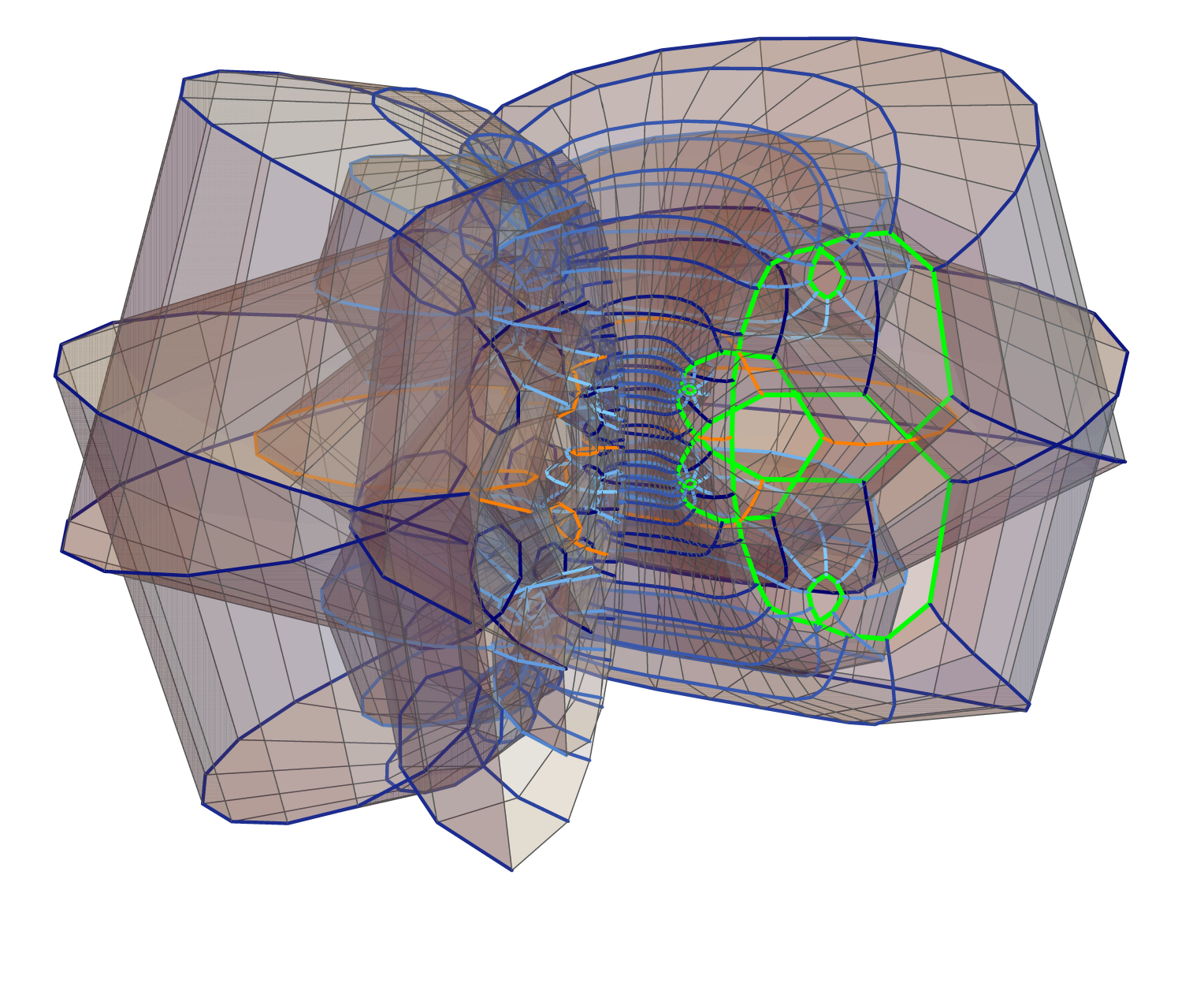}
\end{minipage}
\begin{minipage}{7cm}
\includegraphics[scale=0.34]{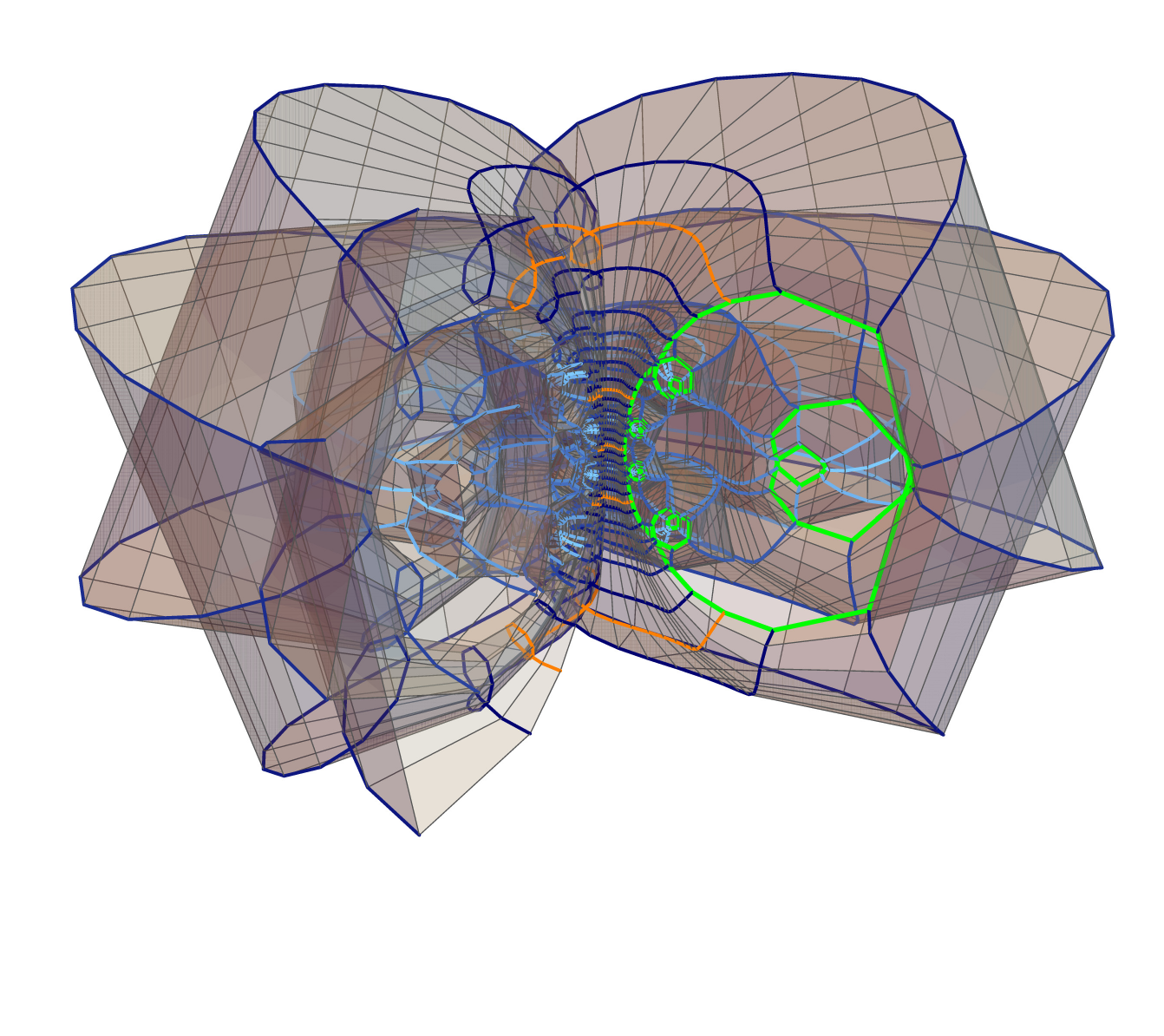}
\end{minipage}
\caption{Discrete isothermic tori consisting of four and six fundamental pieces and their reflectional extensions.}\label{fig_dcircle2}
\end{figure}
%
%
%
\subsection{Discrete isothermic topological cylinders and tori with symmetries}\label{subsect_tori}
The concept of lifted-folding provides an efficient way to control the topology of the generated discrete surfaces with spherical lines of curvature. This relies on the obvious fact that  lifted-foldings preserve closedness of the spherical parameter lines. Specifically, periodic holomorphic maps of $(M)$-type lead automatically to isothermic topological cylinders.

\bigskipp Smooth isothermic tori with planar curvature lines have been constructed in~\cite{bobenko2023isothermic} by piecing together several fundamental pieces. Here we use a similar idea to generate discrete isothermic tori with a family of planar or spherical parameter lines. 

Suppose that $F$ is a discrete isothermic net given by a finite sequence $(f_1, \cdots, f_n)$ of spherical parameter lines. We say that $f$ is a \emph{fundamental piece} bounded by the spheres $\s_1 \ni \f_1$ and $\s_n \ni \f_n$. 

A reflection of the patch in the boundary sphere $s_n$, then yields a canonical extension of the discrete surface: in the light-cone model, this reflection is described by the inversion $\sigma_a$ with respect to $\mathfrak{a}:=\s_n + \inner{\s_n, \p}\p$. In terms of the folding construction described in Theorem~\ref{thm_mflexible}, this amounts to a sign change of the folding parameters. Thus, we obtain the isothermic net $(f_1, \cdots, f_{n-1}, f_n, \sigma_a(f_{n-1}), \cdots, \sigma_a(f_1))$, called an \emph{extended fundamental piece} (see Figure~\ref{fig_eight}, \emph{middle}).

Reflection in the new boundary sphere $\sigma_a(s_1)$, then again extends this isothermic patch. As a consequence, successive reflections lead to an infinite sequence of spherical parameter lines that give a discrete isothermic net compound of extended fundamental pieces. 

\bigskipp Depending on the intersection angle between the two boundary spheres $s_1$ and $s_n$ of a fundamental piece, this construction closes up after a certain number of reflections. As we will see below, by using this reflection principle, we can generate discrete isothermic tori. Some examples are illustrated in the Figures~\ref{fig_eight},~\ref{fig_dcircle} and~\ref{fig_dcircle2}. 

The intersection angle of the two boundary spheres $s_1$ and $s_n$ of a fundamental piece can be changed by varying the folding parameters. For a prescribed angle, formula~(\ref{equ_intersection_angle}) allows to explicitly compute solutions for the folding parameters. Solutions are far from being unique. For example, by just applying lifted-foldings to change the first and the last sphere of the fundamental piece, we gain a two parameter freedom that generically allows to generate any prescribed angle.

\begin{thm}
If the two boundary spheres of an extended fundamental piece obtained from a periodic holomorphic map of $(M)$-type intersect at a rational multiple of~$2\pi$, the isothermic net generated by successive reflections provides a topological torus. 
\end{thm}
\begin{proof}
Suppose that the boundary spheres $s_1$ and $\sigma_a(s_1)$ intersect and therefore determine an elliptic sphere pencil. Since the boundary spheres of all subsequent extended fundamental pieces are obtained by suitable  reflections, all those boundary spheres lie in this elliptic sphere pencil. 
Thus, if the angle is a rational multiple of $2\pi$, the composition of an appropriate number of these   reflections gives the identity and the net indeed closes to a topological torus.
\end{proof}

\appendix
\setcounter{thm}{0}
\section{Light cone model}
\noindent This article is mainly concerned with special configurations of points, lines and circles in 2-space, as well as points, planes and spheres in 3-space. Moreover, to gain insights into the geometry of certain discrete curves we consider them in space form geometry, that is, those are viewed as curves in a space of constant sectional curvature.

For our investigations we make use of the light cone model, where those geometric objects are naturally represented. Moreover, in this setup all space form geometries arise as subgeometries of Lie sphere geometry, which enables us to simply switch between different space forms.

\bigskipp At this point we give a very brief introduction to this subject and set out the notations used in this work. We also  collect some basic facts that will be used throughout the text. Our focus lies on the 2-dimensional case. The generalization to 3-space follows accordingly by considering the 6-dimensional vector space $\mathbb{R}^{4,2}$. 

For more details the reader is referred to the exhaustive literature in this area, for example, to the surveys \cite{blaschke} and \cite{book_cecil}.

\bigskipp We work in a 5-dimensional vector space $\mathbb{R}^{3,2}$ equipped with an inner product of signature $(3,2)$ denoted by $\inner{. ,.}$. The central object in this model is the projective light-cone
\begin{equation*}
\Light:= \{ \mathbb{R}\mathfrak{v} \subset \mathbb{R}^{3,2} \ | \ \inner{\mathfrak{v}, \mathfrak{v}}=0, \mathfrak{v} \neq 0 \}.
\end{equation*}
Elements in $\Light$ are then in 1-to-1 correspondence with the set of points, oriented circles and oriented lines in $\mathbb{R}^2$ via the following identification table
\\\begin{center}
\begin{tabular}{p{5.5cm} | p{7.5cm}}
\textbf{Geometric objects in $\mathbb{R}^2 \cup \infty$} & \textbf{Vectors in $\Light$}
\\\hline $\infty$ & $\q_0:=( 0,0, 1, -1, 0 )$
\\\hline point $(x,y) \in \mathbb{R}^2$ & $\mathbb{R}(x ,y, \frac{1}{2}(1-x^2-y^2), \frac{1}{2}(1+x^2+y^2), 0)$
\\\hline oriented circle with radius $r \in \mathbb{R}^\times$ and center $(x,y) \in \mathbb{R}^2$ & $\mathbb{R}(x ,y, \frac{1}{2}(1-x^2-y^2+r^2), \frac{1}{2}(1+x^2+y^2-r^2), r)$
\\\hline oriented line with normal distance $d$ and unit normal vector $(x,y) \in \mathbb{R}^2$ & $\mathbb{R}(x,y ,-d,d,1)$
\end{tabular}
\end{center} 
\ \\As common in M\"obius geometry, we consider oriented lines as circles with infinite radius.
%
%

\bigskipp Moreover, throughout this work we will use the following notation convention: homogeneous coordinates of elements in the projective space $\mathbb{P}(\mathbb{R}^{3,2})$ will be denoted by the
corresponding black letter; if statements hold for arbitrary
homogeneous coordinates we will use this convention without 
explicitly mentioning it.

\bigskipp In addition to this table, we further define a so-called \emph{point-sphere complex} $\p \in \mathbb{R}^{3,2}$ as 
\begin{equation*}
\p:=(0,0,0,0,1).
\end{equation*} 
An element $v \in \Light$ then represents a point in $\mathbb{R}^2$ if and only if $\inner{\mathfrak{v}, \p}=0$. Similarly, the vector~$\mathfrak{q}_0$ is used to identify Euclidean lines: $v \in \Light$ is identified with an oriented line in $\mathbb{R}^2$ if and only if $\inner{\mathfrak{v}, \q_0 }=0$.

\bigskipp Two geometric objects $u,v$ in $\mathbb{R}^2$ are in \emph{oriented contact}, that is, two circles are tangent with corresponding orientation or, a point lies on a circle, if and only if $\inner{\mathfrak{u}, \mathfrak{v}}=0$. 

The angle $\varphi$ between two circles $u,v
\in \Light$ is given by
\begin{equation}\label{equ_intersection_angle}
 \cos \varphi
  = 1 + \frac{
   \inner{\mathfrak{u},\mathfrak{v}}}
   {\inner{\mathfrak{u},\mathfrak{p}}
   \inner{\mathfrak{v},\mathfrak{p}}}.
\end{equation} 

\bigskipp Any element $a \in \mathbb{P}(\mathbb{R}^{3,2})$ defines
a \emph{linear (circle) complex}
$\mathbb{P}(\mathcal{L}\cap\{a\}^\perp)$,
that is, a 2-dimensional family of circles. Depending on the signature of the vector $a$, those families admit different geometric interpretations. To see those, we introduce the (possibly complex) vector 
\begin{equation*}
\mathfrak{a}^\star:=\mathfrak{a} + \lambda \p, \ \text{ where } \lambda := \inner{\mathfrak{a}, \p} - \sqrt{\inner{\mathfrak{a}, \p}^2 + \inner{\mathfrak{a}, \mathfrak{a}} } 
\end{equation*}
and observe that $\chi:=\frac{\inner{\mathfrak{s}, \mathfrak{a}^\star}}{\inner{\mathfrak{s}, \p}\inner{\mathfrak{a}^\star, \p}}\equiv const.$ for all circles~$s$ in the linear circle complex determined by~$a$. 

We obtain three types of linear complexes:
\begin{itemize}
\item if $\inner{ \mathfrak{a}, \mathfrak{a}} = 0$, then the 2-dimensional family consists of all elements in $\Light$ that are in oriented contact with the circle represented by~$a^\star=a$;
\item if $\inner{\mathfrak{a}, \mathfrak{a}} > 0$, then the vector $a^\star \in \Light$ represents a circle. Due to (\ref{equ_intersection_angle}), all circles in the linear circle complex intersect this circle $a^\star$ at the constant angle~$\cos \varphi = 1 + \chi$; 
\item if $\inner{\mathfrak{a}, \mathfrak{a}} < 0$, there are two cases: either $a^\star$ defines a real-valued vector in $\Light$, then the circles in the linear complex intersect $a^\star$ at a constant imaginary angle~$\cos \varphi = 1 + \chi$; or $a^\star$ is a circle with imaginary radius, where the geometric interpretation is obscure. The latter case includes for example linear circle complexes consisting of all circles with the same constant radius.
\end{itemize}
We call the (possibly complex) light-like vector $a^\star$ the \emph{directrix of the linear circle complex}.
%
%
\begin{fact}\label{fact_const_complex}
Let $\{c_i \}_{i \in I}$ be a family of circles in $\Light$. If there exists a choice of homogeneous coordinates~$\cc_i \in c_i$ and two vectors $\m, \tilde{\m} \in \mathbb{R}^{3,2}$ such that
\begin{equation*}
\inner{\cc_i, \m} =: \xi \equiv const, \ \ \text{and } \ \ \inner{\cc_i, \tilde{\m}} =: \tilde{\xi} \equiv const
\end{equation*}
for all $i \in I$, then those circles lie in the fixed linear complex determined by $\frac{1}{\xi} \m - \frac{1}{\tilde{\xi}} \tilde{\m}$.
\end{fact}
%
%
%
\noindent Any linear complex~$a \in \mathbb{P}(\mathbb{R}^{3,2})$ with $\inner{\mathfrak{a}, \mathfrak{a}} \neq 0$ gives rise to a reflection~$\sigma_a$ via
\begin{equation}\label{equ_formula_inversion}
 \sigma_a:\mathbb{R}^{3,2} \rightarrow \mathbb{R}^{3,2}, \ \ 
 \mathfrak{s} \mapsto \sigma_a(\mathfrak{s}) :
  = \mathfrak{s}-\frac{2\inner{\mathfrak{s}, \mathfrak{a}}}
   {\inner{\mathfrak{a}, \mathfrak{a}}}\mathfrak{a}.
\end{equation}  
Those maps, also called \emph{inversions}, are involutions that map the projective light cone onto itself and preserve oriented contact between circles. However, note that if $\inner{\mathfrak{a}, \p} \neq 0$, an inversion possibly maps points to circles and vice versa. Inversions that preserve the subspace of points are called \emph{M-inversions} and are characterized by the condition~$\inner{\mathfrak{a}, \p}=0$. 

Any M\"obius transformation of $\mathbb{R}^2$ can be composed as a sequence of M-inversions, while the more general group of Lie transformations is obtained by composing all types of inversions. Thus, the orthogonal group~$O(3,2)$ is a double cover of the group of Lie transformations.
%
%
\begin{fact}\label{fact_fixedpoint}
An element~$\s \in \mathbb{R}^{3,2}$ is a fixed point of the inversion $\sigma_a$ if and only if $\s \perp \mathfrak{a}$.
\end{fact}
%
%
\begin{fact}\label{fact_inversion_complex}
Suppose that a Lie inversion $\sigma_z$ is given by $\mathfrak{z}=\mathfrak{a}-\mathfrak{b}$. This Lie inversion maps $\spann{\mathfrak{a}}$ onto $\spann{\mathfrak{b}}$ if and only if $\inner{\mathfrak{a},\mathfrak{a}}=\inner{\mathfrak{b},\mathfrak{b}}$.

In particular, if $a, b \in \Light$ represent two circles, then the Lie inversion with respect to $\mathfrak{z}=\mathfrak{a}-\mathfrak{b}$ maps these two circles onto each other.
\end{fact}
%
%
Space forms in this model are recovered by the choice of \emph{space form vectors $\mathfrak{q} \in \mathbb{R}^{3,2}$} with $\inner{\p, \q}=0$. Geometric configurations composed of circles in~$\mathbb{R}^2$ can then be interpreted in a space form of constant sectional~$\kappa_\mathcal{Q}=-\inner{\q, \q}$ as follows 
\begin{itemize}
\item $\inner{\q, \q}=0$: if $\q \in \spann{\q_0}$ represents infinity, then the configuration is directly interpreted in  \textbf{Euclidean space}; else, if $\q$ represents a point, then we obtain an actual Euclidean space form after applying a M\"obius transformation that maps $\q$ to $\infty$. 
\\[-6pt]\item $\inner{\q, \q}>0$: configurations are interpreted in the Poincare disc or half-plane modelling \textbf{hyperbolic space}. The space form vector $\q$ determines the respective (unoriented) boundary of the model via $\mathfrak{q}\pm\sqrt{\inner{\q,\q}}\p$;
\\[-6pt]\item $\inner{\q, \q}<0$: planar configurations in this case are related to \textbf{spherical geometries} modelled on $S^2 \subset \mathbb{R}^3$ via an appropriate stereographic projection. This projection maps $\q_\infty$ to the north pole $N$ of $S^2$ as well as the point $\q_\infty - \frac{2\inner{\q, \q_\infty}}{\inner{\q,\q}}\q$ to the antipodal point of $N$.
\end{itemize} 
We denote the space form determined by the space form vector $\q$ by $\mathcal{Q}$. 

\bigskipp All \emph{space form motions (isometries) of $\Q$} are then obtained as compositions of any number of M-inversions that have the space form vector~$\q$ as fixed point.

\bigskipp The \emph{geodesic curvature $\kappa$ of a circle $s \in \Light$} with respect to the space form $\mathcal{Q}$ is given by 
\begin{equation}\label{equ_formula_curvature}
\kappa=\frac{\inner{\mathfrak{s}, \mathfrak{q}}}{\inner{\mathfrak{s}, \mathfrak{p}}}.
\end{equation}
Geodesics in space forms, i.\,e.\,curves with vanishing geodesic curvature, are circles. By (\ref{equ_formula_curvature}), it follows that a circle $s \in \Light$ represents a geodesic in the space form~$\mathcal{Q}$ if and only if $\inner{\mathfrak{s}, \q}=0$. 
%
%

\bigskipp All notions discussed above have a natural generalization to 3-space, which is modelled in the vector space~$\mathbb{R}^{4,2}$. 

In particular, for a given geometric configuration in $\mathbb{R}^2$, there exists a \emph{canonical embedding}
\begin{equation*}
\iota: \mathbb{R}^{3,2} \to \mathbb{R}^{4,2}, \ x = \{x_1,x_2,x_3,x_4,x_5 \} \mapsto \ \iota(x):=\{x_1,x_2,0,x_3,x_4,x_5  \}.
\end{equation*}
Geometrically, the map $\iota$ assigns to a vector representing a point in $\mathbb{R}^2$ a vector that represents the corresponding point in the x-y-plane in $\mathbb{R}^3$. Furthermore, any oriented circle in $\mathbb{R}^2$ is identified with an oriented sphere that contains this circle and intersects the x-y-plane orthogonally. This embedding also holds for linear complexes.
%
%
\bibliography{elasticabib}
%
%

\bigskip

\noindent\\[12pt]\begin{minipage}{14cm}
\textbf{Tim Hoffmann}, \href{mailto:tim.hoffmann@ma.tum.de}{tim.hoffmann@ma.tum.de}
\\Department of Mathematics,
Technical University of Munich, 85748 Garching, Germany
\end{minipage}
\\[12pt]\begin{minipage}{14cm}
\textbf{Gudrun Szewieczek}, \href{mailto:gudrun.szewieczek@tum.de}{gudrun.szewieczek@tum.de}
\\Department of Mathematics,
Technical University of Munich, 85748 Garching, Germany
\end{minipage}
\end{document}